\documentclass[10pt]{amsart}

\usepackage[latin1]{inputenc}
\usepackage[T1]{fontenc}
\usepackage[english]{babel}
\usepackage{amsmath}
\usepackage{amsthm}
\usepackage{amsfonts}
\usepackage{amssymb}
\usepackage{booktabs}
\usepackage{mathrsfs}
\usepackage{enumerate}
\usepackage{chngcntr}
\usepackage{microtype}
\usepackage{tabularx}
\usepackage{array}
\usepackage{caption}
\usepackage{multicol}
\usepackage{graphicx}
\usepackage{float}
\usepackage[dvipsnames,svgnames]{xcolor}
\usepackage{cool}
\usepackage{todonotes}

\graphicspath{{figures/}}

\usepackage{lmodern}

\numberwithin{equation}{section}
\numberwithin{figure}{section}

\theoremstyle{plain}
\newtheorem*{theorem*}{Theorem}
\newtheorem*{lemma*}{Lemma}
\newtheorem{lemma}{Lemma}[section]

\newtheorem{proposition}[lemma]{Proposition}
\newtheorem{prop}[lemma]{Proposition}
\newtheorem{theorem}[lemma]{Theorem}

\newtheorem{cor}[lemma]{Corollary}

\theoremstyle{definition}
\newtheorem{defn}[lemma]{Definition}

\theoremstyle{remark}

\newtheorem{remark}[lemma]{Remark}
\newtheorem{rem}[lemma]{Remark}

\newcommand{\ang}[1]{{\left\langle{#1}\right\rangle}}

\renewcommand{\Im}{\operatorname{Im}}
\renewcommand{\Re}{\operatorname{Re}}

\newcommand{\ep}{{\epsilon}}

\newcommand{\id}{\mathrm{Id}}
\newcommand{\Id}{\mathrm{Id}}
\newcommand{\RR}{\mathbb{R}}
\newcommand{\ZZ}{\mathbb{Z}}
\newcommand{\NN}{\mathbb{N}}

\newcommand{\CC}{\mathbb{C}}

\DeclareMathOperator{\WF}{WF}
\DeclareMathOperator{\Op}{Op}

\DeclareMathOperator{\SL}{SL} % Special linear group

\DeclareMathOperator{\ord}{ord}

\DeclareMathOperator{\coker}{coker} % Cokernel
\newcommand{\SD}{\mathcal D}
\newcommand{\SF}{\mathcal F}

\renewcommand{\SL}{\mathcal L}
\newcommand{\SM}{\mathcal M}
\newcommand{\SN}{\mathcal N}
\newcommand{\SQ}{\mathcal Q}
\newcommand{\SR}{\mathcal R}

\newcommand{\SW}{\mathcal W}
\newcommand{\SX}{\mathcal X}
\newcommand{\SY}{\mathcal Y}

\newcommand\mf{\mathrm{mf}}

\colorlet{linkcolour}{black}
\colorlet{urlcolour}{blue}
\usepackage[colorlinks=true]{hyperref}
\hypersetup{colorlinks=false,
	linkcolor=red,
	citecolor=green,
	urlcolor=magenta}

\newcommand{\Psip}[2]{\Psi^{#1,#2}_\mathrm{par}(\RR^{n+1})}
\newcommand{\Psipcl}[2]{\Psi^{#1,#2}_{\mathrm{par},\mathrm{cl}}(\RR^{n+1})}
\newcommand\Spar[2]{S^{#1,#2}_\mathrm{par}(\RR^{n+1})}
\newcommand{\Spcl}[2]{S^{#1,#2}_{\mathrm{par},\mathrm{cl}}(\RR^{n+1})}

\newcommand\Char{\mathop{\rm Char}}
\newcommand\ubar{\overline{u}}

\newcommand\Hpar[2]{H_{\mathrm{par}}^{#1, #2}(\RR^{n+1})}
\newcommand\Hparpm[2]{H_{\mathrm{par}, \pm}^{#1; #2}(\RR^{n+1})}

\newcommand\Hparmod[3]{H_{\modu}^{#1, #2; #3}(\RR^{n+1})}

\newcommand{\Rot}{\mathrm{Rot}}
\newcommand\Ell{\mathrm{ell}}
\newcommand\sw{\mathsf{r}}
\newcommand\Rout{P_+^{-1}}
\newcommand\Rin{R^-}
\newcommand\modu{\mathcal{D}}
\newcommand\upl{\Rout v}

\DeclareMathOperator{\supp}{supp}
\newcommand\lra{\longrightarrow}

\newcommand\Hdata{\mathcal{W}}
\newcommand\Poi{\mathcal{P}}
\newcommand\Poim{\mathcal{P}_-}
\newcommand\Poip{\mathcal{P}_+}
\newcommand\Poipm{\mathcal{P}_\pm}

\newcommand\Nhat{\widehat{\mathcal{N}}}

\newcommand\rmin{\mathbf{r}_{\mathrm{min}}}

\newcommand\rhobase{\rho_{\mathrm{base}}}
\newcommand\rhofib{\rho_{\mathrm{fib}}}

\newcommand\SymbSpa{{}^{\mathrm{par}} \overline{T}^* \overline{\mathbb{R}}^{n + 1}}

\begin{document}

\title[Scattering regularity for 
NLS]{Scattering regularity for small data solutions of the nonlinear
  Schr\"{o}dinger equation} % \ah{Alternative:
  % Microlocal/Fredholm approach to the Nonlinear Schr\"{o}dinger
  % equation I: small data solutions and scattering regularity}

\author{Jesse Gell-Redman}
\address{School of Mathematics and Statistics, University of Melbourne, Melbourne, Victoria, Australia}
\email{jgell@unimelb.edu.au}

\author{Se\'{a}n Gomes}
\address{Department of Mathematics and Statistics, University of Helsinki, Helsinki, Finland}
\email{sean.p.gomes@gmail.com}

\author{Andrew Hassell}
\address{Mathematical Sciences Institute, Australian National University, Acton, ACT, Australia}
\email{Andrew.Hassell@anu.edu.au}

\begin{abstract}Using the Fredholm theory of the linear time-dependent
  Schrodinger equation set up in \cite{TDSL}, we solve the final-state
  problem for the nonlinear Schr\"odinger problem
$$
(D_t + \Delta + V) u = N[u], \quad u(z,t) \sim (4\pi it)^{-n/2} e^{i|z|^2/4t} f\big( \frac{z}{2t} \big), \quad t \to -\infty,
$$
where $u : \RR^{n+1} \to \CC$ is the unknown and $f : \RR^n \to \CC$
is the asymptotic data.  Here $D_t = -i \frac{\partial}{\partial t}$
and $\Delta = \sum_{j=1}^n D_{z_j} D_{z_j}$ is the positive Laplacian,
or more generally a compactly supported, nontrapping perturbation of
this, $V$ is a smooth compactly supported potential function, and the
nonlinear term $N$ is a (suitable) polynomial in $u,\ubar,\partial_{z_j} u$ and $\partial_{z_j}\ubar$ satisfying phase invariance. Our
assumption on the asymptotic data $f$ is that it is small in a certain
function space $\SW^k(\RR^n)$ constructed in \cite{TDSL}, for sufficiently
large $k \in \mathbb{N}$, where the index $k$ measures both regularity
and decay at infinity (it is similar to, but not quite a standard
weighted Sobolev space $H^{k, k}(\RR^n)$). We find that for $N[u] = \pm |u|^{p-1} u$, $p$ odd, and $(n,p) \neq (1, 3)$ then  
 if the asymptotic data as $t \to -\infty$ is small in $\SW^k$, then the
asymptotic data as $t \to +\infty$ is also in $\SW^k$; that is, the
nonlinear scattering map preserves these spaces of asymptotic data. For a more general nonlinearity involving derivatives of $u$, we show that if 
the asymptotic data as $t \to -\infty$ is small in $\ang{\zeta}^{-1} \SW^k_\zeta$, then the
asymptotic data as $t \to +\infty$ is also in $\ang{\zeta}^{-1} \SW^k_\zeta$ where $\zeta = z/2t$ is the argument of $f$. 
\end{abstract}

\thanks{This research was supported in part by the Australian Research
  Council through grants DP180100589 (JGR, SG, AH),  DP210103242 (JGR) and FL220100072 (AH). We also acknowledge the support of MATRIX through the program ``Hyperbolic differential equations in Geometry and Physics'', April 4-8, 2022}

\maketitle

\tableofcontents

%%%%%%%%%%%%%%%%%%%%%%%%%%%%%%%%%%%%%%%%%%%%%%%%%%%%%%%%%%%%%%%%%%%
%%%%%%%%%%%%%%%%%%%%%%%%%%%%%%%%%%%%%%%%%%%%%%%%%%%%%%%%%%%%%%%%%%%
%%%%%%%%%%%%%%%%%%%%%%%%%%%%%%%%%%%%%%%%%%%%%%%%%%%%%%%%%%%%%%%%%%%
%%%%%%%%%%%%%%%%%%%%%%%%%%%%%%%%%%%%%%%%%%%%%%%%%%%%%%%%%%%%%%%%%%%
%%%%%%%%%%%%%%%%%%%%%%%%%%%%%%%%%%%%%%%%%%%%%%%%%%%%%%%%%%%%%%%%%%%

\section{Introduction}
\label{sec:intro}
This article is concerned with the `final state' problem for the nonlinear time-dependent Schr\"{o}dinger equation
\begin{equation}\label{eq:schrod}
Pu=(D_t + \Delta_{g(t)} + V)u=N[u]
\end{equation}
on $\RR^{n+1}$. Here we are given a family of metrics $g(t)$ on $\RR^n$ which are equal to the standard metric for $|t| \geq T$ and also equal to the standard metric for all $|z| \geq R$, that is, it is a compactly supported (in spacetime) perturbation of the flat metric on $\RR^n$. We assume that $g(t)$ is smooth in $t$ and is nontrapping for each $t$. We use the notation $D_t = -i \partial_t$,  $D_{z_j} = -i \partial_{z_j}$, etc, and always use $z = (z_1, \dots, z_n)$ for Cartesian coordinates in $\RR^n$. 
The positive Laplacian on $\RR^n$ with respect to $g(t)$ is denoted
$\Delta_{g(t)}$. The potential $V(z, t)$ is assumed to be smooth and
compactly supported in spacetime, although this condition is mostly
for convenience and could easily be weakened. % \ah{Should we include a
  % result in this paper where $V$ is taken to be in a module regularity
  % space? It would be a natural application of our multiplication
  % results.}
Finally, the nonlinear term $N[u]$ is given by a
polynomial in $u$, $\ubar$ and their first-order spatial derivatives,
satisfying phase invariance:
\begin{equation}
N[e^{i\theta} u] = e^{i\theta} N[u], \quad \forall \theta \in \RR.
\label{eq:phaseinv}\end{equation}
Typical examples are 
\begin{equation}
N[u] = |u|^{p-1} u, \quad p \text{ odd}, \quad N[u] = 2 |\nabla u|^2 u \quad \text{ or } \quad N[u] = \partial_{z_i} (|u|^2 u).
\label{eq:N}\end{equation}
In every case we define $p$ to be the \emph{minimal} degree of each monomial appearing in the polynomial $N$; in view of \eqref{eq:phaseinv} it is necessarily an odd integer, indeed every monomial appearing in $N$ must have odd degree. We assume $p \geq 3$, so that $N$ is genuinely nonlinear. 
%for some odd integer $p \geq 3$. Thus, with $p = 2q + 1$, this term is $u^{q+1} \ubar^q$, i.e. it is a polynomial in $u$ and $\ubar$. It will also be important for us that $N$ satisfies phase invariance, i.e.\ $N[e^{i\theta} u] = e^{i\theta} N[u]$. 

The final state problem is to show the existence and uniqueness of a solution to \eqref{eq:schrod} with the asymptotic 
\begin{equation}
u(z,t) \sim (4\pi it)^{-n/2} e^{i|z|^2/4t} f\big( \frac{z}{2t} \big), \quad t \to -\infty,
\label{eq:fminus}\end{equation}
for prescribed asymptotic data $f$. A similar asymptotic as $t \to +\infty$, 
\begin{equation}
u(z,t) \sim (4 \pi it)^{-n/2} e^{i|z|^2/4t} \tilde f\big( \frac{z}{2t} \big), \quad t \to +\infty,
\label{eq:fplus}\end{equation}
could alternatively be prescribed, but we will arbitrarily choose to prescribe asymptotic data as $t \to -\infty$. 

\subsection{Previous results from \cite{TDSL}}
This paper is a sequel to \cite{TDSL}, where the linear Schr\"odinger equation, that is, with $N \equiv 0$, was analyzed. Initally, the inhomogeneous equation, 
\begin{equation}
Pu(z,t) = F(z,t)
\label{eq:inhom}\end{equation}
was studied, where $F$ is a given function on $\RR^{n+1}$. There the authors constructed pairs of Hilbert spaces $\SX_\pm^{\bullet}(\RR^{n+1})$ and $\SY_\pm^{\bullet}(\RR^{n+1})$ (where $\bullet$ indicates a collection of exponents) such that
\begin{equation}
P : \SX_\pm^{\bullet}(\RR^{n+1}) \to \SY_\pm^{\bullet}(\RR^{n+1}) \text{ is invertible,}
\label{eq:inv}\end{equation}
that is, an isomorphism of Hilbert spaces. In fact, this was done in \emph{two} different settings: variable order (parabolic) Sobolev spaces, and fixed order Sobolev spaces with module regularity. In the present paper, we have modified the function spaces so that they have better multiplicative properties, but we begin by describing the function spaces of \cite{TDSL} as our new function spaces have some of the character of the previous spaces in \emph{both} these settings. 

All the $\SX_\pm^{\bullet}(\RR^{n+1})$ and $\SY_\pm^{\bullet}(\RR^{n+1})$ 
spaces are based on a parabolic pseudodifferential calculus $\Psip sr$ introduced in \cite{TDSL} (it is a global version of a calculus introduced by Lascar \cite{lascar}, with `scattering' type behaviour at spacetime infinity). Here $s$ is the parabolic regularity order and $r$ is the spacetime weight or spacetime decay order, so that $\Psip  sr = (1 + |z|^2 + t^2)^{r/2} \Psip s0 = \Psip s0 (1 + |z|^2 + t^2)^{r/2} $. This calculus gives rise to spaces of weighted (parabolic) Sobolev spaces, $\Hpar sr$, which are precisely the functions that are mapped to $L^2(\RR^{n+1})$ by all elements of $\Psip sr$.  We also defined spaces where the weight $r$ varies microlocally, that is, the weight is a function on phase space (it should be a classical symbol of order zero on the compactified phase space; see Section 2). Correspondingly, we have Sobolev spaces with variable order, $\Hpar s {\mathrm{r}}$. \emph{We will always write a variable spatial order in a bold font}, such as $\mathrm{r}$. 

The first setting in which the authors obtained pairs of spaces satisfying \eqref{eq:inv} is where 
\begin{equation} \SY_\pm^{\bullet}(\RR^{n+1}) = \Hpar {s-1}{\mathrm{r}_\pm+1} 
\label{eq:var-weight}\end{equation}
and $\SX_\pm^{\bullet}(\RR^{n+1})$ is a domain space 
\begin{equation}
\SX^{s, \mathrm{r}_\pm} = \big\{ u \in \Hpar s{\mathrm{r}_\pm} \mid Pu \in \Hpar {s-1}{\mathrm{r}_\pm+1} \big\}.
\label{eq:var-weight-invertibility}\end{equation}
Here the variable order $\mathrm{r}_\pm$ must have good properties with respect to the bicharacteristic flow of $P$ over spacetime infinity (i.e. the Hamiltonian flow of the symbol $p$ of $P$ inside the characteristic set of $P$, $\Char(P)$). 
Recall from \cite{TDSL} (and from Section \ref{sec:tech} below in \eqref{eq:radialsetspole}) in the phase space over spacetime infinity, there are two distinguished submanifolds, that we call the incoming and outgoing radial sets $\SR_-, \SR_+$, which are a source and sink, respectively, for the bicharacteristic flow. Every bicharacteristic starts at $\SR_-$ and ends at $\SR_+$. The variable order $\mathrm{r}_+$ must have the property that it is less than the threshold value of $-1/2$ at $\SR_+$, greater than $-1/2$ at $\SR_-$ and monotone nonincreasing under the bicharacteristic flow inside $\Char(P)$. Similarly, variable order $\mathrm{r}_-$ must have the property that it is less than the threshold value of $-1/2$ at $\SR_-$, greater than $-1/2$ at $\SR_+$ and monotone nondecreasing under the bicharacteristic flow inside $\Char(P)$ (we could take $\mathrm{r}_- = -1 - \mathrm{r}_+$). The threshold value is the smallest weight $r$ such that all global solutions to $Pu= 0$ fail to lie in $(1 + |z|^2 + t^2)^{-r} L^2(\RR^{n+1})$, and is easily seen to be $-1/2$ via \eqref{eq:fminus} and \eqref{eq:fplus}, for example. These conditions on the variable order $\mathrm{r}_\pm$ are essential in order to assemble various microlocal propagation estimates to obtain a global Fredholm estimate, from which the invertibility of $P$ acting between \eqref{eq:var-weight-invertibility} and \eqref{eq:var-weight} can be obtained. 

The second setting requires the notion of module regularity. In \cite{TDSL}, a module (shorthand for a test module of pseudodifferential operators) was defined, following and slightly generalizing  \cite{HMV2004}, as a subspace of $\Psip 11$ that contains and is a module over $\Psip 00$, and is closed under commutators. There were three modules of interest: the modules $\SM_\pm$ of elements of $\Psip 11$ characteristic on $\SR_\pm$, and their intersection $\SN := \SM_+ \cap \SM_-$. Following \cite{HMV2004},  spaces of functions with a finite order of module regularity were defined: 
\begin{multline}
\Hparpm {s,r}{j, k} = \{ u \in \Hpar s r  \mid A_1 \dots A_j B_1 \dots B_k u \in \Hpar s r , \\
A_1, \dots A_j \in \SM_\pm, \ B_1, \dots, B_k \in \SN \}.
\end{multline}
%
%\SY_\pm^{s, r; j, k}(\RR^{n+1}) = \{ u \in \SY_\pm^{s, r}(\RR^{n+1}) \mid A_1 \dots A_j B_1 \dots B_k u \in \SY_\pm^{s, r}(\RR^{n+1}), \ A_1, \dots A_j \in \SM_\m, \ B_1, \dots, B_k \in \SN \}.
%\end{equation}
In words, the elements of this space are those elements of $\Hpar sr$ that remain in this space after the application of $j$ elements of $\SM_\pm$ and $k$ elements of $\SN$. We define the $\SX$- and $\SY$-spaces by 
\begin{align}
\SY_\pm^{s, r; j, k}(\RR^{n+1}) &:= \Hparpm {s}{r; j, k} , \\
\SX_\pm^{s,r; j, k}(\RR^{n+1}) &:= \{ u \in \Hparpm  {s}{r; j, k}\mid Pu  \in \Hparpm {s-1}{r+1; j, k} \}.
\end{align}

The authors then proved \eqref{eq:inv} for $\SY_\pm^{\bullet}(\RR^{n+1}) = \SY_\pm^{s-1, r+1; j, k}(\RR^{n+1})$ and $\SX_\pm^{\bullet}(\RR^{n+1}) = \SX_\pm^{s, r; j, k}(\RR^{n+1})$ when $s \in \mathbb{R}$, $r \in (-3/2, -1/2)$, $j \geq 1$ and $k \geq 0$. 

Note that in either setting, the $\SX_+$-function space is designed so that elements of this space have spacetime decay above threshold at $\SR_-$ but below threshold at $\SR_+$. This is clear in the case of the variable order spaces, while for the module regularity spaces, we note that there are elements of $\SM_+$ that are elliptic of order $(1,1)$ at $\SR_-$. Since we stipulated $j \geq 1$, we have (at least) one order of $\SM_+$-regularity and this raises the spacetime decay by $1$ at $\SR_-$. Since we assumed that $r > -3/2$, this lifts the decay to above threshold at this radial set. 
This microlocal regularity\footnote{The term `regularity' is a synonym for spacetime decay at the spacetime boundary of phase space.} at $\SR_+$ vs. $\SR_-$ means, for elements of $\SX_+$,  that nontrivial expansions such as \eqref{eq:fplus} are allowed, while \eqref{eq:fminus} is prohibited. For $\SX_-$ this is reversed: nontrivial expansions such as \eqref{eq:fplus} are prohibited, while \eqref{eq:fminus} is allowed.

The next result of \cite{TDSL} concerned global solutions of $Pu = 0$ with prescribed final state data $f$ in the sense described above. In \cite{TDSL} the authors defined `Poisson operators' $\Poim$ ($\Poip$) that map $f$ ($\tilde f$) to the solution $u$ of $Pu = 0$ satisfying \eqref{eq:fminus} (\eqref{eq:fplus}). Spaces of final state data $f$, denoted $\Hdata^k(\RR^n)$, were also defined, which are module regularity spaces (over $\RR^n$) for a certain module $\Nhat$; see Section \ref{sec:new linear
stuff}. These spaces increase in both regularity and decay as $k$ increases, and the union over $k \in \ZZ$ is the space of tempered distributions, while the intersection is the Schwartz class. The authors showed that for $k \in \NN$, the range of the Poisson operator $\Poim$ on $\Hdata^k(\RR^n)$ is precisely 
\begin{equation}
\{ u \in \mathcal{X}_+^{1/2, \sw_+; k,0}(\RR^{n+1}) + \mathcal{X}_-^{1/2, \sw_-; k,0}(\RR^{n+1}) \mid Pu = 0 \}, 
\label{eq:Poissonrange}\end{equation}
i.e. that is, those elements of $\mathcal{X}^{1/2, \sw_+} + \mathcal{X}^{1/2, \sw_-}$ in the kernel of $P$ having $\SM_\pm$-module regularity of order $k$. It was shown that for  $k \geq 2$, the global solution $u = \Poim f$ admits the asymptotic \eqref{eq:fminus} in the precise sense that 
\begin{equation}
 \lim_{t \to -\infty} (4\pi it)^{n/2} e^{-it|\zeta|^2} u(2t \zeta, t) = f(\zeta)
\end{equation} 
as a limit in the space $\ang{\zeta}^{1/2 + \epsilon}\Hdata^{k-2}(\RR^n_\zeta)$. (Notice that this limit is the same as 
$$
 \lim_{t \to +\infty, z/t \to 2\zeta} (4\pi it)^{n/2} e^{-i|z|^2/4t} u_+(z, t) . \ )
 $$
The limit as $t \to +\infty$ exists in the same sense. Precisely, the limit 
\begin{equation}
\tilde f :=  \lim_{t \to -\infty} (4\pi it)^{n/2} e^{-it|\zeta|^2} u(2t \zeta, t) 
\end{equation} 
exists in the same space $\ang{\zeta}^{1/2 + \epsilon}\Hdata^{k-2}(\RR^n_\zeta)$, and this limit $\tilde f$ is also in $\Hdata^k(\RR^n_\zeta)$.
The existence of this limit allows us to define the (linear) scattering map, $S: f \to \tilde f$. It is a linear isomorphism $S : \Hdata^k \to \Hdata^k$.

\subsection{New results} An absolutely crucial point distinguishing the results of \cite{TDSL}
from those of the present paper is that the modules $\mathcal{M}_\pm$
and $\mathcal{N}$ necessarily have \emph{pseudodifferential} (as
opposed to differential)
generators.  This is true simply because a full generating set for the
modules must have $\tau$ dependence and the operator $D_t$ is a
parabolic differential operator with parabolic regularity order \emph{two}, in
particular $D_t \in \Psi_{\mathrm{par}}^{2,0}$.

In the present paper, we find that there is a
natural module $\modu$ generated by \emph{differential} operators whose
corresponding module regularity spaces are particularly well suited to
analysis of the nonlinear equation.  This is a module with explicit
differential operator generators which arise naturally from the invariance
properties of the Schr\"odinger equation, namely the linear
derivatives, the generators of rotations, the generators of Galilean
tranformations, and the generator of the natural dilations $2tD_t + z
\cdot D_z$, see \eqref{eq:G.def} below.  In particular, these are
operators that commute with $P_0$, except for $2t D_t + z \cdot D_z$
which satisfies
$$
i[2t D_t + z \cdot D_z, P_0] = 2 P_0.
$$
The module regularity spaces corresponding to $\modu$ satisfy algebra (multiplication)
properties directly analogous to those of \cite{NEH}, meaning they
produce a sharp gain in spacetime decay of order $(n + 1)/2$, allowing for treatment of
nonlinear terms. When proving multiplication properties, having differential generators --- which means that we have the Leibniz rule for applying generators to products --- is a great advantage of working with $\modu$ as opposed to $\SM_\pm$ and $\SN$. 

Taking these algebra properties for granted for the
moment, a trade-off occurs in comparing analysis of the modules of
\cite{TDSL} with that of $\modu$, namely, the module $\modu$ fails to
have some of the basic structural features one encounters in module
regularity in other settings: some of its generators are of
parabolic regularity order $2$ and \emph{it is not
closed under commutators}.  It is however closed under taking
commutators of arbitrary elements in $\modu$ with \emph{generators} of $\modu$ (see \eqref{eq:module properties}) and this
turns out to be enough to prove the relevant propagation of module
regularity estimates in the appendix.

Notice that $\modu$-module regularity cannot distinguish between the two radial sets, since $\modu$ is characteristic at both. The condition that elements of our function space $\SX_+^\bullet$, say, are below threshold at $\SR_+$ and above threshold at $\SR_-$ will be enforced by choosing a variable order $\mathrm{r}_+$ as above. We will \emph{usually} assume in addition that $\mathrm{r}_\pm$ take values in $[-1/2 - \epsilon, -1/2 + \epsilon]$ for some small $\epsilon > 0$. Thus, with the third exponent $k$ denoting the order of $\modu$-module regularity, our function spaces will be 
\begin{align}
\SX^{s, \mathrm{r}_\pm ; k}_{\modu}(\RR^{n+1}) &:= \{ u \in \Hparmod s  {\mathrm{r}_\pm} k \mid Pu \in \Hparmod {s-1} {\mathrm{r}_\pm+1} {k} \}, \label{eq:SXnew} \\
\SY^{s-1, \mathrm{r}_\pm+1; k}_{\modu} (\RR^{n+1}) &:= \Hparmod {s-1}{\mathrm{r}_\pm+1}{k}. \label{eq:SYnew}
\end{align}

We have a corresponding linear theory for these spaces:

\begin{theorem}\label{thm:linear1} Assume that the variable order $\mathrm{r}_+$ ($\mathrm{r}_-$) has the property that it is less than the threshold value of $-1/2$ at $\SR_+$ ($\SR_-$), greater than $-1/2$ at $\SR_-$ ($\SR_+$) and monotone nonincreasing (nondecreasing) under the bicharacteristic flow inside $\Char(P)$. 
Then for any $k \in \NN$, the operator $P$ is a bounded invertible operator between \eqref{eq:SXnew} and \eqref{eq:SYnew}.  Thus we have bounded operators, which we denote $P_+^{-1}$ and $P_-^{-1}$:
\begin{equation}\begin{aligned}
P_+^{-1} : \SY^{s-1, \mathrm{r}_+ +1; k}_{\modu} (\RR^{n+1}) &\to \SX^{s, \mathrm{r}_+ ; k}_{\modu}(\RR^{n+1}), \\
P_-^{-1} : \SY^{s-1, \mathrm{r}_- +1; k}_{\modu} (\RR^{n+1}) &\to \SX^{s, \mathrm{r}_- ; k}_{\modu}(\RR^{n+1}), 
\end{aligned}\end{equation}
and call the outgoing $(+)$ and incoming $(-)$ propagators, respectively.
\end{theorem} 

\begin{theorem}\label{thm:linear2}
If $f \in \ang{\zeta}^{-\ell}\Hdata^k(\RR^n)$ with $\ell \in \NN$, then there is a unique solution $u$ to the equation $Pu = 0$ with final state $f$ as $t \to -\infty$. Moreover, $u$ lies in the space 
\begin{equation}
\Hparmod {1/2+\ell} {\mathrm{r}_-} k + \Hparmod {1/2+\ell} {\mathrm{r}_+} k.
\label{eq:u-modreg}\end{equation}
Conversely, if $Pu=0$ and $u$ lies in this sum of spaces, then the final state data both as $t \to -\infty$ and $t \to +\infty$ lies in $\ang{\zeta}^{-\ell}\Hdata^k(\RR^n)$. As a corollary, if the final state data as $t \to -\infty$ is in $\ang{\zeta}^{-\ell}\Hdata^k(\RR^n)$, then the final state data as $t \to +\infty$ is also in this space (that is, the scattering map preserves $\ang{\zeta}^{-\ell}\Hdata^k(\RR^n)$ regularity). 
\end{theorem}

The proof of Theorem~\ref{thm:linear1} is a minor modification of the corresponding proof in \cite{TDSL}, where a slightly different notion of module regularity was used, and is given in Appendix~\ref{appendix}. Theorem~\ref{thm:linear2} is likewise a minor modification of the the corresponding proof in \cite{TDSL}, and is given in Section~\ref{sec:new linear stuff}. 

To develop a nonlinear theory we then have a multiplication result for module
regularity spaces defined by $\modu$.  What we see is that, after
decomposing distributions $u$ into $u_1 + u_2$, where $u_1$ is
supported in $ t > c_1$ and $u_2$ is supported in
$t < c_2$ for some $c_1, c_2 \in \mathbb{R}$, we can factor out the relevant oscillatory factor from
$u_1$ and $u_2$ and commute the \emph{differential} generators of
$\modu$ past the oscillation to obtain an alternate, simpler
characterization of $\modu$ module regularity, with respect to which
multiplication results are easily deduced. The multiplication result can be
phrased as follows, with a more general version given in Proposition
\ref{thm:s-mult}.

\begin{proposition}\label{prop:mult-intro} Let $p$ be an odd integer. 
Let $k \in \mathbb{N}$ with $k >\max(n/2+2,3) $.  Then, if $N[u] = |u|^{p-1} u$, we have 
\begin{equation}\begin{gathered}
u \in \Hparmod {1/2} {-1/2 \pm \epsilon} k \Longrightarrow N[u] \in \Hparmod {1/2} {-1/2 + (p-1)n/2 \pm p \epsilon} k, \\
\| N[u] \|_{\Hparmod {1/2} {-1/2 + (p-1)n/2 \pm p \epsilon} k} \leq C \| u \|_{\Hparmod {1/2} {-1/2 \pm \epsilon} k}^p. 
\end{gathered}\label{eq:mult0}\end{equation}
% It also satisfies the difference estimate
% \begin{equation}\begin{gathered}
% \| N[u_1] - N[u_2] \|_{\Hparmod 0 {-1/2 + (p-1)n/2 \pm p \epsilon} k} \leq C \| u_1 - u_2 \|_{\Hparmod 0 {-1/2 \pm \epsilon} k} \\ \times \big( \| u_1 \|_{\Hparmod 0 {-1/2 \pm \epsilon} k} + \| u_2 \|_{\Hparmod 0 {-1/2 \pm \epsilon} k} \big)^{p-1}.
% \end{gathered}\label{eq:Ndiff}\end{equation}

If $N$ is a polynomial in $u$, $\ubar$ and the spatial derivatives $D_{z_i} u, D_{z_i}\ubar$, with all monomials of degree at least $p$, and satisfies \eqref{eq:phaseinv}, then 
\begin{equation}
u \in \Hparmod {3/2} {-1/2 \pm \epsilon} k \Longrightarrow N[u] \in \Hparmod {1/2} {-1/2 + (p-1)n/2 \pm p \epsilon} k.
\label{eq:multdiff}\end{equation}
% and we have a difference estimate
% \begin{equation}\begin{gathered}
% \| N[u_1] - N[u_2] \|_{\Hparmod {1/2} {-1/2 + (p-1)n/2 \pm p \epsilon} k} \leq C \| u_1 - u_2 \|_{\Hparmod {3/2} {-1/2 \pm \epsilon} k} \\ \times \big( \| u_1 \|_{\Hparmod {3/2} {-1/2 \pm \epsilon} k} + \| u_2 \|_{\Hparmod {3/2} {-1/2 \pm \epsilon} k} \big)^{p-1}.
% \end{gathered}\label{eq:Ndiff2}\end{equation}
\end{proposition}

Now we can combine Proposition~\ref{prop:mult-intro} and the invertibility of $P$ between \eqref{eq:SXnew} and \eqref{eq:SYnew} for suitable $\mathrm{r}_\pm$ and $k$, to solve the small data final state problem for NLS, following \cite{NEH} and \cite{TDSL}. Here we give a brief sketch of the method; the full proof is in Section~\ref{sec:results}. For simplicity we only consider the case $N[u] = \pm |u|^{p-1} u$ in this introduction. 

Given final state data $f$, we first apply Theorem~\ref{thm:linear2} to solve the linear equation $Pu_0=0$ with this final state. We break up $u_0 = u_+ + u_-$ in such a way that $u_\pm \in \SX^{1/2, \mathrm{r}_\pm; k}_{\modu}$, which is possible according to Theorem~\ref{thm:linear2}. We then seek a solution to the nonlinear equation in the form $u = u_- + w$, where $w \in \SX^{1/2, \mathrm{r}_+; k}_{\modu}$. By construction, $u$ has final state $f$. It remains to solve the equation
$$
P (u_- + w) = N[u_- + w] .
$$
We want to write this as an equation in the function space $\SX^{1/2, \mathrm{r}_+; k}_{\modu}$. To do this, we first replace $Pu_-$ by $- Pu_+$. Second, we use trivial embeddings 
\begin{equation}
\Hparmod {1/2} {-1/2 + \epsilon} k \subset \Hparmod {1/2} {\mathrm{r}_-} k \subset \Hparmod {1/2} {-1/2 - \epsilon} k
\label{eq:embeddings}\end{equation}
to approximate our variable order space by a constant order space, and apply Proposition~\ref{prop:mult-intro}. This proposition shows that the RHS is in $\Hparmod {-1/2} {\mathrm{r}_+ +1} k = \SY^{1/2, \mathrm{r}_+ +1; k}_{\modu}$ provided that 
$$
(p-1) \frac{n}{2} - p \epsilon \geq 1.
$$
Since $\epsilon$ can be taken arbitrarily small, this condition can be achieved provided that 
\begin{equation}
(p-1) \frac{n}{2} > 1,
\label{eq:pn-cond}\end{equation}
which, for $n, p \in \mathbb{N}$ and $p$ odd excludes only the case
$(n,p) = (1,3)$, i.e.\ cubic NLS in one spatial dimension.
Choose $\epsilon$ so that $0 < \epsilon < (p+1)^{-1}$, and choose $\mathrm{r}_\pm$ satisfying the conditions above, with range contained in $[-1/2 - \epsilon, -1/2 + \epsilon]$. \        
We can then apply the outgoing propagator
\begin{equation}
P_+^{-1} : \SY_{\modu}^{1/2, \mathrm{r}_+ + 1; k} \to \SX_{\modu}^{3/2,
  \mathrm{r}_+; k} \subset \SX_{\modu}^{1/2, \mathrm{r}_+;
  k}. \label{eq:intro propagator bound}
\end{equation}
We can now find the solution as the fixed point of a contraction map $\Phi$, defined on a small ball around the origin in $\SX_{\modu}^{1/2, \mathrm{r}_+; k}$, where
\begin{equation}
\Phi(w) = u_+ + P_+^{-1} \big( N[u_- + w] \big).
\label{eq:Phi}\end{equation}
A difference estimate, which is essentially a simple elaboration of
Proposition \ref{prop:mult-intro}, together with the boundedness of $P_+^{-1}$ imply that $\Phi$ is a contraction on a suitably small ball around the origin in $\SX_{\modu}^{1/2, \mathrm{r}_+; k}$. The first main theorem, proved in Section~\ref{sec:results}, is as follows.

\begin{theorem}\label{thm:main}
Assume that the nonlinear term $N$ is given by $c |u|^{p-1} u$ with
$p$ an odd integer and $(n,p) \neq (1,3)$. Choose $\epsilon$ so that $0 < \epsilon < (p+1)^{-1}$, and choose
$\mathrm{r}_\pm$ satisfying the conditions above, with range contained
in $[-1/2 - \epsilon, -1/2 + \epsilon]$. Let $k > \max(n/2+2,3)$,
and prescribe final data $f \in \SW^k(\RR^n)$ with sufficiently small
norm. Then there exists a unique
$u$ with small norm in the sum of module regularity spaces
\begin{equation}
\SX^{1/2, \mathrm{r}_-;  k}_{\modu} + \SX^{1/2, \mathrm{r}_+;  k}_{\modu}
\label{eq:u-modreg-1}\end{equation}
such that
$$
Pu = N[u], 
$$
and such that $u$ has an expansion \eqref{eq:fminus} as $t \to -\infty$, with the prescribed data $f$. Moreover, $u$ has a similar expansion \eqref{eq:fplus} as $t \to +\infty$, with asymptotic data $\tilde f$ also in the space $\SW^k$. Thus the scattering map $f \mapsto \tilde f$ preserves regularity as measured by the $\Hdata^k$-scale of spaces. 
\end{theorem}

Directly from the definition of the $\SX^{\bullet}_{\modu}$ spaces we
deduce the following:
\begin{cor}
  Assumptions as in Theorem \ref{thm:main}, with incoming data $f \in
  \SW^k(\RR^n)$ sufficiently small, the unique (sufficiently small) solution to $Pu
  = N[u]$ with incoming data $f$ satisfies $u \in H^{1/2, -1/2 -
    \epsilon ; k}_{\modu}$, or, concretely,
  $$
D_{z}^\alpha \Rot^\beta (tD_{z_j}-cz_j/2)^j D_t^p(2tD_t + z \cdot D_z)^qu \in
\langle z,t \rangle^{1/2 + \epsilon} H_{\mathrm{par}}^{1/2, 0}(\mathbb{R}^{n + 1}),
$$
for $|\alpha| + |\beta| + j + 2p + 2q \le k$ where $\Rot^\beta$ is a
a product of generators of rotations of order $|\beta|$.
\end{cor}

\begin{remark} We note that the sole excluded case being $(n, p) =
  (1,3)$ allows us to treat cubic NLS for $n \geq 2$, and quintic NLS
  for all $n \in \mathbb{N}$.  The case $(n, p) = (1,3)$ is treated by
  Lindblad and Soffer \cite{MR2199392}. It cannot be expected that our method works for $(n, p) = (1,3)$, since then the asymptotics \eqref{eq:fminus}, \eqref{eq:fplus} do not hold: there is a logarithmic divergence in these asymptotic expansions, in effect due to a resonance between the nonlinearity and the mapping properties of the propagators $P_\pm^{-1}$. We leave it as an interesting open question whether modification of our methodology here could be used to treat that case.
\end{remark}

We notice that in the propagator mapping property \eqref{eq:intro propagator bound}, we simply
`wasted' one order of parabolic regularity, replacing $3/2$ by $1/2$
after application of the propagator $P^{-1}_+$.  This gain of one
order is needed (i.e.\ not wasted) in the treatment of derivative NLS, where $N[u]$ has parabolic regularity order one lower than that of $u$; in this case, the propagator $P^{-1}_+$ acts to restore the parabolic regularity order.
% This is to make use of the fact that
% Proposition~\ref{prop:mult-intro} is most directly proved when $s=0$
% (the interpolation argument is not required then), and is harmless
% because the inverse $P_+^{-1}$ gains back one order of regularity. By
% not wasting this order, we are able to treat nonlinear terms with
% first order (spatial) derivatives, using \eqref{eq:multdiff} instead
% of \eqref{eq:mult0}. The algebra properties treating non-zero $s$ are
% done in Section \ref{sec:interpolation} using interpolatio
The following theorem follows from Proposition \ref{thm:contract
  deriv} below.
  
\begin{theorem}\label{thm:main2}
Assume that the nonlinear term $N$ is given by a sum of monomials in $u$, $\ubar$ and spatial derivatives of order $1$ applied to $u$ and $\ubar$, with each monomial of degree at least $p$. Choose $\epsilon$ and $\mathrm{r}_\pm$ as in Theorem~\ref{thm:main}. Let $k > \max(n/2+2,3)$, and prescribe final data $f \in \ang{\zeta}^{-1} \SW^k(\RR^n_\zeta)$ with sufficiently small norm. Then there exists a unique $u$ with small norm in the sum of module regularity spaces 
\begin{equation}
\SX^{3/2, \mathrm{r}_-;  k}_{\modu} + \SX^{3/2, \mathrm{r}_+;  k}_{\modu}
\label{eq:u-modreg-2}\end{equation}
 such that $Pu = N[u]$, 
and such that $u$ has an expansion \eqref{eq:fminus} as $t \to -\infty$, with the prescribed data $f$. Moreover, $u$ has a similar expansion \eqref{eq:fplus} as $t \to +\infty$, with asymptotic data $\tilde f$ also in the space $\ang{\zeta}^{-1}  \SW^k$. 
\end{theorem}

\begin{remark} In our current setup, the spaces $\ang{\zeta}^{-1} \SW^k(\RR^n_\zeta)$ form a more natural scale of spaces than the $ \SW^k(\RR^n)$ scale to analyze the derivative NLS problem. The extra order of decay of $f$ at infinity due to the $\ang{\zeta}^{-1}$  factor is responsible for the extra order of parabolic regularity in \eqref{eq:u-modreg-2} as compared to \eqref{eq:u-modreg-1}.
%
%With our current theory, in the setting of Theorem~\ref{thm:main2} we are not able to show that if the incoming data $f$ is in $\SW^k$, then the outgoing data $\tilde f$ is in $\SW^k$. However, since we have 
%$$
%\SW^{k+1} (\RR^n_\zeta) \subset \ang{\zeta}^{-1} \SW^k(\RR^n_\zeta) \subset \SW^k(\RR^n_\zeta),
%$$
%it follows from Theorem~\ref{thm:main2} that if $f \in \SW^{k+1}$, then $\tilde f \in \SW^k$. 
\end{remark}

\subsection{Relation to existing literature}  

Work on scattering for NLS goes back at least to Lin and Strauss in \cite{MR515228} on cubic NLS in $\mathbb{R}^{3 + 1}$, in which the authors establish the existence of a Banach space $\mathcal{F}$ of solutions to the linear equation for which, given a solution $u$ to NLS with appropriate initial data, there are $u_\pm \in \mathcal{F}$ such that
$$
\| u(t) - u_\pm(t) \|_{L^2} \to 0, \mbox{ as } t \to \pm \infty.
$$
The literature on global well-posedness and scattering for NLS is vast; we make no attempt to survey it,  but mention a few seminal results \cite{GV85, Bourgain98, CKSTT04, MR2415387, MR2800718, MR3090782}. The approach taken here is competely novel. We mention three main differences between the present approach and previous literature:
\begin{itemize}
\item Almost all previous literature solves the initial value problem for NLS, with data prescribed at time $t=0$. We consider the final state problem, with data $f$ prescribed at $t=-\infty$, and find the corresponding global solution $u$ to $Pu = 0$ \emph{directly}; the initial value problem is never considered.
\item The notion of scaling and critical scaling regularity plays no \emph{explicit} role in our analysis (although we work in a relatively high regularity setting, which corresponds to the subcritical regime in previous literature).
\item We work with Hilbert spaces throughout, and make no use of Strichartz estimates which are fundamental to almost all previous works.
\end{itemize}

%There, the notion of criticality is central, and deep results such as the work of Colliander-Keel-Staffilani-Takaoka-Tao prove scattering for energy critical NLS for initial data of finite energy \cite{MR2415387}.  In particular, in the energy critical context, the sign of the nonlinear term impacts the analysis.  We avoid the criticality issue altogether; not only are we treating as input to the equation only the \textit{asymptotic} data, we also work in high enough regularity spaces that criticality is essentially irrelevant to our work.  In particular, the coefficient of the nonlinear term is arbitrary.  Our work is also notable in its lack of use of Strichartz inequalities, which are the basis of most work on NLS.
%

In our work, the consideration of time-dependent potentials arises as a necessity to avoid the complications in the microlocal analysis which arise from potentials which are not compactly supported in time.  Nevertheless, time-dependent potentials are of independent interest, and arise in the work of Rodnianski and Schlag \cite{MR2038194} for rough potentials and Soffer and Wu \cite{SofferWu}, who prove a local decay for NLS with time-dependent potentials.

% Zihua Guo's recent result on the scattering map preserving regularity. Relation to quadratic-scattering wavefront set.

The methodology used to obtain the results in this paper is part of a
growing body of work in geometric microlocal analysis in which one
constructs Fredholm mappings for nonelliptic operators, and then
leverages this Fredholm theory to solve nonlinear equations.  The
linear aspects of this framework go back to Faure and Sj\"ostrand's work on Anosov flows \cite{FS2011} and, especially,  Vasy's work on
asymptotically hyperbolic and Kerr-de Sitter spaces \cite{VD2013}, which itself
draws from Melrose's work on scattering theory \cite{RBMSpec} --- see for example Vasy's lecture notes \cite{grenoble}. 
This nonelliptic Fredholm theory has been generalized and applied
widely; in particular, in \cite{HVsemi}, \cite{BVW2015} and \cite{GHV2016} it was combined with the
module regularity developed in \cite{HMV2004} to solve
nonlinear wave equations.  The first and third authors, in
collaboration with Shapiro and Zhang, applied this perspective to the
nonlinear Helmholtz equation to study the nonlinear scattering matrix
of Helmholtz operators with decaying potentials \cite{NEH,NLSM}, which
work indicated the possibility of defining and studying nonlinear
scattering for evolution equations in a similar fashion.  In this
direction, the authors developed the (nonelliptic) Fredholm theory of
the linear Schr\"odinger operator in \cite{TDSL}, setting the stage for the
nonlinear analysis of the current paper. The results of \cite{TDSL} form the foundation for the present article. 
Very recently, Sussman \cite{Sussman} has completed an analysis of the Klein-Gordon equation using very similar methods; in particular he also uses an approach based on a modification of the scattering calculus, microlocal propagation estimates within such a calculus, and module regularity. 

Regularity with respect to a test module of operators
was formally introduced in \cite{HMV2004} although the concept arose
earlier. It is very closely related to Klainerman's vector field
method \cite{klainerman1985}. It is noted in volume 4 of H\"ormander's
treatise \cite{Ho4} (and attributed to Melrose) that Lagrangian
regularity is equivalent to iterative regularity with respect to the
Lie algebra of first order pseudodifferential operators characteristic
to the Lagrangian. This Lie algebra is a module over the zeroth order
operators. In our setting, we similarly look at a module $\modu$ over
$\Psip 00$ of operators characteristic at the radial set. As explained
in \cite{TDSL}, the radial set $\SR$ can be viewed as the boundary, at
spacetime infinity, of a conic (with respect to spacetime dilations)
Lagrangian submanifold, and the corresponding module regularity spaces
are analogues of functions with (a finite amount of) Lagrangian
regularity. In fact, in this sense, the `classical' Lagrangian
expansions with respect to $\SR$ are precisely expansions of the form
\eqref{eq:fminus}, \eqref{eq:fplus}.  Infinite module regularity with
respect to $\modu$ is equivalent to infinite module regularity with
respect to the module $\SN$ from \cite{TDSL} defined by the characteristic
condition, so in that sense Lagrangian functions associated to $\SR$ do not distinguish
between the two types of modules.  

Our work is focused on understanding regularity properties of the scattering map, more precisely, showing that the outgoing data $\tilde f$ has the same regularity as the incoming data $f$  in the $\Hdata^k$-scale of spaces. Previously, it has been shown that, for example for the initial value problem for cubic NLS in 3 spatial dimensions, there is scattering for initial data $u_0$ in $H^s$ for certain ranges of $s$, and the corresponding free data $\phi_\pm$ is also in $H^s$ \cite{GV85, Bourgain98, CKSTT04}. Since the incoming and outgoing data are the Fourier transforms of the $\phi_\pm$, this would imply that if $f$ is in a weighted $L^2$ space $\ang{\zeta}^{-s} L^2(\RR^n_\zeta)$, for $s \geq s_0$, then $\tilde f$ is also in this space. More recently, Guo, Huang and Song \cite{GuoHuaSon} have shown that if $u_0$ is in $L^1 \cap H^3$, then $\phi_\pm$ are also in this space. This implies some regularity for both $f, \tilde f$. Beyond these results, to the authors' knowledge this question has not been much studied in the literature.

As mentioned above, for nonlinearities including terms $|u|^{p - 1}
u$, our results do not cover the case $n = 1, p = 3$.  That case is treated in
the paper of Lindblad and Soffer \cite{MR2199392}, and the form of the
solutions they obtain can be seen to incur a log loss under
application of the module derivatives of this paper.  We have not addressed the question of
whether our methods could be suitably
modified to treat that case.

% There the authors solve the initial data problem for the equation
% $$
% i \partial_t v + \partial_z^2 v - \beta |v|^2 v - \gamma |v|^4 v = 0,
% $$
% for small initial data $f \in C^\infty_c(\mathbb{R}^{1 + 1})$.  Their argument includes an ansatz for the solution $v$ of the form
% \begin{equation}
% v(t,z) = t^{-1/2} e^{i z^2 / (4t)} V(t,z/t)\label{eq:soffer ansatz}
% \end{equation}
% where
% $$
% V(s,y) \sim a(y) e^{i \phi(s,y)}, \quad \phi(s,y) = - \beta a(y)^2 \ln |s| + b(y).
% $$
% We show in Remark \ref{rem:soffer} that these solutions fail to lie in the module regularity spaces in which our solutions lie. Specifically, the module derivatives which generate Galilean transformations produce additional blow up when they hit the $e^{i \phi(s,y)}$ oscillations.  Thus, not only do our results and those of \cite{MR2199392} complement one another, but their results make clear why our methods fail in their case.

\subsection{Structure of this article}
 	
We begin, in Section \ref{sec:tech}, by recalling the essential features of the parabolic pseudodifferential calculus, developed in \cite{TDSL}, which will be relevant in our study of NLS.  This includes especially analysis of $P$ as a parabolic differential operator.  We explain the features of the parabolic principal symbol of $P$, specifically how its characteristic set extends to the boundary of an appropriately compactified parablic cotangent bundle $\SymbSpa$, and recall that its Hamilton vector field vanishes on a set $\mathcal{R}$ consisting of two connected components which intersect the corner of $\SymbSpa$ transversally.  These components $\mathcal{R}_\pm$ are (by definition) radial points of $P$, and form families of sources ($-$) and sinks ($+$) for the Hamilton flow.  The radial points estimates and Fredholm theory are then discussed, as well as multiplication results for the basic Sobolev spaces induced by the parabolic PsiDOs.

In Section \ref{sec:mod reg} we introduce the notion of module regularity. We discuss general features of module regularity spaces, and in particular consider modules generated by operators of orders potentially greater than $1$.  This is an important feature of our analysis, namely that we choose a module which is naturally related to the operator, but for which the order of one of the generators is $(2,1)$ (i.e.\ parabolic regularity order $2$, spacetime order $1$) in the parabolic sense; thus module regularity analysis from previous work does not apply directly and must be developed here.  This is discussed in Section  \ref{sec:mod reg} and the appendix.

In Section \ref{sec:mult}, we then prove the algebra property of Proposition \ref{prop:mult-intro} for module regularity spaces, with a gain of $(n+1)/2$ (easily seen to be the best possible gain) in spacetime decay order.
This is the main analytic tool used to accomodate the nonlinear term in the equation.

In Section \ref{sec:results}, we turn to proving our main theorems.  These are, roughly speaking,
theorems which produce global spacetime solutions to NLS with various
nonlinearities and prescribed asymptotic data.  This is accomplished
using a contraction mapping argument that relies on the linear
theory we develop for $P$, i.e. the mapping properties of $P_\pm^{-1}$
on module regularity spaces from Theorem~\ref{thm:linear1}, and the
algebra properties of module regularity spaces in
Sections~\ref{sec:first big algebra} and \ref{sec:interpolation}.

In the appendix we prove the module regularity estimates.  The theorems and proofs are similar to those in \cite{TDSL} (which in turn are closely related to those in \cite{HMV2004}) but it is necessary to revisit them since all earlier results are for modules generated by operators of order not greater than $1$.  Our modules in particular are not closed under commutators, and we must make do with the weaker conditions \eqref{eq:module properties}.  The modifications required to the proofs in \cite{TDSL} are nevertheless straightforward. 
%Moreover, our main Fredholm mapping result Proposition~\ref{prop:big mod reg invert} is distinguished from other Fredholm results in similar cases as it combines variable spacetime orders with module regularity.  

%%%%%%%%%%%%%%%%%%%%%%%%%%%%%%%%%%%%%%%%%%%%%%%%%%%%%%%%%%%%%%%%%%%
%%%%%%%%%%%%%%%%%%%%%%%%%%%%%%%%%%%%%%%%%%%%%%%%%%%%%%%%%%%%%%%%%%%
%%%%%%%%%%%%%%%%%%%%%%%%%%%%%%%%%%%%%%%%%%%%%%%%%%%%%%%%%%%%%%%%%%%
%%%%%%%%%%%%%%%%%%%%%%%%%%%%%%%%%%%%%%%%%%%%%%%%%%%%%%%%%%%%%%%%%%%
%%%%%%%%%%%%%%%%%%%%%%%%%%%%%%%%%%%%%%%%%%%%%%%%%%%%%%%%%%%%%%%%%%%

\section{Parabolic pseudos and Fredholm theory
  for the Schr\"odinger operator}\label{sec:tech}

\subsection{Microlocal analysis} Our approach to NLS is essentially microlocal. We work on the cotangent bundle $T^* \RR^{n+1}$, with spacetime coordinates $(z,t) \in \RR^n \times \RR$ and dual momentum coordinates $(\zeta, \tau)$. We work particularly on a compactification of $T^* \RR^{n+1}$, which is the Cartesian product of the usual radial compactification of $\RR^{n+1}_{z, t}$ and the \emph{parabolic} compactification of the dual space $\RR^{n+1}_{\zeta, \tau}$. For the spacetime factor, this means that a boundary defining function $\rhobase$ for the boundary of the compactification of $\RR^{n+1}_{z, t}$ (often termed spacetime infinity) is $\rhobase := \ang{Z}^{-1}$ where $Z = (z, t)$, and hence $\ang{Z} := (1 + |z|^2 + t^2)^{1/2}$. Angular coordinates near the boundary can be taken to be $z_j/t$ in the region $\{ |z| \leq 2 |t| \}$, or $\hat z := z/|z|$ and $t/|z|$ in the region $\{ |z| \geq  |t|/2 \}$. For the momentum factor, on the other hand, a boundary defining function $\rhofib$ for the boundary of its compactification (often termed fibre infinity) is $\rhofib := R^{-1}$ where $R := ( |\zeta|^4 + \tau^2)^{1/4}$ is a parabolic `radial function' in the fibre, and angular variables near fibre infinity can be taken to be either $\zeta_j/\sqrt{|\tau|}$ in the region $\{ |\zeta| \leq 2 \sqrt{|\tau|} \}$, or $\hat \zeta$ and $\tau/|\zeta|^2$ in the region $\{ |\zeta| \geq  \sqrt{|\tau|}/2 \}$. We note that $\ang{Z}$ is asymptotically homogeneous of degree one, and the spacetime angular variables are homogeneous of degree zero, under scaling in $Z = (z,t)$. Similarly, 
 $R$ is homogeneous of degree one, and the fibre
 angular variables homogeneous of degree zero, under the parabolic
 scaling $(\zeta, \tau) \mapsto (a\zeta, a^2 \tau)$ for $a > 0$ in the
 fibre directions. We denote the compactified cotangent bundle by
 $\SymbSpa$. It is a manifold with corners of
 codimension two.

We use the algebra of \emph{parabolic scattering pseudodifferential operators} introduced in \cite{TDSL}. This is 
easiest to describe in the case of classical symbols. By definition, a classical parabolic symbol of order $(s, r)$ on $\SymbSpa$ is a function in the space $R^{s} \ang{Z}^r C^\infty(\SymbSpa)$. More generally, parabolic symbols (not necessarily classical) are smooth functions $a$ on the cotangent bundle $T^* \RR^{n+1}$ obeying the estimates 
\begin{equation}
\Big| D_z^\alpha D_t^l D_\zeta^\beta D_\tau^m a(z, t, \zeta, \tau) \Big| \leq C_{\alpha, \beta, l, m} \ang{R}^{s - |\alpha| - 2m}\ang{Z}^{r - |\alpha| - l}  .
\end{equation}
We denote the class of such symbols by $S^{s,r}_{\mathrm{par}}(\RR^{n+1})$. Observe that the symbols $\zeta_i$ are order $(1,0)$, while $\tau$ is order $(2,0)$ in this symbol class. Also, notice that each $\tau$-derivative applied to $a$ reduces the parabolic regularity order $s$ by
$2$ due to the parabolic scaling. Notice also that derivatives with
respect to the spacetime variables reduce the order of growth at
spacetime infinity, just as derivatives in the fibre variables reduce
the order of growth at fibre infinity. This is the defining feature of
the \emph{scattering calculus} as introduced by H\"ormander \cite{hormander1979weyl}, Parenti \cite{Parenti} and Melrose \cite{RBMSpec}. As emphasized by Melrose, it means that microlocal analysis takes place at \emph{both} boundary hypersurfaces of $\SymbSpa$ --- spacetime infinity as well as fibre infinity --- as will be explained more below. 

The class of such parabolic symbols, resp. classical parabolic symbols, of order $(s, r)$ is denoted $\Spar sr$, resp. $\Spcl sr$. From here on we will usually omit the qualifier `parabolic'. The corresponding class of operators is denoted $\Psip sr$, resp. $\Psipcl sr$. We refer to $s$ as the \emph{parabolic regularity order} (which we prefer to `differential order' since $D_t$ is order $1$ as a differential operator but has parabolic regularity order $2$) and to $r$ as the \emph{spacetime decay order} or \emph{spacetime weight}. 

One key advantage of working in a parabolic calculus is that the
operator $P$ is an operator of real principal type in this calculus
away from the radial sets, so standard microlocal propagation
estimates will apply.  (These are essentially a combination of the
inhomogeneous propagation theory in \cite{lascar} and the scattering
calculus propagation in \cite{RBMSpec} and is proven rigorously in our
linear paper \cite{TDSL}.)  % On the other hand, the fact that $D_t$
% because an operator of second order leads to some difficulties (and
% eventually to slightly weaker results) than would be the case if we
% were working in a standard calculus (which we would do, for example,
% if analyzing the Klein-Gordon equation).

The usual elements of pseudodifferential calculus apply: 
\begin{itemize}
\item The parabolic scattering pseudodifferential operators form a graded algebra, i.e.\ 
\begin{equation}\label{eq:comp}
\Psip sr \circ \Psip {s'}{r'} \subset \Psip {s+s'}{r+r'},
\end{equation}
\item 
there is a principal symbol map 
$$\sigma^{s,r} : \Psip sr \to S^{s,r}(\SymbSpa) / S^{s-1,r-1}(\SymbSpa),
$$
 which is multiplicative under composition, 
\item Composition is commutative to leading order,  that is, commutators satisfy
 \begin{equation}\label{eq:comm}
\Big[ \Psip sr ,  \Psip {s'}{r'} \Big] \subset \Psip {s+s'-1}{r+r'-1}, 
\end{equation}
and the principal symbol of the commutator is the Poisson bracket of the principal symbols, or equivalently the Hamilton vector field of one symbol applied to the other: 
\begin{multline}\label{eq:Poissonbracket}
\sigma^{s+s'-1, r+r'-1}(i[A, B]) = \{ a,  b\} = H_{a} b, \quad 
a = \sigma^{s,r}(A), \quad b = \sigma^{s',r'}(B). 
\end{multline}
\end{itemize}

In the remainder of this section, we assume for ease of exposition that all our operators are classical. For a classical operator $A$ of order $(s,r)$, let $a$ be the left-reduced symbol of $A$. (The choice of quantization will be irrelevant in our discussion below.) Let $\tilde a = \ang{R}^{-s} \ang{Z}^{-r}  a$, which is a smooth function on $\SymbSpa$. The \emph{elliptic set} of $A$, denoted $\Ell(A)$, is the open subset of $\partial \SymbSpa$ (the union of spacetime infinity and fibre infinity) where $\tilde a \neq 0$. The \emph{characteristic variety} of $A$, denoted $\Char(A)$, is its complement: the zero set of $\tilde a$ restricted to  $\partial \SymbSpa$. 

For operators such as $P$, with a nonempty characteristic variety, the
microlocal estimates we need are usually easily obtained on the
elliptic set --- they follow directly from microlocal elliptic
regularity. We therefore focus on the characteristic variety. Here, as
is well known since H\"ormander's work in the 1970s \cite{Hormander:Existence}, regularity (and singularities) of solutions to $Pu = 0$ propagate along \emph{null bicharacteristics}, which by definition are the flow lines of the Hamilton vector field within $\Char(A)$. (Since the Hamiltonian is constant along the flow, any bicharacteristic that meets $\Char(A)$ is contained within $\Char(A)$.) In the present context, if a (classical) symbol has order $(1,1)$, then it is not hard to check that the Hamilton vector field is a b-vector field on $\SymbSpa$, that is, it is smooth and tangent to the boundary. More generally, for a symbol $a$ of order $(s,r)$, then $\ang{R}^{-s+1} \ang{Z}^{-r+1} H_a$ is smooth and tangent to the boundary; we denote this rescaled vector field $H^{s,r}_a$. This rescaling only has the effect of reparametrizing the flow on $\Char(A)$, since $H_{\rho a} = \rho H_a + a H_{\rho} = \rho H_a$ on $\Char(A)$, for any function $\rho$. 

Where $H^{s,r}_a$ is nonvanishing as a smooth vector field on the characteristic variety (either at fibre infinity or spacetime infinity), we have `propagation of regularity' as usual in the study of operators of real principal type. Thus, the complement, where $H^{s,r}_a$ vanishes as a smooth vector field, is of special importance. This is known as the \emph{radial set}, so-called because at such a point at fibre infinity, the Hamilton vector field must be a multiple of the `radial' vector field (in the sense of the generator of parabolic dilations, $\zeta \cdot \partial_\zeta + 2 \tau \partial_\tau$) in the fibre, while at spacetime infinity, the Hamilton vector field must be a multiple of the radial vector field in spacetime, $z \cdot \partial_z + t \partial_t$.    In the case of our operator $P$, the radial set $\SR$ is contained in spacetime infinity, although it meets the corner. It has two components, $\SR_-$ and $\SR_+$. Since the symbol of $P$ in a neighbourhood of spacetime infinity is $\tau + |\zeta|^2$, the Hamilton vector field is $\partial_t + 2 \zeta \cdot \partial_z$. It is clear that this is a multiple of $z \cdot \partial_z + t \partial_t$ exactly when $\zeta = z/2t$ for $|z|/|t| \leq C$. When $|z|/|t|$ is large, we rescale the Hamilton vector field by multiplying by $|z|/|\zeta|$, obtaining 
$$
\frac{|z|}{|\zeta|} \frac{\partial}{\partial t} + 2 \hat \zeta |z| \cdot \frac{\partial}{\partial z}.
$$
For this to be a multiple of $z \cdot \partial_z + t \partial_t$ we require that $|z|/|\zeta| = |t|$ and either $t \leq 0$ together with $\hat \zeta = - \hat z$, or $t \geq 0$ and $\hat \zeta = +\hat z$. In addition, we have $\tau = -|\zeta|^2$ to be on the characteristic variety. We see from this that the radial set $\SR$ has two components $\SR_-$ and $\SR_+$, which each diffeomorphic to a radially compactified $\RR^n$ with coordinate $\zeta$. They are given in the region $|z|/|t| \leq 2C$ by 
\begin{equation}\label{eq:radialsetspole}
\SR_\pm = \{ \frac{z}{t} = 2\zeta, \tau = - |\zeta|^2, \pm t > 0 \}
\end{equation}
and in the region $|z|/|t| \geq C$ in the following way. We use coordinates $s = t/|z|$, $\hat z$, $\rho_f = |\zeta|^{-1}$, $\hat \zeta$ and $\tau/|\zeta|^2$, and in these coordinates we have 
\begin{equation}\label{eq:radialsetseq}
\SR_\pm = \{s = \pm \frac{\rho_f}{2}, \hat z = \pm \hat \zeta, \frac{\tau}{|\zeta|^2} = -1 \}.
\end{equation}
\begin{figure}
	\centering
	\begin{tikzpicture}[scale=0.8]
		\draw[thick] (5,0) arc (0:360:5cm);
		\draw[dashed] (0,0) ellipse (5cm and 0.9cm);
		\draw[thick] (-5,0) arc (180:360:5cm and 0.9cm);
		%	\draw[dotted] (5,0) arc (0:180:5cm and 1cm);
		%	\draw (-5,0) arc (180:360:5cm and 1cm);
		\draw[->] (-3.536,3.536)--(-3.536*0.8,3.536*0.8) node[anchor=west]  {$\rho_{\mathsf{base}}$};
		
		\draw[->] (6-6,-1.5+1.5)--(6-6,0.2+1.5) node[anchor=west] {$t$};
		\draw[->] (6-6,-1.5+1.5)--(5.1-6,-2.2+1.5) node[anchor=west] {$z_1$};
		\draw[->] (6-6,-1.5+1.5)--(7.7-6,-1.5+1.5) node[anchor=south] {$z_2$};
	\end{tikzpicture}
	\caption{The radially compactified spacetime.  The function
		$\rhobase = (1 + t^2 + |z|^2)^{-1/2}$ defines (i.e.\ vanishes at) spacetime
		infinity.  The radial set $\mathcal{R}_+$ lies over the northern
		hemisphere, while the radial set $\mathcal{R}_-$ lies over the southern
		hemisphere. The equator $E$ is the illustrated great circle.}
	\label{fig1}
\end{figure}
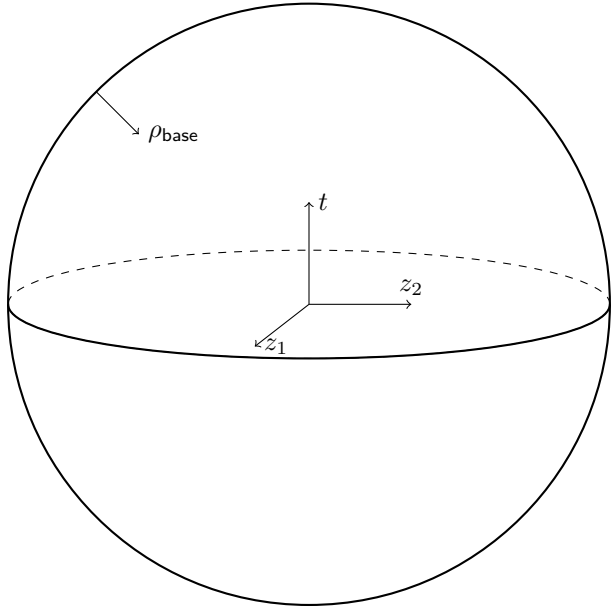

In a picture where compactified spacetime is represented as a solid ball with the time axis vertical (see Figure \ref{fig1}), and so spacetime infinity is the surface of the ball, the radial set $\SR_-$, which we call the incoming radial set, lies over the southern hemisphere $\{ s \leq 0 \}$, while $\SR_+$, which we call the outgoing radial set, lies over the northern hemisphere $\{ s \geq 0 \}$. They both have boundary over the equator $E$, given by 
\begin{equation}
E = \{ \ang{Z}^{-1} = 0, s = 0 \},
\label{eq:equator} \end{equation}
where they reach fibre infinity. The coordinate $\hat z = \pm \hat \zeta$ distinguishes the two components $\SR_\pm$ of the radial set. 

We also make some comments about the rescaled Hamilton vector field at fibre infinity. Here, the rescaling involves multiplication by $\ang{R}^{-1}$. This renormalizes the time component $\partial_t$  to zero, and we obtain, on each time slice, geodesic flow with respect to the metric $g_t$. This reflects ``infinite propagation speed'' for the Schr\"odinger equation. The flow here is clearly non-radial, as there is a nonzero spatial component. Thus there are no radial points at fibre infinity (except at the corner, at the boundary of $\SR_\pm$). 

In summary, we can divide the boundary of compactified phase space, into three disjoint sets: the elliptic region, the characteristic variety where the Hamilton vector field is non-radial, and the radial set. Each requires a different type of microlocal estimate to assemble the overall Fredholm estimate. First, though, we need to discuss modules and module regularity spaces.

%Phase space, compactification, Elliptic vs characteristic variety. Hamilton vector field, radial sets. Mention and refer to microlocal propagation estimates, but refer to previous paper. 

%Introduce c parameter. Note that $\ubar$ solves NLS with $c=-1$. 
\subsection{Parabolic Sobolev spaces}
For real parameters $s, r$, we define parabolic Sobolev spaces $\Hpar sr$  such that 
\begin{equation}
u \in \Hpar sr \Leftrightarrow Au \in L^2(\RR^{n+1}) \text{ for all } A \in \Psip sr. 
\end{equation}
This condition is equivalent to $Eu \in L^2(\RR^{n+1})$ for a single globally elliptic $E = E_{s,r}\in \Psip sr$. It is convenient to choose $E_{s,r}$ to be invertible, for example, $E_{s,r} = \Op_L( \ang{Z}^r \ang{R}^s ) = \ang{Z}^r \Op(\ang{R}^s)$. We define an inner product on this space by 
\begin{equation}
\ang{u, v} = \ang{E_{s,r}u, E_{s,r} v }_{L^2},
\end{equation}
where the inner product on the RHS is the standard inner product on $L^2(\RR^{n+1})$. This gives $\Hpar sr$ a Hilbert space structure. 

The usual H\"ormander ``square root'' trick, and the pseudodifferential calculus, shows that operators of order $(0,0)$ are bounded on $\Hpar sr$. Moreover, elliptic regularity works in the usual way. Namely, if $A \in \Psip sr$ is elliptic on a set $U \subset \partial T^* \RR^{n+1}$, then there exists a microlocal inverse, that is, a $B \in \Psip {-s}{-r}$ so that $AB$ is equal to the identity microlocally on $U$, in the sense that $Q (AB - \id)$ is an operator of order $(-\infty, -\infty)$ for any $Q$ with $\WF'(Q) \subset U$. 

We also note that operators $A \in \Psip sr$ for $s, r < 0$ are compact on $L^2(\RR^{n+1})$, and consequently, the space $\Hpar sr$ is compactly embedded in $\Hpar {s'}{r'}$ provided that $s > s'$ and $r > r'$.

\subsection{Microlocal propagation estimates}
Our global estimates, discussed in the following subsection, arise from combining microlocal propagation estimates that apply in different regions of the characteristic variety, with elliptic estimates that apply in the elliptic region. In these estimates, we denote by $Q, Q'$, etc, operators in $\Psipcl 00$ with microlocal support in a sufficiently small neighbourhood of a point or region in phase space, that will be applied to $u$. We will use $G$ for a similar microlocalizing operator that will be applied to $Pu$. 

We first review elliptic estimates. These take the following form. 

\begin{prop}\label{prop:elliptic.estimate}
Assume that $Q \in \Psipcl 00$ is such that $\WF'(Q)$ is contained in $\Ell(P)$, and suppose that $G \in \Psipcl 0 0$ and $P$ are both elliptic on $\WF'(Q)$. Then for arbitrary $s$, $r$, $M$ and $N$ we have 
Then if $GPu\in \Hpar{s-2}{r}$, we have $Qu\in \Hpar sr$ with the estimate
\begin{equation}\label{eq:elliptic.estimate}
\|Qu\|_{\Hpar sr}\leq C_{M,N} (\|GPu\|_{\Hpar{s-2}{r}}+\|u\|_{\Hpar MN}).
\end{equation}
\end{prop}

Here we think of $M$ and $N$ as being very negative, so that the norm on $u$ is much weaker than the norm on $Qu$. 
Notice that the order of the norm on $GPu$ is the order $(s,r)$ of the norm on $Qu$ minus $(2, 0)$, the order of $P$.

Next, we consider $u$ in a microlocal neighbourhood of a point $q$ in $\Char(P)$. The most straightforward case is when $q$ is not a radial point. That means, by definition, that the rescaled Hamilton vector field $H_p^{2,0}$ is nonvanishing at $q$; in particular, since it is a tangent vector field, it has a nonvanishing tangential component. Then the microlocal propagation estimate works as follows: let $\gamma$ be the bicharacteristic through $q$, and suppose that we know that $u$ is  in $\Hpar sr$ microlocally near another point $q'$ of the bicharacteristic $\gamma$. Then, if $Pu$ is sufficiently regular in a neighbourhood of the bicharacteristic segment from $q'$ to $q$, and $u$ is known to have some weak background regularity of order $(M,N)$,  then $u$ is  in $\Hpar sr$ microlocally near $q$. More precisely, we have 

\begin{prop}[Propagation of regularity]\label{prop:sing.P}
	Let $Q,Q',G\in\Psipcl{0}{0}$ be operators of order $(0,0)$ with $G$ elliptic on $\WF'(Q)$. Let $s, r, M, N$ be arbitrary real numbers. Suppose that for every $\alpha\in\WF'(Q)\cap\Sigma(P)  $ there exists $\alpha'$ such that $Q'$ is elliptic at $\alpha'$ and there is a bicharacteristic $\gamma$ of $P$ from $\alpha'$ to $\alpha$ such that $G$ is elliptic on the bicharacteristic segment along $\gamma$ from $\alpha$ to $\alpha'$. 
	
Then if $GP u\in \Hpar{s-1}{r+1}$ and $Q'u\in \Hpar sr$, we have $Q u\in \Hpar sr$ with the estimate
\begin{equation}\label{eq:sing.P}	
	\|Qu\|_{\Hpar sr}\leq C \Big(\|Q' u \|_{\Hpar sr}+\|GP u\|_{\Hpar{s-1}{r+1}}+\|u\|_{\Hpar MN} \Big).
	\end{equation}
\end{prop}

There are two key differences compared to the elliptic estimate: first, we need an a priori assumption of regularity of $u$ elsewhere along the bicharacteristic, and second, the assumed regularity of the $GPu$ term is increased from $(s-2, r)$ in the elliptic case to $(s-1, r+1)$ here. That is, we need one order of additional regularity for the $GPu$ term in both the spacetime and fibre senses. This is often referred to as a `loss of regularity' of order one compared to the elliptic case. 

The remaining case is a microlocal estimate near a point which is on one of the radial sets $\SR_\pm$. In this case there is an additional subtlety of a threshold value $r = -1/2$ for the spacetime order. This can be understood as follows: either from mass conservation (the spatial $L^2$ mass of any solution of $Pu = 0$ is constant for large $|t|$) or from the asymptotic expansion \eqref{eq:fminus}, \eqref{eq:fplus}, we see that there are no nontrivial solutions of $Pu = 0$ in the space $\Hpar sr$ for $r \geq -1/2$, but there are an infinite-dimensional family of solutions for $r < -1/2$, parametrized by $f \in \mathcal{S}(\mathbb{R}^n)$ in \eqref{eq:fminus} for example. It is not surprising, therefore, that the microlocal estimates have to take into account whether $r$ is below or above threshold. It turns out that the threshold only matters at the radial sets.

When $r < -1/2$ is below threshold, we have a result very similar to that of Proposition~\ref{prop:sing.P}. 

\begin{prop}[Below threshold radial point estimate]
	\label{prop:belowschrod}
	Suppose $r<-\frac{1}{2}$, and let $s, M, N \in \RR$. Assume there exists a neighbourhood $U$ of $\mathcal{R}_\pm$ and $Q', G\in\Psipcl 00 $ such that for every $\alpha\in \Char(P) \cap U \setminus \mathcal{R}_\pm$ the bicharacteristic $\gamma$ through $\alpha$ enters $\Ell(Q')$ whilst remaining in $\Ell(G)$. Then there exists $Q\in\Psipcl 00$ elliptic on $\mathcal{R}_\pm$ such that if $u \in \Hpar MN$,  $Q'u\in \Hpar sr$ and $GP u\in \Hpar{s-1}{r+1}$, then $Qu\in \Hpar sr$ and  there exists $C>0$ such that
\begin{equation}\label{eq:radial.schrod.below}
\|Qu\|_{\Hpar sr}\leq C \Big(\|Q'u\|_{\Hpar sr} +\|GP u\|_{\Hpar{s-1}{r+1}}+\|u\|_{\Hpar MN} \Big).
\end{equation}
\end{prop}

The above threshold works a little differently; it is stronger in one respect, and weaker in another. It is stronger in the sense that we no longer need to assume a priori regularity of $u$ in $\Hpar sr$ in some other region in order to obtain $\Hpar sr$ regularity near the radial set. On the other hand, we do need to assume a priori that $u$ is above threshold regularity near the radial set. This explains the presence of the $Gu$ term below. 

\begin{prop}[Above threshold radial point estimate]
	\label{prop:aboveschrod}
	Suppose $r>r'>-\frac{1}{2}$ and $s>s'$, and let $M,N \in \RR$. Assume that $G\in\Psip{0}{0}$ is elliptic at $\SR_\pm$. Then there exists $Q\in\Psipcl 00$ elliptic at $\SR_\pm$ such that, if $u \in \Hpar MN$, $Gu\in \Hpar{s'}{r'}$ and $GP u\in \Hpar{s-1}{r+1}$, then $Qu\in \Hpar sr$ and there exists $C>0$ such that
	\begin{equation}\label{eq:radial.schrod.above}
		\|Qu\|_{\Hpar sr}\leq C \Big(\|GP u\|_{\Hpar{s-1}{r+1}}+\|Gu\|_{\Hpar{s'}{r'}}+\|u\|_{\Hpar MN} \Big).
	\end{equation}
\end{prop}

\subsection{Fredholm estimates} 
The microlocal propagation estimates can be combined to prove a global ``Fredholm estimate'' (actually a pair of dual estimates) that has, as a consequence, that $P$ is a Fredholm operator between suitable Hilbert spaces. In \cite{TDSL} the authors showed that $P$ is actually an invertible map between the same spaces. However, there is a catch: to make the Fredholm estimate work, we need to have the spacetime  order $r$ \emph{above} a threshold value $-1/2$ at one of the radial sets, and \emph{below} the threshold at the other radial set. Roughly speaking the need to be above threshold at one radial set is so we can exploit Proposition~\ref{prop:aboveschrod} in which regularity ``appears'' at the radial set without having to ``flow in'' from elsewhere (as shown by the absence of a $\| Q' u\|_{\Hpar sr}$ term on the RHS of this estimate, in contrast with Propositions~\ref{prop:sing.P} and \ref{prop:belowschrod}). The need to have one radial set below threshold is because a below threshold radial set becomes above threshold for the dual estimate. 

The threshold value of $-1/2$ can be understood by returning to the asymptotic expansions \eqref{eq:fminus} and \eqref{eq:fplus}.  For $f \in \mathcal{S}(\mathbb{R}^n)$ and any $s \in \mathbb{R}$, these expressions are in $H^{s, r}_{\mathrm{par}}(\RR^{n+1})$ for every $r < -1/2$ but not for $r = -1/2$, so this index is the threshold value below which global solutions can exist. From a technical point of view $r = -1/2$ is the threshold for which the commutant in the required positive commutator estimate changes from positive (spacetime) order to negative, which changes the positivity properties of various terms in the estimate. 

We thus define a \emph{variable} order $\sw_+$, which is a classical symbol of order zero on phase space, that is above $-1/2$ at $\SR_+$ and below threshold at $\SR_-$, and (to allow propagation of regularity in between) is monotone with respect to bicharacteristic flow (i.e. the Hamiltonian flow of $p$). If we define $\sw_- := -1 - \sw_+$, then this has the opposite property: $\sw_-$ is above $-1/2$ at $\SR_-$ and below threshold at $\SR_+$, and monotone (in the opposite direction) with respect to bicharacteristic flow. We then obtain the Fredholm estimates\footnote{We need to work here in a slightly larger calculus of operators with variable orders. These operators are of type $(1, \delta)$ rather than $(1, 0)$ for any $\delta > 0$, thus for example the commutator of two operators as in \eqref{eq:comm} will only lie in $\Psip {s+s'-1+\delta}{r+r'-1+\delta}$. However the differences are inessential and we do not belabour the point here.}
\begin{align}\label{eq:Fredholm}
\| u \|_{\Hpar s {\sw_+}} &\leq C \Big( \| Pu \|_{\Hpar {s-1}{\sw_+ +1}} + \| u \|_{\Hpar MN} \Big), \\
\| u \|_{\Hpar {1-s} {\sw_-}} &\leq C \Big( \| P^*u \|_{\Hpar {-s}{\sw_- +1}} + \| u \|_{\Hpar MN} \Big)
\end{align}
for any $s, M, N \in \RR$. 
These two estimates are dual because the negative of $\sw_+ + 1$ is $\sw_-$, and similarly the negative of $\sw_+$ is $\sw_- + 1$.

\subsection{Preliminary multiplication results}
The main new aspect of this paper is a multiplication result for module regularity spaces. The module regularity result is based on Sobolev algebra properties: functions with $H^s$ regularity on a Riemannian manifold of dimension $N$, with Ricci curvature bounded below and positive injectivity radius, form an algebra under pointwise multiplication provided that $s > N/2$ \cite{MR1828225}. A simple adaptation of the proof for Euclidean space shows that the \emph{parabolic} Sobolev spaces, $\Hpar s0$ on $\RR^{n+1}$ form an algebra for $s > (n+2)/2$ (with the increase of the threshold by $1/2$ due to parabolicity). It is immediate that if $r_1, r_2$ are constant spacetime weights, then we have a bounded bilinear map
$$
\Hpar s{r_1} \cdot \Hpar s{r_2} \subset \Hpar s{r_1 + r_2}, \quad s > \frac{n+2}{2}.
$$

This result, however, is not enough for applications to NLS. The reason is, as we have seen, in the product $N[u]$, the factors are all below threshold, that is, their spacetime weight must be below $-1/2$. Thus applying this product rule will produce a more and more negative spacetime weight. However, to apply the inverse $P_\pm^{-1}$, we need to land in a space with spacetime weight \emph{larger} than $r$, namely $r+1$, in order to ``close the loop'' and be in a position to apply the Contraction Mapping Theorem. 

To overcome this difficulty, we consider for a moment vector fields with linear or constant coefficients. It will be convenient to temporarily discuss this just on a spatial slice, $\RR^n$, to avoid difficulties with the parabolic calculus. These are vector fields of bounded length with respect to a conformal metric $g_b = \ang{z}^2 dz^2$. We therefore have a Sobolev algebra property for this metric. The measure $dg_b$ is $\ang{z}^n$ times the Euclidean measure; this means that the weighted $L^2$ spaces relative to the b-measure are the same as weighted $L^2$ spaces relative to the Euclidean measure, but with a discrepancy of $n/2$ in the weight. If we write the spaces \emph{relative to the Euclidean measure}, therefore, we obtain a mutliplication result 
\begin{equation}\label{eq:b-mult}
H_b^{s, r_1}(\RR^n, dz) \cdot H_b^{s, r_2}(\RR^n, dz) \subset H_b^{s, r_1 + r_2 + n/2}(\RR^n, dz). 
\end{equation}

Away from the corner of the radial sets, in sets of the form
$|1/t | < C, |z/t| < C$, it is to be expected by comparison with
similar situations, in particular from \cite{NLSM}, that distributions
in the module regularity space should be conjugate to distributions in
$H_b$ spaces.  Indeed, removing the oscillation by multiplying by
$e^{-i |z|^2/(4t)}$ reduces the module to one consisting of
  $b$-derivatives.  Outside such regions, we must avoid the
  singularity at $t = 0$, and this is accomplished in Section
  \ref{sec:mod reg} below by our time cutoff decomposition.

\section{Module regularity}\label{sec:mod reg}

\subsection{Module regularity spaces}

We will define module regularity spaces $H_{\modu_c}^{s,r;k}$ within which membership
of $u$ can be characterized by application of a specific, natural
family differential (as opposed to arbitrary pseudodifferential)
operators characteristic to $\SR_c$.  Below we set
$$
\Psi_{\mathrm{par, cl}}^{\mathbb{N}, \mathbb{N}} = \bigcup_{m,l \in
  \mathbb{N} \times \mathbb{N}} \Psi_{\mathrm{par, cl}}^{m,l}.
$$

As foreshadowed in the Introduction, we will adjust the definition of module regularity space used in \cite{} to allow for generators that have parabolic regularity order greater than $1$. This is to allow operators such as $D_t$, which have order $2$ in the parabolic calculus, to be generators. 
We begin with some abstract definitions for general modules and module
regularity spaces, starting with classes of generators. 
\begin{defn}
	\label{def:testmod}
A finite set $\mathbf{A} = \{ A_0 = \Id, A_1, \dots, A_N \} \subset \Psi_\mathrm{par, cl}^{\mathbb{N}, \mathbb{N}}$ of operators of integer order is called an \emph{admissible set of generators} if it contains the identity, and satisfies the following conditions:
\begin{itemize}
\item $\mathbf{A}$ is closed under commutators in the sense that, for all $A_i, A_j$ in $\mathbf{A}$,  there exist operators $C_{ijk} \in \Psip 00$ such that 
\begin{equation}
[A_i, A_j] = \sum_k C_{ijk} A_k ,
\label{eq:commgen}\end{equation}
\item  For all  $A \in \mathbf{A}$  and all $m,l \in \mathbb{R}$, we have 
\begin{equation}
[A, \Psip ml] \subset
    \Psip ml .
\label{eq:module properties}
\end{equation}
\end{itemize}
%\item For all $A \in \mathbf{A}$,      [A, \mathcal{M}^{(k)}] \subset \mathcal{M}^{(k)} \label{eq:module properties}
%  \end{split}
%  \end{equation}
\end{defn}

Given an admissible set of generators, we generate a sequence of modules over the algebra $\Psip 00$. 

\begin{defn}
	\label{def:module.power}
Let $\mathbf{A} = \{ A_0 = \Id, A_1, \dots, A_N \}$ be an admissible set of generators, and let $\SM$ denote the module over $\Psip 00$ generated by $\mathbf{A}$. We define the \emph{reduced powers} $\SM^{(k)}$, $k = 1, 2, \dots$ to 
	to be the $\Psip{0}{0}$-modules generated by the sets
	\begin{equation}
		\{A_\alpha=\mathbf{A}^\alpha:= A_1^{\alpha_1} \cdots
                A_N^{\alpha_N}  \colon|\alpha|\leq k \textrm{ and }A_\alpha\in\Psip{k}{k}\}.
\label{eq:Mkgen}	\end{equation}
It is straightforward to check that this definition is independent of the ordering of the $A_i$, due to properties \eqref{eq:commgen} and \eqref{eq:module properties}. 
\end{defn}

\begin{remark} In previous works, the generators were required to be in $\Psipcl 11$. In that case, we have $\SM^{(1)} = \SM$ and $\SM^{(k)} = \SM^k$, that is, the module generated by $k$-fold products of elements of $\SM$. Clearly this no longer holds when we have generators with parabolic regularity order $2$. However, the properties \eqref{eq:commgen} and \eqref{eq:module properties} turn out to be sufficient to employ these new module sequences $\SM^{(k)}$ in essentially the same way the module powers $\SM^k$ were employed in previous works. 
\end{remark}

We define module regularity spaces with respect to Sobolev spaces as follows.
\begin{defn}
	\label{def:module.reg}
Let $\SM$ be a module generated by an admissible set of generators $\mathbf{A}$, and let $s, l \in \RR$. The module regularity space of order $k$ over the parabolic Sobolev space $\Hpar sl$ is defined by 
	\begin{equation}\label{eq:1module}
		H_\SM^{s,l;k}:=\{u\in H^{s,l}_{\mathrm{par}}(\RR^{n+1}):A u \in H^{s,l}_{\mathrm{par}}\textrm{ for all } A\in\SM^{(k)}\}.
	\end{equation}
\end{defn}

We equip the space $H_\SM^{s,l;\kappa}$ with a Hilbert space structure by fixing a choice of generators $\mathbf{A}$ and taking
\begin{equation}
\label{eq:mod.reg.sob.norm}
	\|u\|_{H_\SM^{s,l;k}}^2:= \sum_{\mathbf{A}^\alpha \in \SM^{(k)}} \|\mathbf{A}^\alpha  u\|_{H_\mathrm{par}^{s,l}}^2
\end{equation}
Definition \ref{def:module.reg} generalises in the natural way to the case of variable spatial weight $\mathsf{I}\in\Psip{0}{0}$.
\begin{rem}
  Given a generating set $\mathbf{A}$, by definition it is the reduced powers $\mathcal{M}^{(k)}$ which determine module
  regularity.  The module $\SM$ is essentially a convenient piece of
  notation, referring to the notion of module regularity.
\end{rem}

At this point, it is natural to ask whether the module regularity spaces depend only on the module $\SM$, or also on the choice of generators. The answer is subtle. We show in the following lemma that the module regularity spaces only depend on the module $\SM$. \emph{However}, property \eqref{eq:module properties} is not invariant under a change of generators. This property will be crucial to prove the positive commutator estimates on which this paper relies. So, it is crucial that the module $\SM$ possesses an admissible set of generators satisfying \eqref{eq:commgen} and \eqref{eq:module properties}. 

\begin{lemma}\label{thm:order doesnt matter}
  Assume that a generating set $\mathbf{A}$ with module $\mathcal{M}$
  satisfies \eqref{eq:module properties}.  Then the module regularity spaces $H_\SM^{s,l;k}$
  only depend on $\SM$, and not on the choice of generators $\mathcal{A}$. Indeed, we have 
  \begin{equation}
\mathcal{M}^{(k)} = \mathcal{M}^{k} \cap
\Psi_{\mathrm{par}}^{k,k},\label{eq:characterization}
\end{equation}
where $\mathcal{M}^{k}$ is the $k$-th power of the module
(the linear span of compositions of up to $k$ elements from $\mathcal{M}$.)
\end{lemma}
\begin{proof}
% Note that all of these properties are immediate for $k = 1$, including
% the final property \eqref{eq:product of modules} if we define
% $\mathcal{M}^{(0)} := \Psi_{\mathrm{par}}^{0,0}$ in which case
% \eqref{eq:product of modules} is true for$k_1 + k_2 \le 1$
% (i.e. either $k_1 = 0, k_2 = 1$ or $k_1 = 1, k_2 = 0$.)  Note that
% When $k_1 = 0, k_2 = 1$ this is just from the definition while when
% $k_1 = 1, k_2 = 0$ it follows from $[\Psi_{\mathrm{par}}^{0,0},
% \mathcal{M}^{(1)}] \subset \Psi_{\mathrm{par}}^{0,0}$.

%We prove the independence of the module on generator order by
%induction on the order.  Thus, assume for induction that the modules
%$\mathcal{M}^{(k)}$ do not depend on order of generators for all
%$k' < k$.  Let
%$B = A_{i_1} \cdots A_{i_r} \in \Psi_{\mathrm{par}}^{k,k}$ with the
%$A_{i_j} \in \mathbf{A}$, potentially not in the correct order.  We
%wish to show $B \in \mathcal{M}^{(k)}$ which amounts only to showing
%that it can be expressed in terms of sums of products of $A_j$ in the
%correct order.  But any transposition of the $A_{i_j}$'s incures a
%commutators terms which is a full parabolic regularity order lower, i.e.
%$$
%A_{i_1} \cdots A_{i_j} A_{i_{j + 1}} \cdots A_{i_r} =  A_{i_1} \cdots A_{i_j}
%A_{i_{j + 1}} \cdots A_{i_r} +  A_{i_1} \cdots [A_{i_j}, A_{i_{j + 1}}] \cdots A_{i_r}
%$$
%where the second term on the right lies in $\mathcal{M}^{(k - 1)}$.
%Thus modulo $\mathcal{M}^{(k - 1)}$ one can transpose these generators
%until they are in the correct order, proving the independence of
%order.

  Since the generators are themselves elements in the module, the containment $\mathcal{M}^{(k)} \subset \mathcal{M}^{k} \cap
  \Psi_{\mathrm{par}}^{k,k}$ is obvious from the definition.

  In the other direction, it suffices to show that if one has $j \le
  k$ and $Q_1
  \dots Q_j \in \Psi_{\mathrm{par}}^{0,0}$, then for $A_{i_1}, \dots,
  A_{i_j} \in \mathbf{A}$ with
  $$
L = Q_1 A_{i_1} \cdots Q_j A_{i_j} \in \Psi_{\mathrm{par}}^{k,k},
$$
then $L \in \mathcal{M}^{(k)}$.  If $j = 1$ this follows by
definition. For $j \geq 2$, and for $l$ of the $Q_i$ not equal to the identity,  we argue by induction: assume the result is true for all smaller values of $j$ or for the same value of $j$ and smaller values of $l$. When $l=0$ or $1$ the result is easy: either the products are already in the order \eqref{eq:Mkgen} in which case the result is immediate, or we reorder using \eqref{eq:commgen} which reduces to the case of $j' < j$. For larger values of $l$, we shift factors of $Q_i$ to the left, using the identity
%and for $j' = j$, whenever fewer than $j$ of the Q_i$ are not equal to the identity.  , while if $Q_2, \dots,  Q_j  = \Id$, then
%this follows again from the definition.  
%Arguing by induction in $j$ and
%in the number $l$ of such that $Q_l, \dots, Q_j = \Id$ note that
\begin{equation*}
  \begin{gathered}
    Q_1 A_{i_1} \cdots Q_{l-2} A_{i_{l-2}} Q_{l-1} A_{i_{l-1}} A_{i_l} \cdots
      A_{i_j} \\
      \qquad = Q_1 A_{i_1} \cdots Q_{l-2} Q_{l-1} A_{i_{l-2}} A_{i_{l-1}}
         \cdots  A_{i_j} + Q_1 [A_{i_1}, Q_{l-2}] A_{i_{l-2}} Q_{l-1} A_{i_{l-1}}  \cdots  A_{i_j}.
      \end{gathered}
\end{equation*}
By induction on $l$ the first term on the right lies in
$\mathcal{M}^{(k)}$, while, by the first property in \eqref{eq:module
  properties} $ [A_{i_1}, Q_{l-2}] \in \Psi_{\mathrm{par}}^{0,0}$, so
the second term on the right hand side lies in $\mathcal{M}^{(k)}$ by
induction on $j$. We can then reorder the generators $A_i$ as before. 
  \end{proof}
  
  \begin{remark} Lemma~\ref{thm:order doesnt matter} shows that the finiteness of the norm \eqref{eq:mod.reg.sob.norm} implies that $Au \in H^{s,l}_{\mathrm{par}}$ for all $A\in\SM^{(k)}$, so this is a suitable norm for the space \eqref{eq:1module}.
  \end{remark}
  
We next prove two straightforward results about boundedness of operators on module regularity spaces.

\begin{lemma}\label{thm:order really doesnt matter}
  Assume that $\mathbf{A}$ is an admissible set of generators of the module $\mathcal{M}$.  Let $u \in \mathcal{S}'$
  and let $L \in \mathcal{M}^{(k)}$.  Then
    $$
\| L u \|_{H^{s, \mathrm{r}}} \lesssim \| u \|_{H^{s, \mathrm{r};k}_{\SM}}.
  $$
\end{lemma}
\begin{proof}
  It suffices to show that if
  $\mathbf{B}^\alpha \in \mathcal{M}^{|\alpha|}$ with
  $B^\alpha \in \Psi^{k,k}_{\mathrm{par}}$.  Then
  $$
\| B^\alpha u \|_{H^{s, \mathrm{r}}} \lesssim \| u \|_{H^{s, \mathrm{r};k}_{\SM}}.
$$

For $|\alpha| = 1$, such $B^{\alpha}$ are finite sums of $Q A$ where
$A \in \mathbb{A} \cap \Psi^{1,1}_{\mathrm{par}}$ and
$Q \in \Psi^{0,0}_{\mathrm{par}}$.  Since $Q$ is bounded on
$H^{s,\mathrm{r}}_{\mathrm{par}}$ (see \cite[Sect.\ 2.4]{TDSL}) we have
  $$
\| QA u \|_{H^{s, \mathrm{r}}} \lesssim  \| A u \|_{H^{s, \mathrm{r}}}
\le \| u \|_{H^{s, \mathrm{r};1}_{\SM}},
$$
where the second inequality holds by definition.

For arbitrary $k$, writing $B^\alpha = Q_1 A_{i_1} \dots Q_j A_{i_j}$
for $j \le |\alpha|$, as in the previous lemma we can argue by
induction on $j$ and the number $l\le j$ such that $Q_l = \dots = Q_j
= \Id.$ \end{proof}

\begin{lemma}\label{thm:bounded on mod reg spaces}
  Let $B \in \Psi_{\mathrm{par}}^{m,l}$ with $m,l \in \mathbb{R}$, then for $s
  \in \mathbb{R}$ and $\mathrm{r}$ a variable order, then
  \begin{equation}
B \colon H^{s, \mathrm{r};k}_{\SM} \lra H^{s - m, \mathrm{r}
  - l;k}_{\SM}\label{eq:bounded on mod reg spaces}
\end{equation}
is bounded.
\end{lemma}
\begin{proof}
  By induction on $k$, the base case being the case $k = 0$ which is
  true by \cite[Sect.\ 2.4]{TDSL}.  Assuming it is true for $k$, let
    $A_i \in \mathbf{G}_c, i = 1, \dots, m$ have
    $A_1 \dots A_m \in \Psi^{k + 1,k + 1}_{\mathrm{par}}$.  Then
    \begin{equation*}
      \begin{gathered}
\| A_1 \dots A_m B u \|_{s - m, \mathrm{r} - l} = \| B A_1 \dots A_m u
\|_{s - m, \mathrm{r} - l} \\
\qquad \qquad \qquad + \sum_{j = 1}^m \| A_1 \dots A_{j - 1} [B,
A_j] A_{j + 1} \dots A_n u \|_{s - m, \mathrm{r} - l}
      \end{gathered}
    \end{equation*}
    The first term on the right is bounded $ \| A_1 \dots A_m u
\|_{s , \mathrm{r} } \le \| u \|_{H^{s,
    \mathrm{r};k}_{\SM}}$ by definition.  The terms in the sum are of the
form $\| A_1 \dots A_{j - 1} \tilde B A_{j + 1} \dots A_n u \|_{s - m,
  \mathrm{r} - l}$ with $\tilde B \in \Psi^{m,l}_{\mathrm{par}}$ but $A_1 \dots
A_{j - 1} \in \Psi^{k_1, k_1}_{\mathrm{par}}$ and $A_{j + 1} \dots A_{m} \in
\Psi^{k_2, k_2}_{\mathrm{par}}$ where $k_1 + k_2 \le k$, and thus by definition
and then induction
$$
\| A_1 \dots A_{j - 1} \tilde B A_{j + 1} \dots A_n u \|_{s - m,
  \mathrm{r} - l} \lesssim \| \tilde B A_{j + 1} \dots A_n u \|_{s - m,
  \mathrm{r} - l; k_1} \lesssim \| \tilde B u \|_{s - m,
  \mathrm{r} - l; k},
$$
which is what we wanted.
\end{proof}

The following corollary is then just the global elliptic parametrix
construction (discussed in \cite{TDSL}) and the boundedness given by
Lemma \ref{thm:bounded on mod reg spaces}
\begin{cor}\label{cor:Bmodinv}
Suppose $B \in \Psip m l$ is elliptic, meaning $\sigma_{m,l}(B)$ is
nonvanishing on $\partial ( \SymbSpa)$. Then the mapping
in \eqref{eq:bounded on mod reg spaces} is Fredholm, and its kernel
and cokernel lie in $\mathcal{S}(\mathbb{R}^{n + 1})$ and are
therefore independent of all regularity parameters.  In particular if
$\ker A = \{ 0 \} = \coker A$ then \eqref{eq:bounded on mod reg
  spaces} is a continuous bijection.
\end{cor}

%Below we prove that if the generators satisfy certain commutation
%properties then the norm above depends neither on the order of the
%generators nor the specific choice of generating set up to a natural
%notion of linear equivalence. 

\subsection{The modules $\modu_c$}
We now introduce the modules of chief interest in this article. These are indexed by a real number $c$. 
Define 
%We define the module $\modu_c$ generated by 
\begin{equation}\label{eq:G.def}
	\mathbf{G}_{c}:=\{D_{z_j},\ D_t,\ \mathrm{Rot_{ij}},\
        tD_{z_j}-cz_j/2,\ 2tD_t+z\cdot
        D_z,\ Id\}
      \end{equation}
      where $\mathrm{Rot_{ij}}=z_iD_{z_j}-z_jD_{z_i}$. \emph{We claim that $\mathbf{G}_c$ is an admissible set of generators in the sense of Definition~\ref{def:testmod}.} The first condition, \eqref{eq:commgen}, is obvious as the elements of $\mathbf{G}_c$ have coefficients that are either linear or constant, and $\mathbf{G}_c$ contains all the constant coefficient vector fields. As for the second condition, \eqref{eq:module properties}, this condition is automatically satisfied for all generators of order $(1,1)$, so we only need to check this for $D_t$ and for $2t D_t +  z \cdot D_z$. For any $B \in \Psip ml$, we have 
$$
\sigma(i[D_t, B]) = \partial_t \sigma(B), 
$$
and 
$$
\sigma(i[2tD_t + z \cdot D_t, B]) = H_{2t\tau + z \cdot \zeta} \sigma(B) \text{ modulo } S_{\mathrm{par}}^{m,l}. 
$$
So we need to check that the Hamilton vector fields of $D_t$ and for $2t D_t +  z \cdot D_z$ map $S_{\mathrm{par}}^{m,l}$ to $S_{\mathrm{par}}^{m,l}$. This is obvious for $D_t$, while the Hamilton vector field of $2t D_t +  z \cdot D_z$ is
$$
2t \partial_t + z \cdot \partial_z - 2\tau \partial_\tau - \zeta \cdot \partial_\zeta,
$$
and it is easily checked that this is a b-vector field on the compactified cotangent space, hence maps $S_{\mathrm{par}}^{m,l}$ to $S_{\mathrm{par}}^{m,l}$. This verifies \eqref{eq:module properties}. We can now define our modules $\modu_c$. 

\begin{defn}
The modules $\modu_c$ are those modules generated by the generators $\mathbf{G}_c$ in \eqref{eq:G.def}. We will usually write $\modu$ instead of $\modu_1$. 
\end{defn}

The form of expansions \eqref{eq:fminus} and \eqref{eq:fplus} explains why we need to consider a one-parameter family of modules $\modu_c$. We are interested in multiplicative properties of these module regularity spaces. Clearly, if we multiply two functions of the form $e^{i|z|^2/4t} \tilde u$, where $\tilde u$ has some conormal regularity, then the product will have the form $e^{i|z|^2/2t} w$ with 
$w$ similarly conormal. More generally, we can multiply functions with oscillatory prefactors $e^{ic_1|z|^2/4t}$ and $e^{ic_2|z|^2/4t}$, and the result will have oscillatory prefactor $e^{i(c_1 + c_2)|z|^2/4t}$. The modules associated with such prefactors are precisely the modules $\modu_c$, and for $c \neq 0$ they are characteristic at a scaled radial set which is the radial set for a modified Schr\"odinger operator $P_c$ where the operator $D_t$ is replaced by $c D_t$. Thus, our multiplicative result 
will show that products of functions with $\modu_{c_1}$-regularity and $\modu_{c_2}$-regularity have $\modu_{c_1+c_2}$-regularity; see Propositions~\ref{thm:final mult} and \ref{thm:s-mult}. 

We also remark that the significance of condition \eqref{eq:phaseinv} is that if $u$ has $\modu_1$-regularity, then $N[u]$ also has $\modu_1$-regularity. This is crucial for the purposes of defining a contraction map $\Phi$ in Section~\ref{sec:results}.

We note that, for $c \neq 0$, there is an alternative generating set 
      
%      
%       and $\SD=\SD_1$. Recall that $D_t$ and
%      $2tD_t + z \cdot D_z$ lie in $\Psip 2*$. Thus $\modu_c^{(1)}$ is the
%      $\Psi^{0,0}_{\mathrm{par}}$-module generated by the $D_{z_i}$, the
%    $\mathrm{Rot}_{ij}$, the $t D_{z_j} - cz_j / 2$, and the identity.
%
%    Note that, for $c \neq 0$, if we consider the generating set
\begin{equation}
  \label{eq:other gen set}
	\tilde{\mathbf{G}}_{c}:=\{D_{z_j},\ D_t,\ \mathrm{Rot_{ij}},\
        tD_{z_j}-cz_j/2,\ t(c D_t+ D_z \cdot D_z),\ \Id\}.
\end{equation}
In fact, the identity 
$$
t(c D_t + D_z \cdot D_z) - D_z \cdot (t D_z - cz/2)
= (c/2)(2 t D_t - z \cdot D_z)
$$
shows that we can switch between these two that, if $\tilde \modu_c$ denotes the
module for  $\tilde{\mathbf{G}}_c$, then
$$
\tilde \modu_c^{(k)} = \modu_c^{(k)}.
$$
This holds despite the fact that $\modu_c \neq \tilde \modu_c$, and despite the fact that $\tilde{\mathbf{G}}_{c}$ is \emph{not} an admissible set of generators.  Thus
the modules generated by $\tilde{\mathbf{G}}_c$ and $\mathbf{G}_c$
are equivalent from the standpoint of module regularity spaces.

%A simple corollary of this lemma clarifies the dependence of
%module regularity on the choice of generating set.  To state this, for
%$\mathbf{A}, \mathbf{A}' \subset \Psi_{\mathrm{par}}^{\mathbb{N},
%  \mathbb{N}}$ finite subsets, we say $\mathbf{A}$ and
%$\mathbf{A}'$ are \textbf{algebraically equivalent} if each $A \in
%\mathbf{A} \cap \Psi^{m,l}_{\mathrm{par}}$ can be written as a finite sum
%$$
%A = \sum_{\alpha} Q_{\alpha} B^\alpha
%$$
%for $B^\alpha = B_1^{\alpha_1} \cdots B_j^{\alpha_j}$, with the
%$B_i \in \mathbf{A}' \cap \Psi_{\mathrm{par}}^{m_i, l_i}$ with
%$\sum m_i \le m, \sum l_i \le l$ and
%$Q_\alpha \in \Psi^{0,0}_{\mathrm{par}}$, and vice versa for each
%$B \in \mathbf{A}'$.
%\begin{cor}
%  Let $\mathbf{A}, \mathbf{A}' \subset
%  \Psi_{\mathrm{par}}^{\mathbb{N}, \mathbb{N}}$ be algebraically
%  equivalent finite sets generating modules $\mathcal{M}$ and
%  $\mathcal{M}'$.  Then $\mathcal{M}^{(k)} = (\mathcal{M}')^{(k)}$ for
%  each $k \in \mathbb{N}$ and $\mathbf{A}$ and $\mathcal{M}$ satisfy \eqref{eq:module
%    properties} if and only if $\mathbf{A}'$ and $\mathcal{M}'$ do. \end{cor}
%We leave the straightforward proof to the reader.
%\ah{I deleted a corollary which seems incorrect, about algebraically equivalent generating sets}

%\subsection{Comparison between the modules $\modu$ and $\mathcal{N}$ from \cite{TDSL}.}

\vskip 8pt 

In \cite{TDSL}, we defined the module $\mathcal{N}$ as the module of all pseudodifferential operators of order $(1,1)$ characteristic on both radial sets $\SR_\pm$. This module is generated by the operators 
\begin{equation}
\Id, \quad \Rot_{ij}, \quad t D_{z_i} - z_i/2, \quad D_{z_j}, \quad E_{-1} D_t, \quad \ang{Z} E_{-1} P, 
\label{eq:Ngen}\end{equation}
where $\ang{Z} = (1 + |z|^2 + t^2)^{-1/2}$, and $E_{-1} = \Op((1 + |\zeta|^2 + \tau^2)^{-1/4})$ is an elliptic element of order $(-1, 0)$. Our present module $\modu$ is a replacement for $\mathcal{N}$: it is also generated by operators characteristic to $\SR_\pm$ (viewed as operators of order $(*, 1)$; notice that $D_z$ and $D_t$ have spacetime order $0$, therefore their symbol at $\SR$ trivially vanishes when viewed as a symbol of order $(*,1)$). 
However, for $\modu$, we have chosen generators that are all differential operators of order $1$, so that a Leibniz rule is satisfied when the generators are applied to products; this is extremely convenient for deducing algebra properties as in the following section. 

Although the concepts of module regularity with respect to the module $\modu$ is distinct from that of module regularity with respect to $\SN$, for solutions $u$ of $Pu = 0$ there is in fact no difference:

\begin{lemma}\label{thm:mod equiv in ker}
For any $s, \mathrm{r}$, $H_\mathcal{N}^{s, \mathrm{r};k} \cap \ker P = 
H_\SD^{s, \mathrm{r};k} \cap \ker P$. 
\end{lemma}
\begin{proof}
  Assume that $Pu = 0$. Membership $u \in H_\mathcal{N}^{s, \mathrm{r};k}$ is
  equivalent to membership in $H_{par}^{s, \mathrm{r}}$ after
  the application of $k$ or fewer generators of $\SN$, which are listed in \eqref{eq:Ngen}. 
  We can order them so that any powers of the generator $\ang{Z} E_{-1} P$, if present, occur on the right. Of course, this generator annihilates $u$, so for $u \in \ker P$, we only need consider products of generators of the form 
 \begin{equation}
 R^j (tD_z - z/2)^\beta D_z^\alpha (E_{-1} D_t)^m u, \quad |\alpha| + |\beta| + j +m \le k,
 \label{eq:Nfactors}\end{equation}
 where $R^j$ is a product of $j$ rotations.
 
 On the other hand, membership of $u \in H_{\modu}^{s, \mathrm{r};k}$ is
  equivalent to membership in $H_{par}^{s, \mathrm{r}}$ after the application of products of generators of $\modu$, with total order at most $(k,k)$. Following the discussion in the previous subsection, we can take the following generators of $\modu$:
  \begin{equation}
\Id, \quad  \Rot_{ij}, \quad t D_{z_i} - z_i/2, \quad D_{z_j}, \quad D_t, \quad  tP, 
\end{equation}
(we may replace $D_t + D_z \cdot D_z$ with $P$ as these are equal in a neighbourhood of spacetime infinity by assumption). So the difference is just in the final two generators. We may arrange these generators in any order, so we put the factors of $tP$, if present, on the right, and then they annihilate $u$. So we only need to consider products of operators of the form 
 \begin{equation}
 R^j (tD_z - z/2)^\beta D_z^\alpha D_t^{m'} u, \quad |\alpha| + |\beta| + j + 2m' \le k. 
 \label{eq:Dfactors}\end{equation}
 The only difference between \eqref{eq:Nfactors} and \eqref{eq:Dfactors} is in the final factor: we have $m$ powers of $(E_{-1} D_t)$ in the former case, each of order $(1, 0)$ and $m'$ powers of $D_t$ in the latter case, each of order $(2, 0)$. 
 
 First assume that $u \in H_\mathcal{N}^{s, \mathrm{r};k}$. We microlocally decompose $u = u_1 + u_2$, where $u_1$ is microsupported near $\Char(P)$ and $u_2$ is microsupported away from $\Char(P)$. Then from $Pu = 0$ we see that $u_2$ is Schwartz, and therefore is contained in both module regularity spaces. Therefore it remains to consider $u_1$. We will exploit the fact that, on the microsupport of $u_1$, on any sufficiently small microlocal neighbourhood, either $\Id$ or one of the $D_{z_i}$ are elliptic (as an operator of order $(1, 0)$). 
 
 To show that $u_1 \in H_{\modu}^{s, \mathrm{r};k}$ we need to \eqref{eq:Dfactors} is in $H_{par}^{s, \mathrm{r}}$. Fix $\alpha$, $\beta$, $j$ and $m'$ as in \eqref{eq:Dfactors}. Choose any multi-index $\alpha'$ with $|\alpha'| = m'$. Then, according to \eqref{eq:Nfactors}, we have 
$$
R^j (tD_z - z/2)^\beta D_z^{\alpha + \alpha'} (E_{-1} D_t)^{m'} u \in H_{par}^{s, \mathrm{r}}.
$$
In any microlocal neighbourhood, we can choose $\alpha'$ so that $D_z^{\alpha'} E_{-1}^{m'}$ is elliptic of order $(0,0)$, and therefore  microlocally invertible. We can therefore remove this factor and deduce \eqref{eq:Dfactors} is in $H_{par}^{s, \mathrm{r}}$, as required. 

The argument can be reversed to show the reverse inclusion. 
 
%  operators characteristic on
%  $\mathcal{R}$.  Thus
%  $D_t^l D_z^\alpha (tD_z - z/2)^\beta R^j u \in H^{s,
%    \mathrm{r}}$ (where $R^j$ is a product of $j$ rotations)
%  provided $l + |\alpha| + |\beta| + j \le k$.  Since $P u = 0$ it is
%  automatic that
%  $D_t^l D_z^\alpha (tD_z - z/2)^\beta R^j P^q u \in H^{s,
%    \mathrm{r}}$ (since it is zero), so, recalling that the
%  order of application of module generators does not matter (Lemmata
%  ~\ref{thm:order doesnt matter} and \ref{thm:order really doesnt matter}.)
%  $u \in H_\mathcal{N}^{s, \mathrm{r};k} \cap \ker P \implies u \in
%  H_\SD^{s, \mathrm{r};k}$.
\end{proof}

\section{Algebra properties}
\label{sec:mult}

\subsection{Time cut-offs and reduction to the ``flat''
  module} \label{sec:flat module}

When proving algebra properties for the module regularity spaces
$H^{s,r;k}_{\mathrm{par}, \modu_c}$, we will conjugate the generators
of $\modu_c$ by the oscillatory factor $e^{ic|z|^2/4t}$ and then
appeal to standard Sobolev algebra properties. The singularity of the
oscillatory factor at $t=0$ motivates the introduction of subspaces of
the module regularity spaces supported on either side of a spatial
slice $\{t=t^*\}$.

We fix a time $t_0$ and let
\begin{equation}
H_{\modu_c}^{0, l; k}(t \geq t_0) = \left\{ u \in
  H_{\modu_c}^{0, l; k}:  \supp (u)
  \subset \{ t \geq t_0 \} \right\}.
\label{eq:D-support}\end{equation}
Define $H_{\modu_c}^{0, l; k}(t \leq t_0)$ analogously. 

We can cut off functions in $H_{\modu_c}^{0,l;k}$ to the half-space $\{t \geq t_0\}$ by multiplying by a smooth cutoff function $\chi_{ \geq t_0}$ with the following properties.
\begin{equation}
  \label{eq:chi time}
  \chi_{\geq t_0} \in C^\infty(\mathbb{R}), \quad \chi(t) = 0 \mbox{ for } t \leq
  t_0 , \ \chi(t) = 1 \mbox{ for } t \ge
  t_0 + 1.
\end{equation}
We similarly can cut off to $\{t \leq t_0\}$ with $\chi_{\leq t_0}$ supported in $t \le t_0$ and
identically $1$ for $t \le t_0 - 1$.  Below we will typically use
$\chi_{\leq t_0} = 1 - \chi_{\geq t_0 -1}$.
\begin{lemma}
Denoting multiplication operators by $m_f(u) =f u$ for
  functions $f$, 
  \begin{equation}
    \label{eq:cont of cutoffs}
    m_{\chi_{\geq t_0}} \colon H_{\modu_c}^{0
      , l; k} \lra
    H_{\modu_c}^{0, l; k}(t \geq t_0), \quad 
    m_{\chi_{\leq t_0}} \colon H_{\modu_c}^{0, l; k} \lra
    H_{\modu_c}^{0, l; k}(t \leq t_0)
  \end{equation}
  are continuous.
\end{lemma}
\begin{proof}
Since they are bounded functions, multiplication by either
$\chi$ or $1 - \chi$ preserves $L^2(dtdz)$.  The spacetime weight
$\langle z, t \rangle$ commutes
with $m_\chi$ and $m_{1 - \chi}$ so the weighted $L^2$ spaces are also
preserved.  

% {\color{blue}When $s = 2p$ is an even integer, the parabolic regularity is determined by
%   application of $D_t$ and the $D_{z_i}$.  The $D_{z_i}$ commute
% with $\chi(t)$.  For $D_t$, we have $D_t(\chi(t) u) = -i (\partial_t
% \chi) u + \chi D_t u$, which lies in $\mathcal{X}_\pm^{s - 2,  l;
%   \kappa, k}$, say by induction on $p$, since $\chi'$ is bounded.}

For the module regularity, the generator
  $2 tD_t + z\cdot D_z$ for the module satisfies
$$
(2 tD_t + rD_r) \chi(t) u = - 2 i t \chi'(t) u + \chi(t) (2 tD_t +
z\cdot D_z) u
$$
and the right hand side is in $L^2$ because $t \chi'(t)$ is bounded
($\chi'$ is supported in the bounded interval $[t_0 - 1, t_0+1]$.)  Arguing identically for
the other (differential) generators of the module, we see that module
regularity of $u$ implies module regularity of $\chi(t) u$.  
\end{proof}

We now ``remove the oscillation'' $e^{i|z|^2/4t}$ appearing in \eqref{eq:fminus} and \eqref{eq:fplus} from functions in the module regularity spaces. However, we wish to avoid the singularity at $t=0$ of this oscillation. To do this, we can take advantage of the invariance of our class of Hamiltonians under time translation, to see that one could replace the factor $e^{i|z|^2/4t}$ by any time translation $e^{i|z|^2/4(t-t_0)}$ of it. Indeed, this would just have the effect of replacing the incoming and outgoing data, $f(\zeta)$ and $\tilde f(\zeta)$, in these expressions\footnote{Multiplication by the factor $e^{it_0 |\zeta|^2}$ is an isomorphism of the spaces $\Hdata^k$, so time translations just gives rise to isomorphisms of the spaces of incoming and outgoing data. We remark in passing that this property is not enjoyed by standard Sobolev spaces, so this is one good reason for using the $\Hdata^k$ scale of spaces for incoming and outgoing data.} by $f(\zeta) e^{it_0 |\zeta|^2}$ and $\tilde f(\zeta) e^{it_0 |\zeta|^2}$.

Thus, after picking a time $t_0$, we will decompose $u \in H^{s,r;k}_{\mathrm{par}, \modu_c}$ using $u = u_> + u_< := \chi_{\geq t_0} u + \chi_{\leq t_0+1} u$ and then divide by the factor $e^{i|z|^2/4(t+1-t_0)}$ for the first term, and $e^{i|z|^2/4(t-2-t_0)}$ for the second term. This avoids the singularity in both cases. 

We aim to show that the resulting functions $w_> := e^{-i|z|^2/4(t+1-t_0)} u_>$ and $w_< := e^{-i|z|^2/4(t-2-t_0)} u_<$ belong to some conjugated ``module regularity space''. One slight technical issue, however, is that the $\chi$ cutoff functions are not smooth on the compactified spacetime $\overline{\RR^{n+1}}$; they have a conic singularity at the equator $E$. To remedy this, we consider the blowup 
\begin{equation}
X := [\overline{\RR^{n+1}}; E]
\label{eq:Xdefn}\end{equation}

\begin{center}
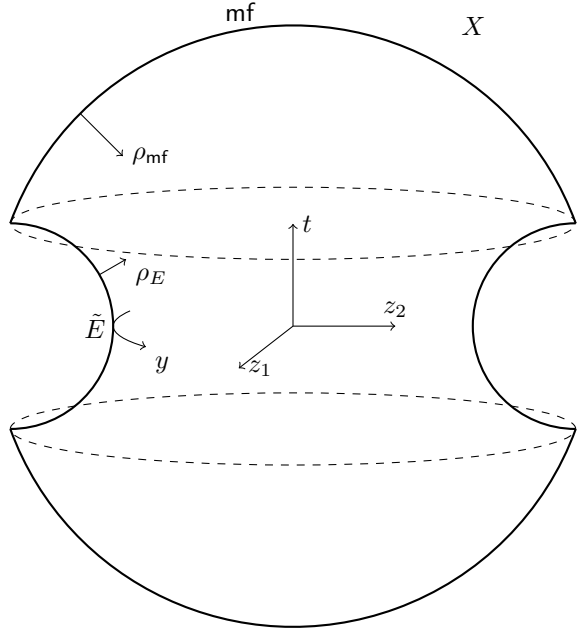
\begin{figure}
	\begin{tikzpicture}[scale=0.8]
		\draw[thick] (4.698,1.71) arc (20:160:5cm);
		\draw[thick] (-4.698,-1.71) arc (200:340:5cm);
		\draw[dashed] (0,1.71) ellipse (4.698cm and 0.6cm);
		\draw[dashed] (0,-1.71) ellipse (4.698cm and 0.6cm);
		\draw[thick] (4.698,1.71) arc (90:270:1.71cm);
		\draw[thick] (-4.698,-1.71) arc (-90:90:1.71cm);
		%	\draw[dotted] (5,0) arc (0:180:5cm and 1cm);
		%	\draw (-5,0) arc (180:360:5cm and 1cm);
		\draw[->] (-3.536,3.536)--(-3.536*0.8,3.536*0.8) node[anchor=west]  {$\rho_{\mathsf{mf}}$};
		\draw (-0.868,4.924) node[anchor=south] {$\mathsf{mf}$};
		\draw (-3.29,0) node {$\tilde{E}$};
		\draw[->] (-4.698+1.48,0.855)--(-4.698+1.48*1.3,0.855*1.3) node[anchor=north west] {$\rho_{E}$};
		\draw (-2.988,0) arc (180:155:3cm and 0.6cm);
		\draw[->] (-2.988,0) arc (180:215:3cm and 0.6cm) node[anchor=north west] {$y$};
		\draw (3,5) node {$X$};
		\draw[->] (6-6,-1.5+1.5)--(6-6,0.2+1.5) node[anchor=west] {$t$};
		\draw[->] (6-6,-1.5+1.5)--(5.1-6,-2.2+1.5) node[anchor=west] {$z_1$};
		\draw[->] (6-6,-1.5+1.5)--(7.7-6,-1.5+1.5) node[anchor=south] {$z_2$};
	\end{tikzpicture}
\caption{The blowup space $X$. It has two boundary hypersurfaces $\tilde{E}$ and $\mathsf{mf}$ (which itself has two components). A boundary defining function for $\tilde{E}$ is $\rho_E=\ang{t}/\ang{z}$, while a boundary defining function for $\mathsf{mf}$ is $\ang{t}^{-1}$.}
\label{fig:blowup}
\end{figure}
\end{center}
using Melrose's notation for the blowup of a manifold with corners at a $p$-submanifold. This blowup is illustrated in Figure \ref{fig:blowup}. Indeed, the submanifold $E$ of $\overline{\RR^{n+1}}$ is given in local coordinates $(s, \rhobase, y)$, where $s = t/\ang{Z}$ and $\rhobase = \ang{Z}^{-1}$, while $y = (y_1, \dots, y_{n-1})$ are local coordinates on the sphere $S^{n-1}_{\hat z}$, by 
$$
E = \{ s = 0, \rhobase = 0 \}.
$$

In the blowup space $X$, each point of $E$ has, by definition, been replaced by its inward-pointing spherical normal bundle, which in this case is a line segment. The quotient $t = s/\rhobase$ is a coordinate on the interior of each line segment. That means that we can regard each preimage of a point in $E$ as a copy of the (compactified) time axis. It follows that functions such as $\chi_{\geq t_0}$ and $\chi_{\leq t_0}$, which are smooth functions of $t$, and constant for $t$ sufficiently large or sufficiently small, lift to smooth functions on $X$. 

We will denote the preimage of $E$ in $X$ by $\tilde E$; as just mentioned, $\tilde E$ has the structure of a product, $\tilde E \equiv E \times \overline{\RR}$ where the factor $\overline{\RR}$ may be viewed as the compactified time axis. 
The rest of spacetime infinity lifts to a second boundary hypersurface that we shall denote $\mf$, standing for `main face'. 
We will denote a generic boundary defining function for $\tilde E$ by $\rho_E$ (even though $\rho_{\tilde E}$ might be more accurate) and a generic boundary defining function for $\mf$ by $\rho_{\mf}$. 

Now, returning to $w_>$ and $w_<$, we will show that these lie in some conjugated ``module regularity space''; however, this term will have a slightly different meaning in this section compared to the previous section, which is the reason for using quotes. In this section, the term `test module' will mean a module of \emph{differential} operators over $C^\infty(X)$, where $X$ is as in \eqref{eq:Xdefn}, which is closed under commutators, and generated over $C^\infty(X)$  by a finite number of first order differential operators including the identity operator.\footnote{We avoid using pseudodifferential operators here because there are additional technicalities with defining them on manifold with corners such as $X$. The fact that our test modules are modules over $C^\infty(X)$, as opposed to an algebra of zeroth order pseudodifferential operators, means that our ability to localize is reduced; however, in this section localization using smooth functions on $X$ is sufficient for our purposes.} In particular, we define the specific ``flat'' test module $\SF$ by listing generators: 

\begin{equation}
  \label{eq:7}
\mathcal{F}  =  C^\infty(X)\text{-span of } \left\{\ang{t} D_{z_j},  \quad \frac{\ang{z}}{ \ang{t}} , \quad \ang{t} D_t, \quad \Rot_{jk} \right\}.
\end{equation}
This generating set $\mathbf{A} = A_1, \dots, A_N$ satisfies an admissibility condition analogous to Definition~\ref{def:testmod}. Namely, we have
\begin{itemize}
\item $\mathbf{A}$ is closed under commutators in the sense that, for all $A_i, A_j$ in $\mathbf{A}$,  there exist operators $C_{ijk} \in C^\infty(X)$ such that 
\begin{equation}
[A_i, A_j] = \sum_k C_{ijk} A_k ,
\label{eq:diffcommgen}\end{equation}
\item  For all  $A \in \mathbf{A}$  we have 
\begin{equation}
[A, C^\infty(X)] \subset C^\infty(X).
\label{eq:diff module properties}
\end{equation}
\end{itemize}

We also define an order function on our module. To respect the parabolic scaling, we define the order of $\ang{t} D_t$ to be $2$, and the order of the remaining generators to be $1$. Of course, elements of $C^\infty(X)$ have order zero. 

We can equivalently write generators in terms of the boundary defining functions of $X$, which shows how $\SF$ relates to the geometry of $X$. For this we choose the specific boundary defining function\footnote{The structure below is not invariant under all diffeomorphisms of $X$; thus, we cannot take an arbitrary boundary defining function here. The question of the precise invariance properties of this module is not relevant to the present article and will considered elsewhere.} $\rho_E = \ang{t} \ang{z}^{-1}$, and we take local coordinates $y = (y_1, \dots, y_{n-1})$ on a coordinate patch of the sphere $S^{n-1}_{\hat z}$. Then, 
\begin{equation}
  \label{eq:77}
\mathcal{F}  =  C^\infty(X)\text{-span of } \left\{\rho_E^2 D_{\rho_E},  \quad \frac{1}{\rho_E} , \quad \rho_{\mf} D_{\rho_{\mf}}, \quad D_{y_j} \right\}
\end{equation}
where all these generators have order $1$ except for $\rho_{\mf} D_{\rho_{\mf}}$ which has order $2$. 

We extend the order function to powers of the module, $\SF^k$, in the natural way. That is, if $A \in \SF^k$ then the order of $A$, $\ord(A)$,  is 
the minimum $j$ such that $A$ can be written as a finite sum of products of module elements, with $C^\infty(X)$ coefficients, such that in each term, the sum of the orders of the factors is at most $j$. 

As in the previous section, the module regularity spaces are defined in terms of \emph{reduced} powers. Thus, we define the $k$th reduced power $\SF^{(k)}$ by 
\begin{equation}
\SF^{(k)} = \{ A \in \SF^k \mid \ord(A) \leq k \}
\label{eq:F-red}\end{equation}

%and \emph{define} an order function on $\cup_{m \in \mathbb{N}}\mathcal{F}(t>t_0)^m$ by analogy with the parabolic regularity orders
%from $\mathcal{N}$:
%\begin{equation*}
%  \begin{gathered}
%    \ord((t - t_0 + 1) D_{z_j}) = \ord(\Rot_{jk}) = \ord(\frac{z_j}{
%      (t - t_0 + 1)}) = 1, \\ \ord((t - t_0 + 1)D_t) = 2, \\ \ord(f)=0\textrm { for }f\in C^\infty_\infty.
%  \end{gathered}
%  \end{equation*}
%More generally, for $A_1,\ldots,A_q\in \SF(t>t_0)$, we set
%    $$\ord(A_1, \dots,  A_q) = \sum_{j =
%  1}^q \ord(A_j).
%    $$
%    Focusing on $ e^{-i c|z|^2 / 4 (t - t_0 + 1)}$, we define
% \begin{equation}
%   \label{eq:flat module}
%   \mathbf{G}_{\mathcal{F}, > t_0}^1 = \left\{D_{z_i} , (t - t_0 + 1) D_{z_i}, m_{c|z| / (t
%       - t_0 + 1)}\right\} , \mathbf{G}_{\mathcal{F}, > t_0}^2 = \left\{ D_t
%     , \langle z, t \rangle (2 (t - t_0 + 1)D_t + z \cdot D_z) \right\},
% \end{equation}
%so is the multiplication operator of the function $c|z|/(t - t_0 + 1)$.
% The notation $ \mathbf{G}_{\mathcal{F}, > t_0}^1$ and
% $\mathbf{G}_{\mathcal{F}, > t_0}^2$ coming from the fact that we continue to
% consider the operators
% $D_t , \langle z, t \rangle (2 tD_t + z \cdot D_z)$ as second order
% and the other operators as first order.  

The module regularity spaces corresponding to the differential module $\SF$ are, for $s=0$ and $k \in \NN$, 
\begin{equation}
  \begin{gathered}\label{eq:flat mod reg space}
    H^{0, r; k}_{\SF} = \left\{ u \in
      \langle z, t \rangle^{-r} L^2(\mathbb{R}^{n + 1}) \mid A u \in \langle z, t \rangle^{-r} L^2(\mathbb{R}^{n + 1})
      \ \forall A \in \SF^{(k)}  \right\}
  \end{gathered}
\end{equation}
with norm 
\begin{equation}
\label{eq:diff.mod.reg.sob.norm}
	\|u\|_{H_\SF^{0,r;k}}^2:= \sum_{\mathbf{A}^\alpha \in \SF^{(k)}} \|\mathbf{A}^\alpha  u\|_{H_\mathrm{par}^{0,r}}^2.
\end{equation}
That this is a suitable norm follows from the admissibility conditions in a similar way to the previous section.

Similarly to \eqref{eq:D-support} we also define spaces with support on one side of a time slice:
\begin{equation}
H^{0, r; k}_{\SF}(t \geq t_0) = \left\{ u \in
H^{0, r; k}_{\SF}  \mid  \supp (u)
  \subset \{ t \geq t_0 \} \right\}.
\label{eq:D-supportF}\end{equation}

%Replacing $t \ge t_0 + \epsilon$ by $t \le t_0 - \epsilon$ we obtain
%an analogous module
%$$
%\mathcal{F}(t < t_0) =  \left\{(t - t_0 - 1) D_{z_j},  \frac{z_j}{ (t - t_0 -
%    1)} , (t - t_0 - 1)D_t, \Rot_{jk} \right\}_{ C^\infty_{\infty}}.
%$$
%and the corresponding space $H^{0, r;k}_{\mathcal{F}(t < t_0)}(t < t_0)$. 
%We will typically write these spaces as $H^{0, r;k}_{\mathcal{F}}(t > t_0)$ and $H^{0, r;k}_{\mathcal{F}}(t < t_0)$ to avoid
%redundant notation.

% \begin{remark}
%   Although we do not use this anywhere directly, we note that for any
%   of the generators $A$ of $\mathcal{F}$ given in the definition
%   above, that $[A, \mathcal{F}] \subset \mathcal{F}$.  As with the
%   module $\mathcal{N}_c$ this allows us to conclude that...
% \end{remark}

\begin{proposition}\label{thm:flat back and forth}
  Let $c \in \mathbb{R} \setminus 0$, $r \in \mathbb{R}$ and $k \in
  \mathbb{N}_0$. Let $t_ 0 \in \RR$, $t_* < t_0 < t^*$.  In the notation introduced above, the multiplication operators
  \begin{equation}
    \label{eq:flat map up}
e^{- i c|z|^2/4(t - t_*)}  \colon H_{\modu_c}^{0, r; k}(t \geq t_0)
\lra H_{\mathcal{F}}^{0,r;k}(t \geq t_0)
\end{equation}
and
  \begin{equation}
    \label{eq:flat map down}
e^{- i c|z|^2/4(t - t^*)}  \colon H_{\modu_c}^{0, r; k}(t \leq t_0)
\lra H_{\mathcal{F}}^{0,r;k}(t \leq t_0)
\end{equation}
are continuous and invertible.  
%with inverse given by multiplication by
%$\exp( i c|z|^2/4(t - t_0 + 1))$ in the first instance and $\exp( i
%c|z|^2/4(t - t_0 - 1))$ in the second.\ah{Stating the obvious?}
\end{proposition}
\begin{proof}
We only prove \eqref{eq:flat map up} as proving the other statement is exactly analogous. 
%Moreover, the support condition is obviously unaffected by the multiplication operator, so we only consider the module regularity spaces without the support condition. 

First, we note that the module $\modu_c$ has an
  equivalent set of generators
  $$
D_{z_j},  \quad D_t,  \quad (t - t_*)D_{z_j}-cz_j/2,  \quad 2(t - t_*)D_t+ z \cdot D_z, \quad \Rot_{jk} .
  $$
Indeed, this set is given by a $\RR$-linear change of
basis with respect to $\mathbf{G}_c$.  Then we compute, with
$\Phi := \Phi_{c, t_*}(z,
t) = e^{ic|z|^2/4(t - t_*)}$
\begin{equation}
  \begin{gathered}
\ [ D_{z_j}, \Phi ] = \frac{cz_j}{2(t - t_*)}, \quad
\ [ D_t,  \Phi ] = -\frac{c|z|^2}{4(t - t_*)^2} \\ 
\ [ (t - t_*)D_{z_j}-cz_j/2, \Phi ] = [ 2(t - t_*)D_t+ z
\cdot D_z , \Phi ] = [\Rot_{jk}, \Phi] = 0
\label{eq:commutators}
\end{gathered}
\end{equation}

This proves directly that $u \in H_{\modu_c}^{0, r; k}$
if and only if $w := e^{- i c |z|^2/(t - t_*)} u$ lies in a module
regularity space with differential module
\begin{multline*}
    \mathcal{F}' = C^\infty(X)\text{-span of }
    \Big\{ D_{z_j} + \frac{cz_j}{2(t - t_*)} , \quad  (t -
      t_*)D_{z_j}, \quad  D_t - \frac{c|z|^2}{4(t -t_*)^2} , \\ 
        2(t - t_*)D_t+ z \cdot D_z, \quad \Rot_{jk}  \Big\}
\end{multline*}
where the orders of generators are $2$ when a time derivative $D_t$ is present and $1$ otherwise. 
%meaning, with order one given to $z$-derivatives and order two to
%$t$-derivatives, that
%$$
%B_1 \cdots B_q w  \in \langle z, t \rangle^{-r} L^2, \ord(B_1, \dots,
%B_q) \le k.
%$$

Now, acting on functions $v$ supported where $t \geq t_0$, we can change to a simpler basis of generators, as follows. First, $t-t_*$ is bounded below on the support of $v$, so we can express $D_z$ as a $C^\infty(X)$-multiple of $(t - t_*) D_z$ on the support of $v$, allowing us to remove the $D_z$ term in the first generator. This reduces to generators
$$
 \Big\{ \frac{cz_j}{2(t - t_*)} , \quad  (t -
      t_*)D_{z_j}, \quad  D_t - \frac{c|z|^2}{4(t -t_*)^2} , \quad  
        2(t - t_*)D_t+ z \cdot D_z, \quad \Rot_{jk}  \Big\}.
$$ 
Next, a combination of the first two generators can eliminate the $z \cdot D_z$ in the fourth generator. Having done this, we can use the $(t - t_*) D_t$ generator, together with the square of the first generator,  to eliminate the third generator altogether. This reduces the set of generators to 
$$
 \Big\{ \frac{cz_j}{2(t - t_*)} , \quad  (t -t_*)D_{z_j},  \quad  (t - t_*)D_t, \quad \Rot_{jk}  \Big\}.
$$ 
We finally note that $z_j$ is a $C^\infty(X)$-multiple of $\ang{z}$, and $t - t_*$ is a $C^\infty(X)$-multiple of $\ang{t}$, and we arrive at the generators for $\SF$ as listed in \eqref{eq:7}. 

We have thus shown that if $v \in H_{\modu_c}^{0, r; k}(t \geq t_0)$, then $e^{- i c|z|^2/4(t - t_*)} v$ is in $H_{\mathcal{F}}^{0,r;k}(t \geq t_0)$. The steps can be reversed, showing that if $w \in H_{\mathcal{F}}^{0,r;k}(t \geq t_0)$, then $e^{i c|z|^2/4(t - t_*)} w$ is in $H_{\modu_c}^{0, r; k}(t \geq t_0)$. 
%
%
%
%But since the function $t - t_0 + 1$ is bound below on the support of
%$w$, $(t - t_0 + 1)D_{z_j} \in \langle z, t \rangle^{-r} L^2 \implies
%D_{z_j} \in \langle z, t \rangle^{-r} L^2$, so the module
%$\mathcal{F}'$ can be replaced by the module with generators
%$$
%\frac{cz_j}{2(t - t_0 + 1)} , \   D_t +
%  \frac{c|z|^2}{4(t - t_0 + 1)^2} , \  (t - t_0 + 1)D_{z_j}, 
% 2(t - t_0 + 1)D_t+ z
%\cdot D_z 
%$$
%and the $z$-rotations $\Rot_{jk}$, acting on functions with support in $t > t_0$, 
%with multplication by the function being order $1$.  But then 
%$\frac{c|z|^2}{4(t - t_0 + 1)^2}$ is order two and $z \cdot D_z$ is
%order two as well (it is a sum of $z_j / (t - t_0 + 1)$ times $(t -
%t_0 + 1) D_{z_j}$) so we see that $H_{\mathcal{F}'}^{0,r;k}(t > t_0) =
%H_{\mathcal{F}}^{0,r;k}(t > t_0)$ and therefore 
%$$
%u \in H_{\modu_c}^{0,r;k}(t > t_0) \implies w \in H_{\mathcal{F}}^{0,r;k}(t > t_0).
%$$
%But this also works in the other direction, i.e.\ $w \in
%H_{\mathcal{F}}^{0,r;k}(t > t_0) \implies w \in
%H_{\mathcal{F}'}^{0,r;k}(t > t_0)$ and again the commutation relations
%give that $u = e^{i c|z|^2 / 4 (t - t_0 + 1)} w$ is in $H_{\modu_c}^{0,r;k}(t > t_0)$, which is what we wanted.
\end{proof}

%\begin{remark}\label{rem:submodule}
%There is a `submodule' $\SF''$ of $\SF$ with generators 
%\begin{equation}
%  \label{eq:777}
% \left\{\rho_E D_{\rho_E},   \quad \rho_{\mf} D_{\rho_{\mf}}, \quad \Rot_{jk} \right\}
%\end{equation}
%where the first generator now has order two, as it is the product of two first order generators $\rho_E^{-1} \circ \rho_E^2
% D_{\rho_E} $ of $\SF$. As this is a submodule, the reduced powers of $\SF''$ are submodules of the reduced powers of $\SF$ (and we remark that it is easy to see that the inclusion is proper). It follows that the module regularity spaces for $\SF$ and $\SF''$ satisfy 
%\begin{equation}\begin{aligned}
%H_{\mathcal{F}}^{0,r;k}(t \geq t_0) &\hookrightarrow H_{\mathcal{F}''}^{0,r;k}(t \geq t_0), \\
%H_{\mathcal{F}}^{0,r;k}(t \leq t_0) &\hookrightarrow H_{\mathcal{F}''}^{0,r;k}(t \leq t_0), \\
%H_{\mathcal{F}}^{0,r;k} &\hookrightarrow H_{\mathcal{F}''}^{0,r;k}.
%\end{aligned}\label{eq:embeddings}\end{equation}
%We will use these embeddings to prove multiplicative properties in the next subsection, exploiting the fact that regularity with respect to $\SF''$ is easier to work with, as it has a product-type structure. 
%\end{remark}

\begin{remark}
The modules $\SN$ and $\modu$ are intimately connected with the existence of asymptotic expansions \eqref{eq:fminus}, \eqref{eq:fplus}. As mentioned in the introduction, the functions with infinite module regularity with respect to $\SN$, or equivalently, infinite module regularity with respect to $\modu$, are the `Lagrangian distributions' (termed Legendre distributions in \cite{MZ96}) associated to the Lagrangian submanifold that is the conic extension (in the sense of spacetime dilations) of the radial set $\SR$. To illustrate the point, we mention the following result:
\end{remark}

\begin{proposition}
Assume that $Pu = 0$, and that $u \in H_{\SN}^{1/2, \rmin; k}$ for all $k$, or, equivalently in view of Lemma~\ref{thm:mod equiv in ker}, that 
$u \in H_{\modu}^{1/2, \rmin; k}$ for all $k$. Then with $\chi_{\geq t_0}$ and $\chi_{\leq t_0 + 1}$ as above, and choosing $t_* < t_0  < t_0 + 1 < t^*$, we can write 
$$
u \chi_{\geq t_0} =(t-t_*)^{-n/2} e^{i|z|^2/4(t-t_*)} \tilde u,
$$
where $\tilde u$ is a $C^\infty$ function on the blowup space $X$, vanishing to infinite order at $\tilde E$; consequently, $\tilde u$
is actually $C^\infty$ on $\overline{\RR^{n+1}}$, vanishing to infinite order on the `lower half' of the boundary. A similar 
statement can be made for $u \chi_{\leq t_0 + 1} = u - u \chi_{\geq t_0}$, with $t^*$ replacing $t_*$.  In particular, the incoming and outgoing data for such $u$ with infinite module regularity are both Schwartz. 
\end{proposition}

We only provide brief remarks on the proof since we do not need this result in the sequel. Passing to the flat module $\SF$, we see that $\tilde u$ has fixed regularity under repeated applications of the vector fields $D_y, \rho_{\mf} D_{\rho_{\mf}}$ and $D_{\rho_E}$. This is slightly weaker than claimed as we need regularity under repeated application of $D_{\rho_{\mf}}$, not merely $\rho_{\mf} D_{\rho_{\mf}}$. To obtain this improvement we need to exploit the PDE $Pu = 0$. Using Prop. 7.7 of \cite{TDSL}, or Prop.~\ref{prop:limits} of the present paper, we find that $\tilde u$ has a limit $f$ at the boundary, which is in $\Hdata^k$ for all $k$ and therefore Schwartz. Then, given such data, and any large integer $N$, we can form a formal solution as a finite power series in $(t-t_*)$ with $N+1$ terms:
$$
v' = (t-t_*)^{-n/2} e^{i|z|^2/4(t-t_*)} \tilde v
$$
where $\tilde v$ is a $C^\infty$ function with boundary value $f$, and $Pv = O((t-t_*)^{-n/2-N-1})$. We let $v = v' - P_-^{-1}(Pv)$; this is a solution to $Pv = 0$ and it has the same outgoing data as $u$, thus $v=u$. However, using the mapping properties of $P_-^{-1}$ from Theorem~\ref{thm:linear1}, using a variable weight $\mathrm{r}_-$ equal to $N + O(1)$ outside a small neighbourhood of $\SR_-$, we can see that the correction term $P_-^{-1}(Pv)$ vanishes to order $ N + O(1)$ at the northern hemisphere (in reference to Figure~\ref{fig1}), and thus $\tilde u$ itself is $C^{N + O(1)}$ at the northern hemisphere. Since $N$ is arbitrary, $\tilde u$ is smooth at the northern hemisphere, and vanishes to infinite order at the equator.  Of course, this means that it  vanishes to infinite order at the entire southern hemisphere due to the cutoff $\chi_{\geq t_0}$.

\begin{remark}\label{rem:inv} We can observe that the spaces $H_{\mathcal{F}}^{0,r;k}$ and $H_{\mathcal{F}}^{0,r;k}(t \geq t_0)$ are invariant under multiplication by functions of the form
\begin{equation}
e^{i f(t) |z|^2/\ang{t}^2}, \quad f \in S^0_{\mathrm{cl}}(\RR; \RR),
\label{eq:mult-inv}\end{equation}
where $f(t)$ is a real-valued classical symbol of order zero in $t$. To see this, let $u \in H_{\mathcal{F}}^{0,r;k}$ for $k \geq 1$; we apply the generators of $\SF$ in the form \eqref{eq:7} to 
$$
e^{i f(t) |z|^2/\ang{t}^2} u.
$$
For example, applying $\ang{t} D_{z_j}$, we obtain 
$$
e^{i f(t) |z|^2/\ang{t}^2} \Big( \ang{t} D_{z_j} u + \frac{2 f(t) z_j}{\ang{t}} u \Big).
$$
The first term is in $\ang{z,t}^{-r} L^2$ by assumption, and the second term is also since $\ang{z}\ang{t}^{-1}$ is a generator of $\SF$. The other generators, and repeated applications of generators, are analyzed in the same way. Moveover, if $u$ has support in $\{ t \geq t_0 \}$, then we only need $f$ to be a classical symbol of order zero for $t \geq t_0$. 

This observation clarifies why the value of $t_*$ is not important when we multiply by $e^{-i c |z|^2/4(t-t_*)}$ in Proposition~\ref{thm:flat back and forth}. If we take two different values $t_*, t_*'$ of $t$, then we can easily check that 
$$
e^{-i c |z|^2/4(t-t_*)} e^{i c |z|^2/4(t-t_*')}
$$
is of the form \eqref{eq:mult-inv}, at least for $t > \max(t_*, t_*')$. 
\end{remark}

\subsection{Algebra properties of module regularity
  spaces}\label{sec:first big algebra}
In this section we prove the main multiplication properties we will
use to bound the nonlinear terms.  These results are phrased in terms
of Banach space containments and should be interpreted as follows,
given $X, X'$ Banach spaces which are subspaces of
$\mathcal{S}'(\mathbb{R}^{n + 1})$, $X \cdot X \subset X'$ means that
for all $u, v \in X$, $uv$ is a tempered distribution which in fact
lies in $X'$, and the mapping $(u, v) \mapsto uv$ is a continuous
bilinear map $X \times X \lra X'$.

We begin with the basic algebra property of parabolic Sobolev spaces, which is in the same spirit as the results in \cite[Section~4B]{HVsemi}, and indeed is proved by the same Fourier methods.

Writing $x\in \RR^{n+1}$ as $(x',x'')\in \RR\times \RR^{n}$,  we define
\begin{equation}
H_\SQ^{k}(\RR\times\RR^{n})=\{u\in L^2(\RR^{n+1}):D_{x'}^\alpha D_{x''}^\beta u\in L^2(\RR^{n+1})\textrm{ for }2\alpha+|\beta| \leq k  \}. 
\label{eq:HQ-def}\end{equation}
Thus $x'$ is the parabolic direction in this setting. 

\begin{lemma}
\label{lem:basic.sob.alg}
For  $k_1,k_2\in \NN$ such that $k_1+k_2 \geq k> \max(\frac{n+2}{2},2)$, we have
\begin{equation} H_\SQ^{k_1}(\RR\times\RR^{n})\cdot H_\SQ^{k_2}(\RR\times\RR^{n})\subset L^2(\RR^{n+1}) .
\label{eq:Sob-prod}\end{equation}
\end{lemma}

\begin{proof}
Let $u_j\in H_\SQ^{k_j}(\RR\times\RR^{n})$. We introduce the notation
$$\xi=(\xi',\xi'')\in \RR \times \RR^{n}$$
for the dual variables to $x=(x',x'')$, and the weight function $$ w(\xi)=(1+|\xi'|^2+|\xi''|^4)^{1/4}.$$
Then for even $k$, we have
$$u\in H_\SQ^{k}(\RR\times\RR^n) \Leftrightarrow w(\xi)^k\hat{u}\in L^2,$$
and for odd $k$, we have
$$u \in H_\SQ^{k}(\RR \times\RR^n) \Leftrightarrow \ang{\xi''}w(\xi)^{k-1}\hat{u} \in L^2. $$
In particular, we have 
\begin{equation}
\label{eq:fourier.weight}
u\in H_\SQ^k(\RR\times \RR^{n}) \Rightarrow \ang{\xi''}w(\xi)^{k-1}\hat{u} \in L^2
\end{equation}
regardless of parity of $k$. 

Without loss of generality we assume $k_1\geq k_2$.

If $k_2=0$, then $k_1\geq k$ and the inclusion follows from the computation  
\begin{equation}
	\|u_1\|_{L^{\infty}}\leq \|\hat{u}_1\|_{L^1}\leq  \|\ang{\xi''}w(\xi)^{k-1}\hat{u}_1\|_{L_2}\|\ang{\xi''}^{-1}w(\xi)^{1-k}\|_{L^2} 
\end{equation}
The first factor is finite by \eqref{eq:fourier.weight} and the second can be bounded using
$$\int \ang{\xi''}^{-2}w(\xi)^{2(1-k)} \, d\xi \lesssim \int \ang{\xi''}^{2(1-r)(1-k)-2}\ang{\xi'}^{r(1-k)}  \, d\xi $$
for any $r\in[0,1]$. Fixing $r=\frac{1+\ep}{k-1}$ for $\ep>0$ small, the $\xi'$-integral is finite. The $\xi''$-integral is
$$\int \ang{\xi''}^{-2k+2(1+\ep)}\, d\xi''$$
which is finite provided that $ -2k +2(1+\ep) < -n$, and so $k > \frac{n+2}{2}+\ep$. Hence by taking $\ep$ sufficiently small we obtain $u_1\in L^\infty$ and so $u_1u_2 \in L^2$ provided that $k>\max(\frac{n+2}{2},2)$ and $k_2=0$.

If $k_1 \geq k_2\geq 1$ we compute by Cauchy--Schwarz
\begin{equation}\begin{gathered}
	\|\hat{u}_1\ast \hat{u}_2\|_{L^2}^2 \leq\int \left(\int |\hat{u}_1(\xi_1-\xi_2)\hat{u}_2(\xi_2)| \, d\xi_2\right)^2 \, d\xi_1\\
	\leq \int M(\xi_1)\int \ang{\xi_1''-\xi_2''}^2w(\xi_1-\xi_2)^{2k_1-2}|\hat{u}_1(\xi_1-\xi_2)|^2\ang{\xi_2''}^2w(\xi_2)^{2k_2-2}|\hat{u}_2(\xi_2)|^2 \, d\xi_2 \, d\xi_1\\
	\lesssim \sup_{\xi_1} M(\xi_1) \label{eq:sup.bound}
\end{gathered}\end{equation}
where
\begin{multline}
	M(\xi_1)= \int \frac{ d\xi_2}{\ang{\xi_1''-\xi_2''}^2\ang{\xi_2''}^2 w(\xi_1-\xi_2)^{2k_1-2}w(\xi_2)^{2k_2-2}}\\
	\lesssim \int \int \frac{d\xi_2'\, d\xi_2''}{\ang{\xi_1''-\xi_2''}^{2k_1-2r(k_1-1)}\ang{\xi_2''}^{2k_2-2r(k_2-1)}\ang{\xi_1'-\xi_2'}^{r(k_1-1)}\ang{\xi_2'}^{r(k_2-1)}}
\end{multline}
for any $r\in[0,1]$. 

By H\"{o}lder's inequality, both the $\xi_2'$ and $\xi_2''$ integrals are finite and uniformly bounded in $\xi_1$ provided that
\begin{equation}
\label{eq:sobalg.cond.1}
r(k_1+k_2-2)>1
\end{equation}
and
\begin{equation}
\label{eq:sobalg.cond.2}
2(k_1+k_2)-2r(k_1+k_2-2) >n .
\end{equation}
For $k> 3$ and small $\ep>0$, we can thus choose $r\in (0,1)$ so that $r(k_1+k_2-2)=1+\ep$, and so \eqref{eq:sobalg.cond.1} is satisfied. Moreover if $k>\frac{n+2}{2}$ we then have
$$2(k_1+k_2)-2r(k_1+k_2-2)=2(k_1+k_2)-2(1+\ep) > n $$
for sufficiently small $\ep$ and so \eqref{eq:sobalg.cond.2} is also satisfied. Hence $\hat{u}_1\ast \hat{u}_2\in L^2$ from \eqref{eq:sup.bound} and Plancherel implies $u_1u_2\in L^2$ provided that $k>\max(\frac{n+2}{2},3)$ and $k_2\geq 1$.

Finally, to relax the assumption $k>\max(\frac{n+2}{2},3)$ to $k>\max(\frac{n+2}{2},2)$, it suffices to show $u_1u_2\in L^2$ if $k=3>\frac{n+2}{2}$ and $k_2\geq 1$ which we shall now assume.

As $k_1+k_2\geq k$, we either have $k_1+k_2 \geq k+1$ or $k_1+k_2 =k=3$. In the first case the improvement by $1$ to $k>\max(\frac{n+2}{2},2)$ is immediate from the preceding argument. In the latter case, we must in fact have $k_1=2,k_2=1$. We can then argue as before but using the stronger weight condition $w(\xi)^{2}\hat{u}_1\in L^2$.

The integrand in $M(\xi_1)$ is replaced by
\[\frac{1}{\ang{\xi_2''}^2 w(\xi_1-\xi_2)^{2k_1}w(\xi_2)^{2k_2-2}}\]
and consequently the factor $k_1+k_2-2$ in the integrability conditions \eqref{eq:sobalg.cond.1}, \eqref{eq:sobalg.cond.2} can be replaced by $k_1+k_2-1$. We then choose $r\in (0,1)$ such that $r(k_1+k_2-1)=1+\ep$ and conclude as before provided that $k>\max(\frac{n+2}{2},2)$.
\end{proof}

\begin{cor}\label{cor:Sob-alg}
For $k > \max((n+2)/2,2)$, the space $H_\SQ^{k}(\RR^\times\RR^{n})$ is an algebra. 
\end{cor}

\begin{proof} Let $v_1, v_2 \in H_\SQ^{k}(\RR\times\RR^{n})$. We need to show that if we apply an derivative $D_{x'}^\alpha D_{x''}^\beta u$ to the product $v_1v_2$, satisfying $2\alpha+|\beta| \leq k$, then we land in $L^2$. We apply the generalized Leibniz rule to obtain a sum of terms, each of which has the form of the LHS of \eqref{eq:Sob-prod}. Lemma~\ref{lem:basic.sob.alg} then completes the proof. 
\end{proof}

We are now ready to prove our main algebra result.

\begin{proposition}\label{thm:flat mult}
Let $t_0 , t_0' \in \mathbb{R}$, $k \in \mathbb{N}$ and $r_1, r_2 \in \mathbb{R}$.  Provided $k>\max(n/2+2,3)$, 
  \begin{equation}
H_{\mathcal{F}}^{0,r_1;k}(t \geq t_0) \cdot H_{\mathcal{F}}^{0,r_2;k}(t \geq 
t_0) \subset H_{\mathcal{F}}^{0,r_1 + r_2 + (n+1)/2;k}(t \geq t_0),\label{eq:basic mult}
      \end{equation}
and the same holds
with $t \leq t_0$ replacing $t \geq t_0$ everywhere.
Also, for $k$ satisfying the same condition,
\begin{equation}
H_{\mathcal{F}}^{0,r_1;k}(t \geq t_0) \cdot H_{\mathcal{F}}^{0,r_2;k}(t \leq t_0') \subset H_{\mathcal{F}}^{0 ,r_1 + r_2  + (n + 1)/2
  ;k}(t \geq t_0)\label{eq:mult not same}
\end{equation}
% For the spaces with no module regularity but high parabolic
% regularity, provided $k > n + 2$?
%   \begin{equation}
% H_{\mathcal{F}}^{k,r + k;0}(t \geq t_0) \cdot H_{\mathcal{F}}^{0 + k,r + k;0}(t \geq t_0) \subset H_{\mathcal{F}}^{0 + k , 2(r +k);0}(t \geq t_0),\label{eq:good mult}
%       \end{equation}
% and the same holds
% with $t < t_0$ replacing $t \geq t_0$ everywhere.
\end{proposition}
\begin{proof}
  We begin by showing that \eqref{eq:mult not same} follows from
  \eqref{eq:basic mult}.  Indeed, for $v_1 \in H_{\mathcal{F}}^{0,r_1;k}(t \geq t_0)$ and $v_2 \in H_{\mathcal{F}}^{0,r_2;k}(t \leq t_0')$, using the cutoff functions from \eqref{eq:chi time}, we have $v_1 = v_1 \chi_{> t_0 - 1}$, so $v_1 v_2 = v_1 (\chi_{> t_0
  - 1} v_2)$, but it is easy to check that $\chi_{> t_0
  - 1} v_2 \in H_{\mathcal{F}}^{0,r_1;k}(t \geq t_0-1)$, so
\eqref{eq:basic mult} (with $t_0$ replaced by $t_0 - 1$) applies to $v_1 v_2$. 

It remains to prove \eqref{eq:basic mult}.  We assume $r_1 = r_2 = 0$, as the general case follows trivially by
  factoring out the weight functions.  Given $v_1, v_2 \in
 H_{\mathcal{F}}^{0,0;k}(t \geq t_0)$, since we can localize these functions in $H_{\mathcal{F}}^{0,0;k}(t \geq t_0)$ using elements of $C^\infty(X)$,  we can assume either that the support of $v_1$ lies outside some neighbourhood of $\tilde E$, the blowup of the equator, or it lies in some small neighbourhood of $\tilde E$. Then we can localize $v_2$ similarly, so that it is not changed on the support of $v_1$. 
 
 We first consider the case where $v_1$ has support disjoint from $\tilde E$. Then the blowup is irrelevant and we may use $\rho = \rho_{\mf} = \ang{t}^{-1}$ as a boundary defining function for $X$ and $y_j = z_j/\ang{t}$, $j = 1 \dots n$, as angular coordinates. In this region, 
 containment in $ H_{\mathcal{F}}^{0,0;k}(t >
t_0)$ is equivalent to
$$
(\ang{t} D_z) ^\alpha \left( \ang{t} D_t \right)^j \Rot^l
v_1 \in  H_{\mathcal{F}}^{0, 0;0} = L^2(dt\, dz), |\alpha| + 2j + l
\le k
$$
where $\Rot^l$ is any product of $l$ rotations. Noting that rotations $\Rot_{jk} = z_j D_{z_k} - z_k D_{z_j}$ in $z$ transform into rotations $y_j D_{y_k} - y_k D_{y_j}$ in $y$, we  see that this is b-regularity, i.e. holds if and only if
\begin{equation}
D_y^\alpha (\rho D_\rho)^j v_1 \in L^2(dt\, dz) = \rho^{(n + 1)/2} L^2\left(
  \frac{d\rho\, dy}{\rho}\right).\label{eq:reduced to b}
  \end{equation}
  Note that $\rho \sim \langle z, t \rangle^{-1}$ in this region. 
  Setting $\log \rho = \sigma$, we see that \eqref{eq:reduced to b} is equivalent to 
%  $D_y^\alpha (x D_x)^j u \in L^2(dt\, dz) = x^{(n + 1)/2} L^2\left(
%  \frac{dx\, dy}{x}\right)$  holds if
%  and only if
  $$
D_y^\alpha D_\sigma^j v_1 \in  \ang{z,t}^{(n+1)/2}L^2(d\sigma\,  dy)
  $$
and thus $v_1 \in  H_{\mathcal{F}}^{0,0;k}(t \geq t_0)$ if and only if  $v_1(z(\sigma, y),t(\sigma, y)) \in
  \ang{z,t}^{(n+1)/2}H^{k}_{\SQ}(\RR_\sigma\times  \RR^{n}_y)$ as defined in \eqref{eq:HQ-def}. As $H^{k}_{\SQ}(\RR_\sigma\times  \RR^{n}_y)$ is an algebra provided $k > \max((n+2)/2,2)$, by Corollary~\ref{cor:Sob-alg}, equation \eqref{eq:reduced to b} implies that $v_1, v_2$ (supported away from $\tilde E$) satisfy
  $$
v_1 v_2 \in  H_{\mathcal{F}}^{0,(n+1)/2;k}(t \geq t_0).
$$

It remains to treat the case that $v_1$ is supported near $\tilde E$. We may further localize in the $\hat z$ variable to a coordinate patch in $S^{n-1}$ with local coordinates $y = (y_1, \dots, y_{n-1})$.  In this case, we choose any $t_* < t_0$ and take  $\rho_{\mf} = (t-t_*)^{-1}$, $\rho_E= (t-t_*)/|z|$, which are valid coordinates in a neighbourhood of the support of $v_1$ in $X$, together with $y$, for a local coordinate system. To simplify notation, define $T = t-t_*$ and $S = \rho_E^{-1} = |z|/T$. Notice that $S$ is bounded below on the support of $v_1$. 
%In the region $|z /(t - t_0 + 1)| \ge C$, set $x = 1/(t - t_0 + 1)$ and
%$\rho = |z|/(t - t_0 + 1)$.
%Note that $\rho$ is bounded below.  
Then containment of $v_1$ in $H_\SF^{0,0;k}(t>t_0)$ is equivalent to 
$$(T D_T)^i S^{j} D_S^l \Rot^m v_1 \in L^2(dt\,dz)\textrm{ for } 2i+j+l+m \leq k .$$
For such $v_1$ we certainly have 
$$(T D_T)^i D_S^l  D_y^\alpha v_1 \in L^2 (dt \, dz) \textrm{ for } 2i+l + |\alpha| \leq k$$
and so setting $\log T=\sigma$ and writing the measure in our new coordinates, we have
$$D_\sigma^i D_S^l  D_y^\alpha v_1 \in L^2 (dt \, dz) = r^{(n-1)/2} (t-t_*) L^2(d\sigma dS dy)\textrm{ for } 2i+| + |\alpha| \leq k.$$

it follows that $H_\SF^{0,0;k}(t \geq t_0)$ is contained in a weighted parabolic Sobolev space with coordinates $(\sigma, S, y)$ and parabolic direction $\sigma$: 
$$H_\SF^{0,0;k}(t \geq t_0) \subset r^{-(n-1)/2}(t-t_*)^{-1}H_\SQ(\RR_\sigma \times (\RR_S\times \RR_y^{n-1})) .$$
Now take $v_1, v_2 \in H_\SF^{0,0;k}(t \geq t_0)$. 
If we apply a product of generators of $\SF$ of the form $(TD_T)^i S^{j} D_S^l D_y^\alpha$ to the product $v_1v_2$, by the Leibniz rule we obtain a sum of terms of the form $\tilde v_1 \tilde v_2$ where $\tilde v_1 \in H_\SF^{0,0;k_1}(t \geq t_0) $ and $\tilde v_2 \in H_\SF^{0,0;k_2}(t \geq t_0)$ with $k_1+k_2 \geq k$. Then for $k>\max((n+2)/2 ,2)$, it follows from the Leibniz rule and Lemma \ref{lem:basic.sob.alg} that for each such term we have
$$
\tilde v_1 \tilde v_2 \in r^{-(n-1)} (t-t_*)^{-2}L^2(d\sigma \, d\rho\, dS_y^{n-1} )  =r^{-(n-1)/2}(t-t_*)^{-1}L^2(dz\, dt).
$$
Hence $v_1 v_2 \in r^{-(n-1)/2}t^{-1}H_{\SF}^{0,0;k}$. Furthermore, if $k > \max(n/2+1,2)+1=\max(n/2+2,3)$ then from the second generator in \eqref{eq:77},  we gain an additional factor of
$S^{-1}$, so
$$
v_1 v_2 \in S^{-1} r^{-(n-1)/2} (t-t_*)^{-1}  H_{\mathcal{F}}^{0,0;k}(t
 \geq t_0) = r^{-(n + 1)/2}   H_{\mathcal{F}}^{0,0;k}(t
 \geq t_0) .
$$
This concludes the proof of \eqref{eq:basic mult}.\end{proof}

We can now prove
\begin{proposition}\label{thm:final mult}
Let $r, r_1, r_2 \in \mathbb{R}$, $c, c_1, c_2 \in \mathbb{R} \setminus
\{ 0 \}$, $k \in \mathbb{N}$ with $k >\max(n/2+2,3) $.  Then
  \begin{itemize}
  \item $H^{0, r_1; k}_{\mathrm{par},\modu_{c_1}} \cdot H^{0, r_2;
      k}_{\mathrm{par},\modu_{c_2}} \subset H^{0, r_1 + r_2 + (n + 1)/2; k}_{\mathrm{par},\modu_{c_1 +
        c_2}}$
    \item $u \in H^{0, r; k}_{\mathrm{par},\modu_{c}} \iff \overline u \in
      H^{0, r; k}_{\mathrm{par},\modu_{-c}}$ where $\overline{u}$ is the
      complex conjugate.
  \end{itemize}
\end{proposition}
\begin{proof}
The second item is immediate from taking the conjugate of the generator $tD_{z_j}-cz_j/2$.

To prove the first item, let $v_1 \in H^{0, 0; k}_{\mathrm{par},\modu_{c_1}}, v_2 \in H^{0, 0;
    k}_{\mathrm{par},\modu_{c_2}}$.  Then, fixing $t_0$ arbitrarily, and choosing $t_*, t^*$ such that $t_* < t_0$, $t^* > t_0 + 1$, 
    we can write 
  \begin{equation}
    \label{eq:finally mult}
    \begin{split}
      v_1 v_2 &= (\chi_{> t_0} v_1 + \chi_{< t_0 + 1} v_1) (\chi_{> t_0} v_2 +
      \chi_{< t_0 + 1} v_2) \\
      &= \Big(e^{i c_1 |z|^2/4(t - t_*)} w_{1, >} + e^{i c_1|z|^2/4(t - t^*)}
      w_{1, <} \Big)  \\
      &\quad  \times \Big(e^{i c_2 |z|^2/4(t - t_*)}  w_{2, >} + e^{i c_2|z|^2/4(t -
        t^*)} w_{2, <} \Big) \\
      &= e^{i (c_1 + c_2) |z|^2/4(t - t_*)} w_{1, >} w_{2,>} + e^{i
        (c_1 + c_2) |z|^2/4(t - t^*)} w_{1, <} w_{2, <}  \\
      &\quad  + e^{i c_1 |z|^2/4(t - t_*) + i c_2 |z|^2/4(t - t^*)} w_{1, >} w_{2, <} \\
      & \quad + e^{i c_1 |z|^2/4(t - t^*) + i c_2 |z|^2/4(t - t_*)}  w_{1, <} w_{2, >}
    \end{split}
  \end{equation}
  where we have 
  \begin{equation*}
    \begin{split}
      w_{1, >} &= e^{-i c_1 |z|^2/4(t - t_*)}\chi_{> t_0} v_1 \in
      H^{0,0;k}_{\mathcal{F}}(t \geq t_0) \\
      w_{1, <} &= e^{- i c_1|z|^2/4(t - t^*)}\chi_{< t_0 + 1} v_1 \in
      H^{0,0;k}_{\mathcal{F}}(t \leq t_0 + 1) \\
      \tilde w_{2, >} &= e^{- i c_2 |z|^2/4(t - t_*)}\chi_{> t_0} 
      v_2 \in H^{0,0;k}_{\mathcal{F}}(t \geq t_0) \\
      \tilde w_{2, <} &= e^{- i c_2 |z|^2/4(t - t^*)} \chi_{< t_0 + 1}
      v_2 \in H^{0,0;k}_{\mathcal{F}}(t \leq t_0 + 1).
    \end{split}
  \end{equation*}
  By Proposition \ref{thm:flat mult}, the products $w_{1, >} w_{2, >}$ lies in $H^{0, (n + 1)/2 ; k}_{\mathcal{F}}(t \geq t_0)$, while 
  $w_{1, <} w_{2, <}$ lies in $H^{0, (n + 1)/2 ; k}_{\mathcal{F}}(t \leq t_0)$. 
%  
%  It follows that 
%  , w_{1, >} w_{2, <}$ and $w_{1, <} w_{2, >}$ all lie in  \in H^{0, (n + 1)/2 ; k}_{\mathcal{F}}(t \geq t_0)$ and
%    $w_< \tilde w_< \in H^{0, (n + 1)/2 ; k}_{\mathcal{F}}(t < t_0 +
%    2)$.  
    Then from Proposition \ref{thm:flat back and forth}, the terms
    $$
    e^{i (c_1 + c_2) |z|^2/4(t - t_*)} w_{1,>}  w_{2, >} \mbox{ and }
    e^{i
        (c_1 + c_2) |z|^2/4(t - t^*)} w_{1, <} w_{2, <} \mbox{ lie in } H^{0, (n
        + 1)/2; k}_{\mathrm{par},\modu_{c_1 + c_2}}.
      $$
      
To see that the last two terms in \eqref{eq:finally mult} also lie in $H^{0, (n
  + 1)/2; k}_{\mathrm{par},\modu_{c_1 + c_2}}(t \geq t_0)$ we notice that these terms have the form 
 $$
 e^{i (c_1 + c_2) |z|^2/4(t - t_*)} \Big( g_1 \, w_{1, >} w_{2, <} + g_2  \,  w_{1, <} w_{2, >} \Big)
$$
where $g_1$ and $g_2$ are in $C^\infty(X)$ on the support of the products $w_{1, >} w_{2, <}$ and $w_{1, <} w_{2, >}$, which is contained in the time slice $\{ t_0 \leq t \leq t_0 + 1 \}$. This is because $z_j/(t-t_*)$ and $z_j/(t-t^*)$ are smooth bounded functions on $X$ on this support. 
It follows that $g_1 w_{1, >} w_{2, <} + g_2   w_{1, <} w_{2, >}$ lies in $H^{0, (n + 1)/2 ; k}_{\mathcal{F}}(t \geq t_0)$, and again we appeal to Proposition \ref{thm:flat back and forth} to complete the proof. (We could also appeal to Remark~\ref{rem:inv} in the last step.) 
      \end{proof}

\subsection{Further multiplication results} \label{sec:interpolation}

In this section, we extend the multiplication result of Proposition~\ref{thm:final
  mult} to arbitrary nonegative parabolic Sobolev regularity $s$ using a complex interpolation argument. Note that Proposition~\ref{thm:final mult} is the case $s=0$ of the next proposition. 

% \begin{lemma}\label{lem:Bmod}
% Let $B \in \Psip m l$. Then $B$ is a bounded map from $H_{\modu_c}^{s, r; k}$ to $H_{\modu_c}^{s-m, r-l; k}$.
% \end{lemma}

% \begin{proof}
% Let $u$ be in $H_{\modu_c}^{s, r; k}$. To show that $Bu$ is in $H_{\modu_c}^{s-m, r-l; k}$, we need to show that $A_1 \dots A_j Bu \in \Hpar {s-m}{r-l}$ for every product of generators in $\modu_c^{(k)}$. We write 
% $$
% A_1 \dots A_j Bu = B A_1 \dots A_j u + \sum_{i=1}^j A_1 \dots A_{i-1} [A_i, B] A_{i+1} \dots A_j u.
% $$
% Using property \eqref{eq:module properties}, we see that $[A_i, B] \in \Psip m l$ for each generator $A_i$. We shift these operators to the left of the composition, producing double commutators which are also in $ \Psip m l$. Continuing in this way, we express $A_1 \dots A_j Bu$ as a sum of terms $B' A_{m_1} \dots A_{m_q} u$ with $B' \in \Psip m l$ and $A_{m_1} \dots A_{m_q} \in \modu_c^{(k)}$. From this expression it is clear that $B$ is a bounded map from $H_{\modu_c}^{s, r; k}$ to $H_{\modu_c}^{s-m, r-l; k}$. 
% \end{proof}

% \begin{cor}\label{cor:Bmodinv}
% Suppose $B \in \Psip m l$ is an invertible elliptic operator. Then $B$ is a continuous bijection between $H_{\modu_c}^{s, r; k}$ and $H_{\modu_c}^{s-m, r-l; k}$.
% \end{cor}

\begin{proposition}\label{thm:s-mult}
Let $r_1, r_2 \in \mathbb{R}$, $c_1, c_2 \in \mathbb{R} \setminus
\{ 0 \}$, $k \in \mathbb{N}$ with $k >\max(n/2+2,3) $. Let $s \geq 0$. Then
\begin{equation}
H^{s, r_1; k}_{\mathrm{par},\modu_{c_1}} \cdot H^{s, r_2;
      k}_{\mathrm{par},\modu_{c_2}} \subset H^{s, r_1 + r_2 + (n + 1)/2; k}_{\mathrm{par},\modu_{c_1 +
        c_2}}, 
\end{equation}        
 with a norm estimate 
\begin{equation}
\| u_1 u_2 \|_{ H^{s, r_1 + r_2 + (n + 1)/2; k}_{\mathrm{par},\modu_{c_1 + c_2}}} \leq C \| u_1 \|_{H^{s, r_1; k}_{\mathrm{par},\modu_{c_1}}} \cdot \| u_2 \|_{H^{s, r_2; k}_{\mathrm{par},\modu_{c_2}}}. 
\label{eq:prod-norm-est}\end{equation}
\end{proposition}

\begin{proof}
We first note that the result for $s=0$ is precisely Proposition~\ref{thm:final mult}. For $s=2$, we note that the statement that $u \in \Hpar 2 r$ is equivalent to the statement that $D_t u \in \Hpar 0 r$ and $D_{z_i} D_{z_j} u \in \Hpar 0 r$ for all pairs $(i, j)$. In fact, at any point of $T^* \RR^{n+1})$ at fibre-infinity, at least one of $D_t$ or $D_{z_i} D_{z_j}$ is elliptic as an operator of order $(2,0)$. We can therefore form an elliptic operator $B \in \Psip 2 0$ as a sum $Q D_t + \sum Q_{ij} D_{z_i} D_{z_j} + Q_0$ with $Q, Q_{ij}, Q_0 \in \Psip 0 0 $. It follows that $Bu \in \Hpar 0 r$, so by elliptic regularity, we obtain $u \in \Hpar 2 r$ as claimed.  

Now assuming that $u_i$ is in $H_{\modu_{c_i}}^{2, r_i; k}$, $i = 1, 2$,  we consider an expression $A_1 \dots A_j (u_1u_2)$, where the $A_i$ are generators of the module $\modu_{c_1 + c_2}$ and the order is at most $k$. To prove the proposition, we need to show that this is in $\Hpar 2 r$. We test this by applying $D_t$ and $D_{z_i} D_{z_j}$ to this expression in turn, as discussed in the previous paragraph. We only discuss the case of $D_t$ but the argument is exactly the same for $D_{z_i} D_{z_j}$.

We commute $D_t$ to the right of the composition $A_1 \dots A_j$, obtaining commutator terms:
\begin{equation}
D_t A_1 \dots A_j (u_1 u_2) = A_1 \dots A_j D_t (u_1 u_2) + \sum_{i=1}^j A_1 \dots A_{i-1} [D_t, A_i] A_{i+1} \dots A_j (u_1u_2). 
\label{eq:1comm}\end{equation}
We notice that $D_t u_i$ is in $H_{\modu_{c_i}}^{0, r_i; k}$, $i = 1, 2$ using Lemma~\ref{thm:bounded on mod reg spaces}. So writing $D_t(u_1u_2) = (D_t u_1) u_2 + u_1 (D_t u_2)$, and applying Proposition~\ref{thm:final mult}, we see that the first term on the RHS is in $\Hpar 0 {r_1 + r_2 + (n+1)/2}$, as desired. To treat the second term on the RHS we commute once again obtaining 
\begin{multline}
\sum_{i=1}^j A_1 \dots A_{i-1} [D_t, A_i] A_{i+1} \dots A_j (u_1u_2) = \sum_{i=1}^j A_1 \dots A_{i-1} A_{i+1} \dots A_j [D_t, A_i]  (u_1u_2) \\ + 
\sum_{i=1}^j \sum_{l \neq i} A_1 \dots A_{i-1} A_{i+1} \dots A_{l-1} [[D_t, A_i], A_l] A_{l+1} \dots A_j  (u_1u_2).
\label{eq:2comm}\end{multline}
We now note that, since the generators $A_i$ are all first order differential operator with either constant or linear coefficients, the commutator of $A_i$ with $D_t$ is a constant coefficient differential operator, and its parabolic order is at most $2$. It follows that this commutator is a multiple of either $D_t$ or $D_{z_m} D_{z_q}$ or $D_{z_m}$ or $1$. Thus, the first term on the RHS of \eqref{eq:2comm} is treated the same as the first term of \eqref{eq:1comm}. As for the second term, the double commutator is again a multiple of either $D_t$ or $D_{z_m} D_{z_q}$ or $D_{z_m}$ or $1$, and is therefore a product of one or two generators, of order two, in the module. Thus, the second term on the RHS of \eqref{eq:2comm} is a composition of generators, of order at most $k$, applied to $u_1u_2$. We have shown that all the terms are in $\Hpar 0 {r_1 + r_2 + (n+1)/2}$, as desired. This completes the proof for the case $s=2$. It is clear that this method extends to any positive even integer $s$. 

% As the $A_i$ are first order differential operators, we apply the Leibniz rule and obtain a sum of terms, each of the form 
%$$
%\big( A_{\sigma_1} \dots A_{\sigma_{m_1}} u_1 \big) \big( A_{\tau_1} \dots A_{\tau_{m_2}} u_2 \big).
%$$
%We wish to know whether this is in $\H_{\modu_{c_1 + c_2}{2, r_1 + r_2 + (n+1)/2; k}$. We use the characterization just noted, so we apply $D_t$ or $D_{z_i} D_{z_j}$, and distribute these differential operators to the factors in each term. The result is by assumption in $H_{\modu_c}^{0, r; k}$, and  we apply Proposition~\ref{thm:final mult} to deduce that $D_t(u_1u_2)$ and $D_{z_i} D_{z_j} (u_1u_2)$ are both in $H_{\modu_c}^{0, r_1 + r_2 + (n+1)/2; k}$. We then form an elliptic operator . Using Lemma~\ref{lem:Bmod} we see that $Bu \in H_{\modu_c}^{0, r_1 + r_2 + (n+1)/2; k}$. Finally, we invert $B$ microlocally, obtaining 
%$$
%\Id = B^{-1}  B + R
%$$
%where, with a slight abuse of notation, we denote the microlocal inverse by $B^{-1}$ and $R \in \Psip {-\infty}{-\infty}$. It is obvious that $R(u_1 u_2) \in H_{\modu_c}^{2, r_1 + r_2 + (n+1)/2; k}$, while we apply Lemma~\ref{lem:Bmod} once more to obtain $B^{-1} Bu \in H_{\modu_c}^{2, r_1 + r_2 + (n+1)/2; k}$. This completes the proof for $s=2$. 

To prove for other values of $s$, we use complex interpolation
\cite{MR82586} on $L^2(\RR^{n+1})$. We use the complex analytic family of operators 
$$
E_z = \Op((1 + |\zeta|^4 + \tau^2)^z) = \Op((1 + |\zeta|^4 + \tau^2)^a) \circ  \Op((1 + |\zeta|^4 + \tau^2)^{ib}) , \quad z = a + ib, 
$$
where $a , b \in \RR$, which is an operator in $\Psip {4a} 0$. We consider the families of operators $E_z$ and $E_{-z}$ as $z$ ranges over a strip $S := \{ z \mid \Re z \in [0, q] \}$ for some positive integer $q$. To apply the method, we need to have uniform estimates for these operators as $z$ ranges over $S$. 

\begin{lemma}\label{lem:strip} Let $z = a+ib$ range over the strip $S$. 

(i) The operators $E_z$, as well as multiple commutators
$$
[ A_1, [A_2, \dots [A_j, E_z] \dots ]]
$$
are bounded as maps from $\Hpar {4a} 0$ to $\Hpar 0 0 = L^2$, with operator norm uniform for $a \in [0, q]$ and growing at most polynomially as $|b| \to \infty$. Similarly, the operators $E_{-z}$, as well as multi-commutators with generators of $\modu_{c_1 + c_2}$, 
are bounded as maps from $\Hpar {0} 0 = L^2$ to $\Hpar {4a} 0$, with operator norm uniform for $a \in [0, q]$ and growing at most polynomially as $|b| \to \infty$. 

(ii) The operators $E_{-z}$, $z = a+ib$, are bounded from $H_{\modu_c}^{0, r; k}$ to $H_{\modu_c}^{4a, r; k}$, and the operators $E_z$ are bounded from $H_{\modu_c}^{4a, r; k}$ to $H_{\modu_c}^{0, r; k}$, with norms uniformly bounded for $a \in [0, q]$ and growing at most polynomially as $|b| \to \infty$. 
\end{lemma}

\begin{proof} 

(i) To begin we need to fix a particular norm on the spaces $\Hpar {4a} 0$. The easiest way to do this is to \emph{define} the norm of $f \in \Hpar {4a} 0$ to be the norm of $E_{a} f$ in $L^2$. Then, for $z=a+ib$, the operator norm of $E_z : \Hpar {4a} 0 \to L^2$ is equal to the operator norm of $E_{ib}$ acting on $L^2$. This can be estimated using the Calder\'on-Vaillancourt theorem, which bounds the operator norm of a pseudodifferential operator in terms of the sup norm of a finite number of derivatives of  the symbol. It is easy to see that derivatives of the symbol of $E_{ib}$ are bounded polynomially in $b$ as $|b| \to \infty$. 

Now, to treat multi-commutators, we use again the fact that the $A_i$ are first order differential operators with coefficients that are either constant or linear functions. It follows that the \emph{full} symbol of a commutator $[A_i, E_z]$ is given by a constant-coefficient vector field in $\zeta, \tau$ applied to the symbol of $E_z$, and inductively, the full symbol of a multi-commutator is given by a constant coefficient differential operator applied to the symbol of $E_z$. We again apply the Calder\'on-Vaillancourt theorem to see that the operator norm of such a multi-commutator, from $\Hpar {4a} 0 \to L^2$ is bounded polynomially in $b$ as $|b| \to \infty$. 

The proof for $E_{-z}$ is similar, mutatis mutandis. 

(ii) This follows from part (i) and the proof of Lemma~\ref{thm:bounded on mod reg spaces}; we omit the details. 
\end{proof}

Returning to the proof of the Proposition, we choose elements $u_i \in H_{\modu_{c_i}}^{0, r_i; k}$ and choose an arbitrary $f \in \Hpar 0  {-(r_1 + r_2)/2}$. We fix a composition $A_1 \dots A_j$ of generators of $\modu_{c_1 + c_2}$ with total order at most $k$, and consider the analytic function 
\begin{equation}
z \mapsto \Big\langle E_z \Big( A_1 \dots A_j (E_{-z} u_1)(E_{-z} u_2) \Big), f \Big\rangle.
\label{eq:z}\end{equation}
Using Corollary~\ref{cor:Bmodinv}, for real $z = a$, $E_{-a} u_i$ ranges over all elements of $H_{\modu_{c_i}}^{4a, r_i; k}$. Moreover, by Corollary~\ref{cor:Bmodinv}, $E_a$ is a linear isomorphism between $H_{\modu_{c_1+c_2}}^{4a, r_1 + r_2 + (n+1)/2; k}$ and $H_{\modu_{c_1+c_2}}^{0, r_1 + r_2 + (n+1)/2; k}$. 
So it is enough to prove that \eqref{eq:z} evaluated at $z=s/4$ is bounded by 
\begin{equation}
C \| u_1 \|_{H^{0, r_1; k}_{\mathrm{par},\modu_{c_1}}} \cdot \| u_2 \|_{H^{0, r_2; k}_{\mathrm{par},\modu_{c_2}}} \| f \|_{\Hpar 0  {-(r_1 + r_2)/2}}.
\label{eq:s-bound}\end{equation}
We use the three lines lemma. By Lemma~\ref{lem:strip} and Proposition~\ref{thm:final mult}, when $\Re z = 0$, the norm of the product $(E_{-z} u_1)(E_{-z} u_2)$ is bounded in $H_{\modu_{c_1 + c_2}}^{0, r_1 + r_2 + (n+1)/2; k}$ by 
\begin{equation}
C \ang{\Im z}^N \| u_1 \|_{H^{0, r_1; k}_{\mathrm{par},\modu_{c_1}}} \cdot \| u_2 \|_{H^{0, r_2; k}_{\mathrm{par},\modu_{c_2}}}
\label{eq:stripedge}\end{equation}
 for some $N$. It follows that the left-hand side of the inner product is bounded in $\Hpar 0 {r_1 + r_2 + (n+1)/2}$ by  
 $$
 C \ang{\Im z}^N \| u_1 \|_{H^{0, r_1; k}_{\mathrm{par},\modu_{c_1}}} \cdot \| u_2 \|_{H^{0, r_2; k}_{\mathrm{par},\modu_{c_2}}}
 $$ 
 for some (new) $N$, and so, on the edge $\Re z = 0$, the analytic function \eqref{eq:z} is bounded by 
 \begin{equation}
 C \ang{\Im z}^N \| u_1 \|_{H^{0, r_1; k}_{\mathrm{par},\modu_{c_1}}} \cdot \| u_2 \|_{H^{0, r_2; k}_{\mathrm{par},\modu_{c_2}}}  \| f \|_{\Hpar 0  {-(r_1 + r_2)/2}}
\label{eq:stripbound}\end{equation}
 for some $N$. On the other edge of the strip $S$, for $a = \Re z = q$, we find using part (ii) of Lemma~\ref{lem:strip} that $E_{-z} u_i$ is in $H_{\modu_{c_i}}^{4q, r_i; k}$, with norm growing at most polynomially in $\Im z$, and therefore, since $4q$ is an even integer, that the product $(E_{-z} u_1)(E_{-z} u_2)$ is in $H_{\modu_{c_1 + c_2}}^{4q, r_1 + r_2 + (n+1)/2; k}$ with norm in this space bounded by \eqref{eq:stripedge}. As before, this means that the left hand side of the inner product in \eqref{eq:z} is in $\Hpar 0 {r_1 + r_2 + (n+1)/2}$, and therefore, for $\Re z = q$, the norm of this analytic function is again bounded by \eqref{eq:stripbound}. Applying the three lines lemma we deduce that, for $0 \leq s \leq 4q$, the analytic function \eqref{eq:z} at $z = s/4$ is bounded by \eqref{eq:s-bound}, as required. 
%\begin{equation}
%C  \| u_1 \|_{H^{0, r_1; k}_{\mathrm{par},\modu_{c_1}}} \cdot \| u_2 \|_{H^{0, r_2; k}_{\mathrm{par},\modu_{c_2}}}.
%\label{eq:stripedge}\end{equation}
%
%It follows, taking the sup over all $f \in \Hpar 0  {-(r_1 + r_2)/2}$ of norm $1$,  that the product $E_{-s/4} u_1 \cdot E_{-s/4} u_2$ is in $H^{s, r_1 + r_2 + (n + 1)/2; k}_{\mathrm{par},\modu_{c_1 + c_2}}$ with norm bounded by \eqref{eq:stripedge}. 
%We now take arbitrary $v_i \in H^{s, r_i; k}_{\mathrm{par},\modu_{c_i}}$\ah{We probably need a density argument here} and let $u_i = E_{s/4} v_i$. Writing in terms of $v_1$ and $v_2$, we have shown that $v_1v_2$ is in $H^{s, r_1 + r_2 + (n + 1)/2; k}_{\mathrm{par},\modu_{c_1 + c_2}}$ with norm bounded by \eqref{eq:prod-norm-est}, as required. 
\end{proof}

%%%%%%%%%%%%%%%%%%%%%%%%%%%%%%%%%%%%%%%%%%%%%%%%%%%%%%%%%%%%%%%%%%%
%%%%%%%%%%%%%%%%%%%%%%%%%%%%%%%%%%%%%%%%%%%%%%%%%%%%%%%%%%%%%%%%%%%
%%%%%%%%%%%%%%%%%%%%%%%%%%%%%%%%%%%%%%%%%%%%%%%%%%%%%%%%%%%%%%%%%%%
%%%%%%%%%%%%%%%%%%%%%%%%%%%%%%%%%%%%%%%%%%%%%%%%%%%%%%%%%%%%%%%%%%%
%%%%%%%%%%%%%%%%%%%%%%%%%%%%%%%%%%%%%%%%%%%%%%%%%%%%%%%%%%%%%%%%%%%

\section{Existence and asymptotics of solutions}\label{sec:results}

\subsection{Linear solutions and asymptotic data}\label{sec:new linear
stuff}
 We recall that the Poisson operators $\Poipm$ were defined in section~\ref{sec:intro}, above \eqref{eq:Poissonrange}, and that for the free operator $P_0 = D_t + \Delta_0$, we have $\Poim = \Poip$; these Poisson operators in the free case are denoted $\Poi_0$. 
As in \cite{TDSL}, we consider asymptotic data $f$ (see
\eqref{eq:fminus}--~\eqref{eq:fplus}) in a space
$\mathcal{W}^k(\mathbb{R}^n)$, which is the $k$th order module regularity space with respect to a module $\Nhat$. This module is such that the free Poisson operator intertwines the generators of $\Nhat$, which are 
\begin{equation}
  \label{eq:Wk generators}
\hat{\mathcal{N}}_{gen} :=  \left\{  \Id, \quad \zeta_j D_{\zeta_l} -
  \zeta_l D_{\zeta_j}, \quad D_{\zeta_j}, \quad \zeta_j  \right\}
\end{equation} 
with the generators of the module $\modu$. 
Then $\mathcal{W}^k(\mathbb{R}^n)$ consists exactly of those
distributions $f \in \mathcal{S}'(\mathbb{R}^n)$ such that
$$
\forall A_1, \dots, A_k\in \hat{\mathcal{N}}_{gen}, \qquad  A_1 \cdots A_k f \in L^2(\mathbb{R}^n_{\zeta}, d\zeta)
$$
See \cite[Sect.\ 7.1]{TDSL} for further background on the
$\mathcal{W}^k$ spaces.
Let $f \in \mathcal{W}^k$ with $k \geq 0$.  In \cite[Section 7]{TDSL}, we established asome precise mapping
properties for the Poisson operators $\Poipm$; here we slightly generalize these results. 
%, defined initially on
%smooth compactly supported $f$, which produces a solution $u_0$ to
%$Pu_0 = 0$ with incoming data $f$.  
%Specifically, we showed that
%$\mathcal{P}_+$ intertwines the natural differential operators from
%the module $\SN$, the module in $\Psi^{1,1}_{par}$ of
%operators characteristic on the radial set discussed in the introduction, with the module $\hat{\mathcal{N}}_{gen}$ (see \cite[Eq. (7.7)]{TDSL}).  

\begin{proposition}\label{thm:Poisson mapping}[c.f. Prop. 7.5 \cite{TDSL}]
Let $\epsilon > 0$ and let $\mathrm{r}_{\min} \in C^\infty(\SymbSpa)$ satisfy
$\mathrm{r}_{\min} \leq -1/2$, $\mathrm{r}_{\min} = -1/2$ outside small neighbourhoods of the radial set $\SR$, 
 and $\mathrm{r}_{\min} = -1/2 - \epsilon$ in a smaller neighborhood of
$\mathcal{R}$.  For $\ell, k \in
\mathbb{N}_0$, the Poisson operators extend from $\mathcal{S}$ to 
continuous mappings 
  $$
\mathcal{P}_\pm \colon \langle \zeta \rangle^{-\ell} \mathcal{W}^k
\longrightarrow H_{\SN}^{1/2 + \ell,
  \mathrm{r}_{\min};k} \cap \ker P.
  $$ %\ah{Removed $\mathcal{N}$ and replaced with $\modu$.}
\end{proposition}

\begin{proof} (sketch) The proof is an easy modification of Prop. 7.5 of \cite{TDSL}, which is the case $l=0$. There it is shown that $\Poip f$ is equal, microlocally near $\SR_+$, to $\Poi_0 f$. Since we have propagation of regularity of solutions to $Pu = 0$, which follows from \cite[Eq. (7.2)]{TDSL} and the mapping properties of the propagators for example, it is enough to prove the regularity for $\Poi_0 f$. This follows straightforwardly from the fact that $\Poi_0$ intertwines $\zeta_i$ and $D_{z_i}$, and the fact that, on the characteristic variety at fibre infinity and away from the radial set, at least one the operators $D_{z_i}$ is elliptic.
\end{proof} 

\begin{prop}\label{cor:NvD} 

(i) In the same setting as Prop.~\ref{thm:Poisson mapping}, for $k \geq 0$ and $l \in \ZZ$, the Poisson operators extend from $\mathcal{S}$ to 
continuous mappings 
  $$
\mathcal{P}_\pm \colon \langle \zeta \rangle^{-\ell} \mathcal{W}^k
\longrightarrow H_{\modu}^{1/2 + \ell,
  \mathrm{r}_{\min};k} \cap \ker P.
  $$

(ii) For $k < 0$ and $l \in \ZZ$, the Poisson operators 
extend from $\mathcal{S}$ to 
continuous mappings 
  $$
\mathcal{P}_\pm \colon \langle \zeta \rangle^{-\ell} \mathcal{W}^k
\longrightarrow \Hpar{1/2 + k + \ell}{-1/2 + k} \cap \ker P.
  $$ 
  
  (iii) For $k, l \in \ZZ$ and $Q \in \Psip 0 0$, microsupported in the region $\{\mathrm{r}_{\min} = -1/2 \}$  the operator $Q \Poipm $ extends to a continuous mapping
$$
Q \Poipm  : \langle \zeta \rangle^{-\ell} \mathcal{W}^k
\longrightarrow \Hpar{1/2 + k + \ell}{-1/2 + k}.
  $$ 
  
  (iv)
  For $k, l \in \ZZ$ and $Q \in \Psip 0 0$, the operator $\Poipm^* Q $ extends to a continuous mapping
$$
 \Poipm^* Q  : \Hpar{-1/2 + k + \ell}{1/2 + k} 
\longrightarrow \langle \zeta \rangle^{-\ell} \mathcal{W}^k.
  $$

\end{prop}

\begin{proof} (i) This is an immediate consequence of Prop.~\ref{thm:Poisson mapping} and Lemma~\ref{thm:mod equiv in ker}. 

(ii) The proof is very similar to the proof of Prop.~\ref{thm:Poisson mapping}; we use propagation of regularity to observe that we only need to obtain the result in a deleted neighbourhood of $\SR_\pm$. Notice that as $k < 0$ we are below threshold here in the spatial decay order, so we propagate regularity into the radial set directly (Proposition~\ref{prop:mod.belowschrod}); no variable order is required in this case. Then in a small deleted neighbourhood of $\SR_\pm$, we only need to prove the result for $\Poi_0$ since $\Poi_\pm f$ is equal, microlocally near $\SR_\pm$, to $\Poi_0 f$. This is straightforward. 

(iii) This is an immediate consequence of (i) and (ii). 

(iv) This is obtained by dualizing (iii). Notice that to obtain (iv) for $ k > 0$, as in our application in the proofs of Theorems~\ref{thm:linear2}, \ref{thm:main} and \ref{thm:main2}, we need (iii) for $k < 0$; this is the reason for including the case $k< 0$ in the  Proposition.

\end{proof}

%Note in particular that one can take $\mathrm{r}_{\min} =
%\min(\mathrm{r}_+, \mathrm{r}_-)$ where $\mathrm{r}_\pm$ are the
%incoming/outgoing spacetime weights defined
%below \eqref{eq:var-weight-invertibility}, which are
%$\equiv -1/2$ away from $\mathcal{R}$.  
%
%%\ah{Removed sentence about $\mathcal{N}$, may need remark comparing $\mathcal{N}$ and $\modu$.}
%
%\ah{Moved lemma to Section 3}
%
%Let $u_0 = \Poim f$ solve $Pu_0
%= 0$ with incoming data $f$.  With $Q_+$ a cutoff to $\mathcal{R}_+$
%as in the previous section, \ah{Need ref} write
%\begin{equation}
%u_0 = Q_+ u_0 + Q_- u_0 = u_+ + u_-,\label{eq:uminus uplus}
%\end{equation}
%with
%$Q_- = I - Q_+$, so that $Q_\pm u_0 \in H_{\modu} ^{1/2,
%  \mathrm{r}_\pm^0 ; k}(\mathbb{R}^{n + 1})$ where for any $N \in \mathbb{R}$,
%$$
%\mathrm{r}_\pm^0 = - 1/2 - \epsilon \mbox{ on } \WF'(Q_\pm)
%\mbox{ and } \mathrm{r}_\pm^0 \equiv -N  \mbox{ near } \mathcal{R}_\mp.
%$$
%We will retain the notation $u = u_- + u_+$ throughout.

We have the following statement regarding the asymptotics of $u_0$,
which will apply also to solutions to the nonlinear equations.
We have the analogue of Proposition 7.7 from
\cite{TDSL} (which is in turn the analogue of Proposition 3.4 of
\cite{NLSM},
adapted to our differential  module $\modu$.
Let $\sw_{\max}$ be a
variable weight with $\sw_{\max} = -1/2 + \epsilon$  near the radial
sets, $-1/2$ off  neighborhoods of the radials sets, and monotone
along the flow.  
\begin{proposition}\label{prop:limits}
  Let $v \in H^{-1,\sw_{\max} + 1; k}_{ \modu}(\RR^{n+1})$ with
  $k \geq 2$ and $\epsilon > 0$. Then $\upl$ is such that
  the limits
\begin{equation}\label{eq:tlimit}
\mathcal{L}_+ \upl(\zeta) := \lim_{t \to +\infty} (4\pi it)^{n/2} e^{-it|\zeta|^2} \upl(2t \zeta, t) 
\end{equation} 
and
\begin{equation}\label{eq:zerolimit}
\mathcal{L}_- \upl(\zeta) := \lim_{t \to -\infty} (4\pi it)^{n/2} e^{-it|\zeta|^2} \upl(2t \zeta, t) 
\end{equation}
exist in $\ang{\zeta}^{1/2 + \epsilon}\mathcal{W}^{k-2}(\RR^n_\zeta)$, with the limit \eqref{eq:zerolimit} identically zero. 

Moreover, we have estimates 
\begin{equation}\begin{gathered}
\| \mathcal{L}_+ \upl \|_{\ang{\cdot}^{1/2 + \epsilon}\mathcal{W}^{k-2}} \leq C \| v \|_{H^{-1, \sw_{\max} + 1; k}_{\modu}},  \\
\| t^{n/2} e^{-it|\zeta|^2/4} \upl(t \zeta, t) - \mathcal{L}_+ \upl \|_{\ang{\cdot}^{1/2 + \epsilon}\mathcal{W}^{k-2}} = O(t^{-\epsilon'}), \quad t \to \infty
\end{gathered}\end{equation}
for $\epsilon'$ sufficiently small. A similar statement is true for $u_- := \Rin v$, with a zero limit $\mathcal{L}_+ u_-$ as $t \to +\infty$ and a (potentially) nonzero limit $\mathcal{L}_- u_-$ as $t \to -\infty$. 
\end{proposition}

\begin{proof}
  The proof follows that in \cite[Prop.\ 7.7]{TDSL} closely.  The
  modules involved are nominally different, but operators used in the
  proof in the cited paper lie in both modules.  For the
  convenience of the reader we recall the argument here.

It begins by defining
$$
\tilde u(\zeta, t) = (4\pi i t)^{n/2} e^{-it|\zeta|^2} \upl(2t \zeta, t),
$$
and computing the partial derivative of $\tilde u(\zeta, t)$; as in
the cited work we denote this derivative by $D_t \rvert_{\zeta}$ to
distinguish it from the partial derivative with respect to $t$ with $z$ fixed.
Then,
\begin{equation}\begin{gathered}
D_t \rvert_{\zeta} \tilde u(\zeta, t) = (4\pi it)^{n/2} e^{-it|\zeta|^2}\left( v(2t\zeta,t)  - \left( t^{-2}  (t D_z -
 \frac{z}{2} ) \cdot (t D_z - \frac{z}{2}) \right) \upl(2t \zeta, t) \right)
 \end{gathered}\label{Dtzeta}\end{equation}
 We recognize the factor $t D_{z_i} - z_i/2$ as an element of the
 module $\modu$. So we are now in a similar position to the
 proof in \cite{NEH}.

 Starting with $k = 2$, by Theorem \ref{prop:big mod reg invert},
 $\upl \in H_\modu^{0, \sw_+;  2}$, which allows us to conclude that the
 second term in the parenthesis has
 $$
\left((t D_z -
 \frac{z}{2} ) \cdot (t D_z - \frac{z}{2}) \right) \upl(2t \zeta, t)  \in H_{\mathrm{par}}^{0, \sw_+}
 $$
 % Moreover, from the assumption on $v$, using: (1) 
 % $\Psi^{1,0}_{\mathrm{par}} \subset \modu$, (2) that $\sw_{\max} = -1/2 + \epsilon$ near the radial sets, and (3) that $\modu$ is
 % elliptic away from the radial sets, we have
Now we decompose $v=v_1+v_2$ into a sum of terms microlocalised near and away from $\SR$ respectively. 
Using that $\sw_{\max}=-1/2+\epsilon$ near $\SR$, and $\Psip{2}{0}\subset \modu^{(2)}$, we get $v_1\in H_{\mathrm{par}}^{1,1/2+\ep}$.
As $\modu $ is elliptic away from the radial set, we also get $v_2\in H_{\mathrm{par}}^{1,\sw_{\max}+3}$.
Hence
$$v \in \ang{(t,z)}^{-1/2 - \epsilon}
L^2(dtdz).
 $$
 Thus, on $0 < T < t$
\begin{multline}
D_t \tilde u (\zeta, t)  \rvert_{\zeta} \in   \ang{(t,z)}^{-1/2 - \epsilon}t^{n/2}
L^2(dtdz) + \ang{(t,z)}^{1/2 + \epsilon}t^{n/2} \ang{t}^{-2} L^2(dt dz) \\ \subset t^{-1/2 - \epsilon} \ang{\zeta}^{1/2 + \epsilon} L^2(dt d\zeta).
\label{eq:Dtuliesin}\end{multline}
 This is in $$t^{-\epsilon'} L^1([T, \infty)_t; \ang{\zeta}^{1/2 +
  \epsilon}L^2(\RR^n_\zeta))$$ for  $0 < \epsilon' < \epsilon$. We can
thus integrate the $t$-derivative, for fixed $\zeta$, of $\tilde u$,
viewed as a function of $t$ with values in $\ang{\zeta}^{1/2 +
  \epsilon}L^2(\RR^n_\zeta)$, out to infinity, showing that the limit
exists. Moreover, the convergence is at a rate of $O(t^{-\epsilon'})$
as we see by integrating $D_t \rvert_{\zeta} \tilde u$ back from $t = \infty$. 

The rest of the proof (i.e.\ treating $k > 2$) is identical verbatim
to the $k > 2$ part of the proof of Proposition 7.7 of \cite{TDSL}.
 \end{proof}
 
 \begin{proof}[Proof of Theorem~\ref{thm:linear2}]
The first statement in Theorem~\ref{thm:linear2} regarding uniqueness was proved in \cite{TDSL}, and the second is just part (i) of  Prop.~\ref{thm:Poisson mapping}. 

To prove the converse, assume that $u$ satisfies $Pu = 0$ and lies in the space 
$$
\Hparmod {1/2+\ell} {\mathrm{r}_-} k + \Hparmod {1/2+\ell} {\mathrm{r}_+} k = \Hparmod {1/2+\ell} {\mathrm{r}_{\min}} k .
$$
We choose a microlocal partition $\Id = Q_+ + Q_-$, such that $Q_\pm$ is equal to the identity microlocally near $\SR_\pm$, and so that $\WF'[Q_\pm P] \cap \Char(P)$ is contained in the set where $\mathrm{r}_{\min} = -1/2$. We make use of the identities \cite[Eq. (7.2)]{TDSL} 
\begin{equation}
u = (P_+^{-1} - P_-^{-1}) [P, Q_+] u
\label{eq:Qu}\end{equation}
and \cite[Prop. 7.9]{TDSL}
\begin{equation}
\Poim \Poim^* = \Poip \Poip^* = i (2\pi)^{-n} (P_+^{-1} - P_-^{-1}).
\label{eq:PP*}\end{equation}
%We break up $ u = u_+ - u_-$ where $u_\pm = P_\pm^{-1} [P, Q_+] u$. 
Then, using part (iii) of Prop.~\ref{thm:Poisson mapping}, with an adjustment of orders due to the fact that 
$[P, Q_+]$ is in $\Psip 1 {-1} $, we have 
\begin{equation}
[P, Q_+] u \in \Hpar {-1/2 + \ell + k}{1/2 + k} .
\label{eq:PQp}\end{equation}
Using identity \eqref{eq:PP*} we write $u$ as 
$$
u = \Poim f, \quad f = -i (2\pi)^n \Poim^* [P, Q_+] u.
$$
We further choose a microlocal cutoff $Q' \in \Psip 0 0$ such that $Q'$ is microlocally equal to the identity on $\WF'([P, Q_+])$ and $\WF'(Q')$ is disjoint from the radial set. Then we have 
$$
u = \Poim f, \quad f = -i (2\pi)^n \Big( \Poim^* Q'[P, Q_+] u + \Poim^* (\Id - Q')[P, Q_+] u \Big) .
$$
The contribution of the second term to $f$ is Schwartz, as $(\Id - Q') [P, Q_+] \in \Psip {-\infty}{-\infty}$, according to \cite[Corollary 7.6]{TDSL}. As for the first term, we use \eqref{eq:PQp} and the mapping property of $\Poim^* Q'$ from part (iv) of Prop.~\ref{cor:NvD} to see that $f$ is in $\ang{\zeta}^{-l} \Hdata^k$. A similar argument, using $\Poip, \Poip^*$ instead of $\Poim, \Poim^*$ shows that the final state data as $t \to +\infty$ is also in this space. 
\end{proof}

%When combined with \cite[Theorem 1.3]{TDSL}, Proposition \ref{thm:Poisson mapping}, Lemma \ref{thm:mod equiv in ker}, and  Proposition \ref{prop:limits}, together with \eqref{eq:invertible mapping} furnish a proof of Theorem \ref{thm:linear}.

%To perturb off of these linear solutions we use a contraction
%mapping argument together with mappings properties of the resolvent,
%the technical parts of which are established in the appendix below. 
%See the appendix for the proof.

\subsection{Existence of solutions via contraction mapping, and the proof of Theorem~\ref{thm:main}}
Given $f \in \mathcal{W}^k$ we want to solve the nonlinear equation
$$
Pu = N[u], 
$$
with final state data $f$ as $t \to -\infty$. We perturb off linear solutions using a contraction mapping argument, together with mapping properties or the forward and backward propagators as in Theorem~\ref{thm:linear1}. Throughout this section, we assume that 
\begin{multline}
\text{ $\mathrm{r}_\pm$ are variable orders obeying the conditions of Theorem~\ref{thm:linear1},} \\ \text{ and in addition, take values in $[-1/2 - \epsilon, -1/2 + \epsilon]$}
\label{eq:rcond}\end{multline}
for some specified small $\epsilon > 0$. 
%We note the following simple consequence of Theorem~\ref{thm:linear1} which allows us to pass from variable order spaces to fixed order spaces in order to apply the multiplicative results of Section~\ref{sec:mult}. 
%
%\begin{prop}\label{thm:main module propagator}
%  For $s \in \mathbb{R}$, $\epsilon > 0$, $\mathrm{r}_\pm$ as above,  and $k \in \mathbb{N}_0$, the forward and backward propagators are continuous maps 
%  \begin{equation}
%    \label{eq:isomorphism}
%P_\pm^{-1} \colon H_{\modu}^{s, 1/2 + \epsilon; k} \lra
%\mathcal{X}_{\modu}^{s + 1, \mathrm{r}_\pm ; k}.
%\end{equation}
%%where $\mathrm{r}$ is any variable order which is monotone decreasing
%%along the flow satisfying
%%$$
%%\mathrm{r} = -1/2 + r_0 \mbox{ near } \mathcal{R}_-, \mbox{ and }
%%\mathrm{r} = -1/2 -\epsilon \mbox{ near } \mathcal{R}_+
%%$$
%\end{prop}
%
%\begin{proof} This is an immediate consequence of Theorem~\ref{thm:linear1} and the embedding 
%$$
% H_{\modu}^{s, 1/2 + \epsilon; k} \hookrightarrow  H_{\modu}^{s, \mathrm{r}_\pm + 1; k}
%.
%$$
%\end{proof}

Given $f \in \SW^k$, let $u_0 = \Poim f$ be the linear solution with final state data $f$ as $t \to -\infty$, and decompose $u_0 =  u_- + u_+$ as discussed above. We look for solutions to
the nonlinear equation $Pu = N[u]$  of the form $u = w + u_-$ with
$$
w \in \mathcal{X}^{1/2, \mathrm{r_+}; k}.
$$
Specifically, we try to solve 
\begin{equation}
w = u_+ + \Rout N[u_- + w];
\end{equation}
given this equation, then we find that 
\begin{multline}
Pu = P(u_- + w) = P \Big( u_- + u_+ + \Rout N[u_- + w] \Big) \\ = P \Rout N[u_- + w] = N[u_- + w] = N[u],
\end{multline}
as required. In the above calculation we use the fact that $P(u_- + u_+) = 0$ and that $N[u_- + w]$ has spatial decay index strictly greater than $1/2$, so we can apply $\Rout$ to it. The second point is elaborated in the proof below. 
%
% such that
%\begin{equation}
%\mathrm{r} = -1/2 - \epsilon \mbox{ near } \mathcal{R}_+ \mbox{ and }
%\mathrm{r} \equiv -1/2 +  r_{p} \mbox{ off  } \WF'(Q_+),\label{eq:w space}
%\end{equation}
%\jg{why is this called $r_{p}$?}\ah{Because it has to be taken $O(1/p)$}
%where $r_{p}$ is a constant satisfying the bound stated below;
%then $u = u_- + w$ is easily seen to solve
%$P u = N[u]$, and, as we will show below it also has the correct
%asymptotic properties.  

Thus we define
\begin{equation}\label{eq:contration}
\Phi(w) = u_+ + \Rout(N[u_- + w]),
\end{equation}
and show that it is a contraction mapping on an appropriate set:

\begin{proposition}\label{thm:contract power}
  For $k > \max(n/2+2,3)$, $\epsilon < (p+1)^{-1}$ 
  and $f$ sufficiently small in $\mathcal{W}^k$, with
  $N[u]$ the power nonlinearity in \eqref{eq:N} (or more generally a
  finite sum of such power nonlinearities with minimum $p$ satsifying
  $(n,p) \neq (1,3)$), the mapping $\Phi$ is
  a contraction on a sufficiently small ball around the origin in
  $\SX_{\modu}^{1/2, \mathrm{r}_+; k}$.  
%  \begin{equation}
%    \label{eq:w vanishing product}
%   0 <  r_{p} < -1 +    (p-1)\frac{n}{2}
%    \end{equation}
%Writing instead $\Poip f = u_0 = u_- + u_+$ we obtain a solution $u =
%w + u_+$ to the final state problem with $w \in  H_{\modu}^{1/2,
%  \mathrm{r}'; k}$ where $\mathrm{r}'$ satisfies \eqref{eq:w space}
%and \eqref{eq:w vanishing product} at the opposite radial sets.
      \end{proposition}
% \begin{remark}\label{rem:soffer}
% As mentioned in the introduction, our methods cannot work in the case
% $n = 1, p = 3$, the only case for odd $p \ge 3$ excluded in the
% theorem.  Recalling the ansatz of Lindblad and Soffer in
% \eqref{eq:soffer ansatz}, simply note that applying the Galillean
% elements in the module to the leading order term from this ansatz gives
% \begin{equation}
% \begin{gathered}
%   (t D_z - z/2) \left( t^{-1/2} e^{i z^2 / (4t)} a(z/t) \exp\left( - i\beta
%       a(z/t)^2 \ln |t| + i b(z/t)\right)\right) = \\
%   \qquad O(t^{-1/2}) + t^{-1/2} e^{i z^2 / (4t)} a(z/t)    t D_z   \exp\left( - i\beta
%     a(z/t)^2 \ln |t| + i b(z/t)\right) \\
% \qquad \qquad   = O(t^{-1/2} \ln |t|)
% \end{gathered}\label{eq:2}
% \end{equation}
% and no better.  Whence repeated applications of this module derivative
% produce additional blow up, and thus our methods are not suited
% to this case.
% \end{remark}    
\begin{proof}[Proof of Proposition \ref{thm:contract power}]
  To begin with,
  \begin{equation}
    \| |u_- + w|^{p-1} (u_- + w) \|_{H_{\modu}^{1/2, p (-1/2 - \epsilon) + (p - 1)(n +
        1)/2;k}} \lesssim \left( \| u_- \|_{H_{\modu} ^{1/2,
          \mathrm{r}_- ; k}} +  \| w \|_{H_{\modu}^{1/2, \mathrm{r}_+;
          k}}\right)^{p}.\label{eq:power estimate}
  \end{equation}
  To see this holds, note that both
  $H_{\modu}^{1/2,\mathrm{r}_-; k}$ and
  $H_{\modu} ^{1/2, \mathrm{r}_+ ; k}(\mathbb{R}^{n + 1})$
  lie in $H_{\modu}^{1/2,-1/2 - \epsilon; k}$. Hence,  we have
$$
    \| u_- + w \|_{H_{\modu}^{1/2, -1/2 - \epsilon;k}} \lesssim \| u_- \|_{H_{\modu} ^{1/2,
          \mathrm{r}_- ; k}} +  \| w \|_{H_{\modu}^{1/2, \mathrm{r}_+;
          k}}.
      $$
Then \eqref{eq:power estimate} follows from Proposition
      \ref{thm:s-mult} applied repeatedly, noting that the final module is $\modu = \modu_1$, with $c=1$, due to phase invariance \eqref{eq:phaseinv}.  Thus in particular for any $f$, $\Phi$
      maps $H_{\modu}^{1/2, \mathrm{r}; k}$ to
        $H_{\modu}^{1/2, p(-1/2 - \epsilon) + (p - 1)(n + 1)/2;k}$.
 The spatial exponent here can be written 
 $$
p(-\frac1{2} - \epsilon) + \frac{(p - 1)(n + 1)}{2} = \frac{(p-1)n}{2}  - \frac1{2} - p \epsilon.
$$
It is easy to check that if $n \geq 1$ is a positive integer and $p \geq 1$ is an odd integer, with $(n, p) \neq (1, 3)$, then 
$$
\frac{(p-1)n}{2}  - \frac1{2} \geq \frac{3}{2}.
$$
So this exponent is at least $1/2 + \epsilon$ provided that $\epsilon < (p+1)^{-1}$. Given this we have 
\begin{multline}
\| N[u_- + w] \|_{H_{\modu}^{1/2, \mathrm{r}_+ + 1;k}}  \lesssim \| N[u_- + w] \|_{H_{\modu}^{1/2, 1/2 + \epsilon;k}}  \\ \lesssim \left( \| u_- \|_{H_{\modu} ^{1/2,
          \mathrm{r}_- ; k}} +  \| w \|_{H_{\modu}^{1/2, \mathrm{r}_+; k}}\right)^{p}
\label{eq:Nbound}\end{multline}
We can then apply $\Rout$ to land back in $H_{\modu}^{1/2, \mathrm{r}_+; k}$ (in fact, we land in the more regular space $H_{\modu}^{3/2, \mathrm{r}_+; k}$, a fact that will be exploited in the next subsection) . Thus 
 $\Phi$ is a mapping on $H_{\modu}^{1/2, \mathrm{r}_+; k}$. Moreover, it is easy to check that $P \circ \Phi$ maps into $H_{\modu}^{-1/2, \mathrm{r}_+ +1; k}$, thus in fact, $\Phi$ maps into $\SX_{\modu}^{1/2, \mathrm{r}_+; k}$. 

To show that $\Phi$ is a contraction, on a sufficiently small ball around the origin in $\SX_{\modu}^{1/2, \mathrm{r}_+; k}$, we must find a bound for $\Phi(w) -
\Phi(\tilde w)$, for $w, \tilde w \in H_{\modu}^{1/2, \mathrm{r}_+; k}$.  We have
\begin{equation}
  \label{eq:difference for contract}
  \begin{gathered}
    N[u_- + w]) - N[u_- + \tilde w]) = (w - \tilde w) p_1(u_-, w,
    \tilde w) + (\overline{w} - \overline{\tilde w})p_2(u_1, w, \tilde
    w)
  \end{gathered}
\end{equation}
where $p_1(x, y, z)$ and $p_2(x, y, z)$ are polynomials satisfying
$$
p_1(x, y, z) = \sum_{q \in \mathbb{N}_0^6} c_q x^{q_1}
\overline{x}^{q_2} y^{q_3} \overline{y}^{q_4} z^{q_5} \overline{z}^{q_6}
$$
where $c_q$ is only non-zero when $q_1 + q_3 + q_5 = q_2 + q_4 + q_6 =
(p - 1)/2$, meaning the total number of conjugates is equal to the
total number of non-conjugates.  Similarly 
$$
p_2(x, y, z) = \sum_{q \in \mathbb{N}_0^6} c_q' x^{q_1}
\overline{x}^{q_2} y^{q_3} \overline{y}^{q_4} z^{q_5} \overline{z}^{q_6}
$$
where the $c_q' = 0$ unless $q_1 + q_3 + q_5 = q_2 + q_4 + q_6 + 2=
(p - 1)/2 + 1$, i.e. there are two fewer conjugates.  This is easy to
check e.g.\ by induction on the even number $p-1$.  Thus we can apply
the same rationale used for \eqref{eq:power estimate} to obtain
\begin{equation}
  \begin{gathered}
    \| N[u_- + w]) - N[u_- + \tilde w])\|_{H_{\modu}^{1/2, \mathrm{r}_+ + 1;k}} \qquad \\
    \qquad  \le  C \| w - \tilde w
    \|_{H_{\modu}^{1/2, \mathrm{r}_+; k}} \left( \| u_-
      \|_{H_{\modu} ^{1/2, \mathrm{r}_- ; k}} + \| w
      \|_{H_{\modu}^{1/2, \mathrm{r}_+; k}} + \| \tilde  w
      \|_{H_{\modu}^{1/2, \mathrm{r}_+; k}}
    \right)^{p-1}.\label{eq:product diff bound}
  \end{gathered}
      \end{equation}
    Using
    $$
 \| u_- \|_{H_{\modu} ^{1/2,\mathrm{r}_- ; k}} \lesssim \| f \|_{\mathcal{W}^k},
 $$
 and the mapping property of $\Rout$ from Theorem~\ref{thm:linear1}, we
 conclude that for $f$ sufficiently small in $\mathcal{W}^k$, $\Phi$
 is a contraction mapping on a small neighborhood of the origin in
 $\SX_{\modu}^{1/2, \mathrm{r}_+; k}$.
 \end{proof}
 
 \begin{proof}[Proof of Theorem~\ref{thm:main}]
 Given $f$ sufficiently small in $\mathcal{W}^k$, we define $w$ to be the unique fixed point of $\Phi$ as described above Proposition~\ref{thm:contract power}, and define $u = u_- + w$. Then as explained above, $Pu = N[u]$ so $u$ solves the nonlinear Schr\"odinger equation. Moreover, if $u' = u_- + w'$ is another solution, with $w'$ also small in $\SX_{\modu}^{1/2, \mathrm{r}_+; k}$, then we have 
 $$
 P(u_- + w') = N[u_- + w'] \Longrightarrow P w' = P u_+ + N[u_- + w']
 $$
 since $Pu_+ = - Pu_-$. This is an equation in the space $H_{\modu}^{1/2, \mathrm{r}_+ + 1; k} = \SY_{\modu}^{1/2, \mathrm{r}_+ + 1; k}$, so we can apply $P_+^{-1}$ and obtain 
 $$
   w' =  u_+ + P_+^{-1} N[u_- + w'] = \Phi(w').
 $$
 So $w'$ is also a fixed point of $\Phi$. Using the uniqueness part of the contraction mapping principle shows that $w' = w$, so we have shown the uniqueness of the small solution $u$ in $H_{\modu}^{3/2, \mathrm{r}_+; k} + H_{\modu}^{3/2, \mathrm{r}_-; k}$. 
 
To check that the incoming data of $u$ is equal to $f$, we note that this is true for $u_0 = \Poim f$; that is, $f = \SL_- u_0$. On the other hand, we can write $u_0 = u_- + u_+$, and $u_+ = P_+^{-1} [P, Q_+] u_0$ is the outgoing propagator applied to a function in $H_{\modu}^{1/2, 1+\mathrm{r}_{\max}; k}$ so, by Proposition~\ref{prop:limits}, we have $\SL_- u_+ = 0$. It follows that $\SL_- u_- = f$. In a similar way, $w$ is also the outgoing propagator applied to a function in $H_{\modu}^{1/2, 1+\mathrm{r}_{\max}; k}$, so we also have $\SL_- w = 0$. Then 
$$
\SL_- u = \SL_- u_- + \SL_- w = f + 0 = f,
$$
showing that $u$ solves the final state problem. 

To show that the outgoing data is also in $\Hdata^k$, we let $\tilde w = - P_-^{-1} N[u_- + w]$ and write $v = w + \tilde w$, so that $Pv = 0$. It follows that $\SL_+ w = \SL_+(v - \tilde w) = \SL_+ v$, as $\tilde w$ is the incoming propagator applied to something in $H_{\modu}^{1/2, 1+\mathrm{r}_{\max}; k}$, using \eqref{eq:Nbound}. So it suffices to show that the outgoing data of $v$ is in $\Hdata^k$. Similarly to the proof of Theorem~\ref{thm:linear2}, we use the identities \eqref{eq:Qu}
 and \eqref{eq:PP*}, and introduce the operator $Q'$ as in that proof to see that the outgoing data of $v$ is 
 $$
 \Poip^* Q' [P, Q_+] v.
 $$
 modulo a Schwartz function. The part (iv) of 
 Prop.~\ref{cor:NvD}, with $l=0$, shows that $\SL_+ v$ is in $\Hdata^k$. This completes the proof of Theorem~\ref{thm:main}.
\end{proof}

\subsection{Derivative nonlinearities, and the proof of Theorem~\ref{thm:main2}}

As foreshadowed in the Introduction, and observed below \eqref{eq:Nbound} in the proof of Theorem~\ref{thm:main}, we wasted one order of parabolic regularity: applying the outgoing propagator improves the regularity exponent $s$ to $3/2$ but we merely embed this back in the space with $s=1/2$. This immediately suggests that we will be able to treat derivative nonlinearities, and this is the case. The only slight change required in the proof is that we start with incoming data in the space $\ang{\zeta}^{-1} \Hdata^k$ instead of $\Hdata^k$; this is to avoid negative regularity exponents $s$, which is not permitted in Proposition~\ref{thm:s-mult}. 

\begin{proposition}\label{thm:contract deriv}
  For $k > \max(n/2+2,3)$ 
  and $f$ sufficiently small in $\langle \zeta \rangle^{-1} \mathcal{W}^k$, with
  $N[u]$ the derivative nonlinearity in \eqref{eq:N}, or more
  generally a finite sum of derivative and power nonlinearities with
  $(n,p) \neq (1,3)$, satisfying phase invariance \eqref{eq:phaseinv}, the mapping $\Phi$ defined by \eqref{eq:Phi} is
  a contraction on a sufficiently small ball around the origin in
  $\SX_{\modu}^{3/2, \mathrm{r}_+; k}$. Thus for $f$ sufficiently small in $\langle \zeta \rangle^{-1} \mathcal{W}^k$ there exists $w \in
\SX_{\modu}^{3/2, \mathrm{r}_+; k}$, unique in a small ball around
$0$, such that $u = u_- + w$ solves $Pu = N[u]$, with $u_- \in \SX_{\modu} ^{3/2,
          \mathrm{r}_- ; k}$ defined as before.  Moreover, $u$ solves the final state problem in the sense that $\mathcal{L}_- u = f$. 
 \end{proposition}

      \begin{proof}
        The proof is nearly identical to that of Proposition
        \ref{thm:contract power}.  For
        $f \in \langle \zeta \rangle^{-1} \mathcal{W}^k$, by Proposition
        \ref{thm:Poisson mapping} we construct
        $\Poim f = u_0 = u_- + u_+$ and writing $u = u_- + w$ with
        $w \in H_{\modu}^{3/2, \mathrm{r}; k}$.  By Proposition
        \ref{thm:s-mult}, taking into account the derivative in the
        nonlinearity, we obtain
        $N[u] \in H_\modu^{1/2,p(-1/2 -\epsilon) + (p-1))(n + 1)/2 ;
          k}$, i.e.\ the analogue of \eqref{eq:power estimate}.
        Defining $\Phi$ exactly as in \eqref{eq:Phi} we note that
        $\Rout$ restores the order of (parabolic) differential
        regularity to $3/2$, and thus arguing exactly as in the
        previous proof we obtain that for
        $f$ sufficiently small in
        $\langle \zeta \rangle^{-1} \mathcal{W}^k$, $\Phi$ is a contraction
        mapping on $H_{\modu}^{3/2, \mathrm{r}; k}$.  The rest follows
        as above.
      \end{proof}
      
\begin{proof}[Proof of Theorem~\ref{thm:main2}]       
Again writing instead $\Poip f = u_0 = u_- + u_+$ we obtain a solution $u =
w + u_+$ to the final state problem with $w \in  H_{\modu}^{3/2,
  \mathrm{r}_+; k}$ the fixed point of the map of the contraction map $\Phi$ from Proposition~\ref{thm:contract deriv}. 
The proof is essentially the same as the proof of Theorem~\ref{thm:main} apart from some changes to the exponents. We find that $u_\pm$ lie in $\Hpar {3/2}{\mathrm{r}_\pm}$ instead of $\Hpar {1/2}{\mathrm{r}_\pm}$ in the previous proof, due to the extra decay factor $\ang{\zeta}^{-1}$ assumed for the prescribed incoming data $f$. Since the nonlinearity involves first order derivatives, this reduces the parabolic decay index by $1$, leading to $N[u_- + w]$ lying in the space $H_{\modu}^{1/2, 1/2 + \epsilon; k}$. Then applying the outgoing propagator restores the parabolic decay index to $3/2$. 

Showing that the incoming data is $f$ is exactly as before. The treatment of the outgoing data is almost the same as before, with the main difference that we need to apply part (iv) of Proposition~\ref{cor:NvD} with $l=1$ rather than $l=0$. 
\end{proof}

\subsection{Inhomogeneous equations}
We can also solve inhomogeneous nonlinear equations
\begin{equation}
  \label{eq:inhomogeneous}
  Pu - N[u] = g
\end{equation}
provided $g$ is sufficiently small in an appropriate space, using the
same arguments above, with prescribed, sufficiently small incoming
data $f$.  In particular, with $f = 0$ we have $u$ in a purely
outgoing (or incoming) solutions:
\begin{proposition}
  Let $N[u]$ be a linear combination of power nonlinearities in
  \eqref{eq:N} and assume $k > \max(n/2+2,3)$. For
  $$
  g \in H^{1/2, \mathrm{r}_{\max}; k}_{\modu}
  $$
  and $f \in \mathcal{W}^k$, both sufficiently small in their respective spaces, there
  exists $w$ in $H_{\modu}^{1/2, \mathrm{r}_+; k}$, such that
  $u = u_- + w$ solves $Pu - N[u] = g$, where $u_-$ is the incoming
  part of the linear solution $u_0$ of $f$ described at the beginning
  of this section.
In particular, for $g \in H^{1/2, \mathrm{r}_{\max};
  k}_{\modu}$ sufficiently small and $f \equiv 0$, there exists $u \in H_{\modu}^{1/2,
  \mathrm{r}_+; k}$ sufficiently small such that $Pu - N[u] = g$,
i.e. $u$ is a purely outgoing solution.

When $N[u]$ is a derivative nonlinearity, the same holds with $f
\in \langle \zeta \rangle^{-1}\mathcal{W}^{k}$, $g \in H_{\modu}^{3/2, \mathrm{r}_{\max}; k}$.
\end{proposition}
\begin{proof}
Analogously to the previous proofs, set $\Phi(w) = u_+ + \Rout(N[u_- +
w] + g)$.  The proof then proceeds exactly as those of
Propositions \ref{thm:contract power} and \ref{thm:contract deriv}.
\end{proof}

%\jg{nonlinear scattering matrix commented}
% Finally, we conclude that the nonlinear scattering operator is equal
% to the linear scattering operator plus a regularizing operator.  Indeed, the
% solutions $u$ produced in Proposition \ref{thm:contract power} and
% \ref{thm:deriv NLS} all satisfy $u = u_- + w$ with
% $w - u_+ = \Rout(N[u_- + w])$.  From this expression we can understand
% the asymptotics of $u$.  Since
% $N[u_- + w] \in H_\modu^{1/2,p(-1/2 -\epsilon) + (p-1))(n + 1)/2 ;
%   k}$, by Proposition \ref{prop:limits} we have that $\Rout(N[u_- +
% w]) \in $

\appendix

%%%%%%%%%%%%%%%%%%%%%%%%%%%%%%%%%%%%%%%%%%%%%%%%%%%%%%%%%%%%%%%%%%%
%%%%%%%%%%%%%%%%%%%%%%%%%%%%%%%%%%%%%%%%%%%%%%%%%%%%%%%%%%%%%%%%%%%
%%%%%%%%%%%%%%%%%%%%%%%%%%%%%%%%%%%%%%%%%%%%%%%%%%%%%%%%%%%%%%%%%%%
%%%%%%%%%%%%%%%%%%%%%%%%%%%%%%%%%%%%%%%%%%%%%%%%%%%%%%%%%%%%%%%%%%%
%%%%%%%%%%%%%%%%%%%%%%%%%%%%%%%%%%%%%%%%%%%%%%%%%%%%%%%%%%%%%%%%%%%

\section{Propagation of small module regularity}
\label{appendix}

We now prove ``module regularity'' for the Schr\"odinger operator
$ P = D_t + \Delta_{g(t)} + V(z,t)$ for the module $\modu = \modu_1$, i.e.\
the $\Psi^{0,0}_{\mathrm{par}}$-module generated by the operators in
$\mathbf{G} = \mathbf{G}_1$ (see ~\eqref{eq:G.def}).  Much of what
follows is a minor modification material of Section 5 of \cite{TDSL},
the key difference between that setting and the current one being that
the modules $\modu^{(k)}$ are not closed under commutation, instead
satisfying the properties \eqref{eq:module properties}.  The
assumptions on $g(t)$ and $V(z,t)$ are the same as those from the
introduction, namely $g(t)$ is a smooth family of metrics on $\RR^n$
equal to the standard metric for $|t| \geq T$ and also equal to the
standard metric for all $|z| \leq R$ and $V(z, t)$ is assumed to be
smooth and compactly supported.

First we have the elliptic estimates, similar to \cite[Prop.~5.7]{TDSL}
\begin{prop}
	\label{prop:module.elliptic.estimate.old}
	Suppose $Q, Q' \in\Psip{0}{0}$ and $\WF'(Q) \subset \Ell_{2,0}(P)
        \cap \Ell_{0,0}(Q')$.  Then for any $u \in \mathcal{S}'$ with $Q'Pu
        \in H_{\modu}^{s-2,\sw;k}$, we have $Qu \in
        H_{\modu}^{s,\sw;k}$, and for any $M,N \in \mathbb{R}$,
        there is $C$ such that
	\begin{equation}\label{eq:mod ellip}
          \|Qu\|_{H_{\modu}^{s,\sw;k}}\leq C(\|Q'P u\|_{H_{\modu}^{s-2,\sw;k}}+\|u\|_{H_{\mathrm{par}}^{M,N}})
	\end{equation}
        for all such $u$.
      \end{prop}
\begin{proof}
	Following \cite{TDSL}, the standard elliptic parametrix
        construction works in the parabolic calculus, and thus allows
        us to construct $G \in \Psi^{-2,0}_{\mathrm{par}}$ such that
        $\WF'(GP - I) \cap \WF'(Q) = \varnothing$.  Using the
        boundedness of parabolic PsiDOs on module regularity spaces
       \ref{thm:bounded on mod reg spaces} we have the
        standard argument:
        $$
\|Qu\|_{H_{\modu}^{s,\sw;k}} \le \|Q GP  u\|_{H_{\modu}^{s,\sw;k}} + \|Q(GP - I) u\|_{H_{\modu}^{s,\sw;k}} \lesssim \|QGP
u\|_{H_{\modu}^{s - 2,\sw;k}} + \|
u\|_{_{\mathrm{par}}^{N,M}}.
        $$
Using $\|Q GP  u\|_{H_{\modu}^{s,\sw;k}} \le C \|P
u\|_{H_{\modu}^{s - 2,\sw;k}}$ (by boundedness of $QG$ on module
regularity spaces) gives the
standard elliptic estimate (i.e.\ without the $Q'$) and noting that
$Q'P$ is elliptic on $\Ell_{2,0}(P) \cap \Ell_{0,0}(Q')$ gives the estimate \eqref{eq:mod ellip}.
        \end{proof}

We now prove the module-regularity propagation of
singularities estimates away from radial points.  We argue by
reduction to the standard (non-module) parabolic propagation using the
following lemma, which states, essentially, that on the characteristic
set of $P$ away from the radial set, module regularity is ``the
same'' as standard parabolic regularity. 
\begin{lemma}\label{thm:mod is par}
  Let $q \in \SymbSpa$ have $q \in \Sigma_{P} \setminus
  \mathcal{R}$.  Then at least one of the $D_{z_i}$ or $t D_{z_i} -
  z_i/2$ is elliptic at $q$ in $\Psi^{1,1}_{par}$.

  In particular, at each such $q$, and for all $k \in \mathbb{N}_0$,
  $s \in \mathbb{R}$ and $\mathrm{r}$ a variable order, there are $Q, Q' \in \Psi^{0,0}_{par}$ with
  $q \in \Ell_{0,0}(Q) \cap \Ell_{0,0}(Q')$ such that $\WF'(Q) \subset
  \Ell_{0,0}(Q')$ and for any $M,N$,
  \begin{equation*}
    \begin{gathered}
      \| Q u \|_{H^{s,r; k}_{\modu}} \lesssim \| Q' u \|_{H^{s + k,r
          + k}} + \| u \|_{H_{\mathrm{par}}^{M,N}} \mbox{ and } \\
      \| Q u \|_{H^{s + k,r
          + k}} \lesssim \| Q' u \|_{H^{s,r; k}_{\modu}} + \| u
      \|_{H_{\mathrm{par}}^{M,N}}
    \end{gathered}
\end{equation*}
\end{lemma}
Note that this is not true if one is not on the characteristic set of
$P$, for example over the interior on the set $\{ \zeta_i = \sigma_{1,0}
(D_{z_i}) = 0 : i = 1, \dots, n \}$.
\begin{proof}
  First, assume that all the $D_{z_i}$ and $t D_{z_i} - z_i/2$ are
  characteristic at $q$.  Then, as $q \in \Sigma_{P}$, if $q$ lies
  over the interior of $M$, then it is in the closure of the set
  $$
  \{ \zeta = 0, t \zeta - z_i /2 = 0, \tau = |\zeta|^2 \} = \{ \zeta
  = 0, z = 0, \tau = 0 \}.
  $$
This set does not intersect fiber infinity, such $q$ in fact lies over
the boundary.

Over the boundary, the $D_{z_i}$ are characteristic in
$\Psi^{1,1}_{\mathrm{par}}$, so $q$ lies in the closure of $\{t \zeta - z_i /2 = 0, \tau = |\zeta|^2 \},$ which is exactly the radial set.  This
  proves the first part.

On to the inequalities, the first follows directly from the definition.
  For the second, assume one of the $D_{z_i}$ is elliptic in
  $\Psi^{1,1}$ at $q$. Then $q$ lies over the interior of $M$, since
  $D_{z_i}$ is characteristic in $\Psi^{1,1}$ over the spacetime
  boundary.  But then, taking $Q \in \Psi^{0,0}_{par}$ elliptic at $q$
  with $\WF'(Q)$ in the interior of $M$ and with $\WF'(Q)
  \subset \Ell_{1,1}(D_{z_i})$, since $D_{z_i}^k u \in
  H^{s,r}$, by elliptic regularity for the parabolic calculus we have
$$
\| Q u \|_{H^{s + k,\mathrm{r} + k}} \lesssim \| Q' D_{z_i}^k u\|_{H^{s +
      k, \mathrm{r} + k}}, + \| u \|_{M,N} \lesssim \| Q' u\|_{H^{s,
      \mathrm{r}; k}_{\modu}}, + \| u \|_{M,N} 
$$
Here the spatial decay order is irrelevant since $Q$ is supported away
from spacetime infinity; on the right hand side it can be taken to be
arbitrary.  The argument is similar if $t D_{z_i} - z_i/2$ is
    elliptic at $q$. 
  \end{proof}

  Using the previous lemma we can now prove the following by arguing
  along the lines of \cite[Prop.\ 5.8]{TDSL}
\begin{prop}
	\label{prop:module.propagation}
	Let $Q,Q',G\in\Psip{0}{0}$ with $G$ elliptic on $\WF'(Q)$, and
        such that
        $\WF'(Q') \cup \WF'(G) \cap \mathcal{R} = \varnothing$.  Let
        $\sw$ be a variable spacetime order that is non-increasing in
        the direction of the bicharacteristic flow of $P$.
        Suppose also that for every
        $\alpha\in\WF'(Q)\cap\Sigma_P $ there exists $\alpha'$ such
        that $Q'$ is elliptic at $\alpha'$ and there is a forward
        bicharacteristic curve $\gamma$ of $P $ from $\alpha'$ to
        $\alpha$ such that $G$ is elliptic on $\gamma$.
	
	Then if $GP u\in H_{\modu}^{s-1,\sw+1;k}$ and $Q'u\in
        H_{\modu}^{s,\sw;k}$, we have $Q u\in
        H_{\modu}^{s,\sw;k}$ with the estimate	
\[\|Qu\|_{H_{\modu}^{s,\sw;k}}\leq C(\|Q' u \|_{H_{\modu}^{s,\sw;k}}+\|GP u\|_{H_{\modu}^{s,\sw;k}}+\|u\|_{H_{\mathrm{par}}^{M,N}})\]
	for any $M,N\in\RR$.
\end{prop}
\begin{proof}
  By Lemma  \ref{thm:mod is par} this follows from the propagation of
  standard (i.e.\ not module)
  parabolic regularity, which is Proposition 8.1 of \cite{TDSL}.
  %Let $B \in \Psi^{0,0}_{par}$ have $\WF'(B) \cap \mathcal{R} =
        % \varnothing$.  For any point in $q \in \WF'(B) \cap \Sigma_P$, there
        % is a $B_q \in \Psi^{0,0}_{par}$ such that $q \in
        % \Ell_{0,0}(B_q)$ and [BLORCH]we can use a
        % microlocal partition of unity to write $B=B_1+\ldots+B_m$
        % where each $\WF'(B_j)$ is contained in $\Ell(A_j)$ for some
        % $A_j\in \SN$. As such we have that the norms
        % $\|B_jv\|_{H^{s,r;\kappa,k}_{\pm}}$ and
        % $\|B_jv\|_{H_{\mathrm{par}}^{s+\kappa+k,r+\kappa+k}}$ are
        % equivalent. Summing in $j$ we obtain eqivalence between
        % $\|Bv\|_{H^{s,r;\kappa,k}_{\pm}}$ and
        % $\|Bv\|_{H_{\mathrm{par}}^{s+\kappa+k,r+\kappa+k}}$. We can
        % then directly apply Proposition \ref{prop:sing} to complete
        % the proof, noting that the operators $Q,Q',G$ in these two
        % propositions enjoy this same microsupport condition.
\end{proof}

We now have the two main module regularity estimates near
the radial set.  This requires a closure property for commutators of generators
with the operator $P$ itself. 
\begin{lemma}
\label{lem:P.crit}
For any generator $A_j\in\mathbf{G}=\{A_0,\ldots,A_N\}$, we have
$$i\ang{Z}[A_j,P]=\sum_{k=0}^N C_{jk}A_k+C_j'P $$
where $C_{jk}\in \Psip{1}{0}$ and $C_j'\in \Psip{0}{1}$ are such that
$$\sigma_{\mathrm{base},1,0}(C_{jk})|_{\SR}=0 \quad \textrm{for }j\neq k$$
and
$$\mathrm{Re}(\sigma_{\mathrm{base},1,0}(C_{jj}))|_{\SR}=0 $$
and $C_{jk}=0$ for each $k$ such that $A_k$ has higher parabolic regularity order than $A_j$.
\end{lemma}

% The norm on $H^{s, \mathrm{r};k}_{\modu}$ is by definition
% \begin{equation}
%   \label{eq:3}
% \| u \|_{H^{s, \mathrm{r};k}_{\modu}} :=
% \sum_{\mathbf{A}^\alpha \in \mathbf{G}^{|\alpha|},\  A^\alpha \in \Psi^{k,k}_{\mathrm{par}}}
% \| A^\alpha u \|_{H^{s, \mathrm{r};k}_{\modu}}.
% \end{equation}

\begin{prop}
	\label{prop:mod.belowschrod}
	Assume $r \in \mathbb{R}$, $r<-1/2$. Assume that there exists a neighbourhood $U$
        of $\SR_{\pm}$ and $Q', G\in\Psip{0}{0}$ such that for every
        $\alpha\in\Sigma_P \cap U\setminus \SR_{\pm}$ the
        bicharacteristic $\gamma$ through $\alpha$ enters $\Ell_{0,0}(Q')$
        whilst remaining in $\Ell_{0,0}(G)$. Then there exists
        $Q\in\Psip{0}{0}$ elliptic on $\SR_{\pm}$ such that if
        $u \in H_\mathrm{par}^{M,N}$, $Q'u\in H_{\SD}^{s,r;k}$, and
        $GP u\in H_{\SD}^{s-1,r+1;k}$, then
        $Qu\in H_{\SD}^{s,r;k}$. Moreover, there exists $C>0$ such
        that
	\begin{equation}\label{eq:mod.radial.schrod.below}
		\|Qu\|_{H_{\SD}^{s,r;k}}\leq C(\|Q'u\|_{H_{\SD}^{s,r;k}} +\|GP u\|_{H_{\SD}^{s-1,r+1;k}}+\|u\|_{H_\mathrm{par}^{M,N}})
	\end{equation}
\end{prop}

\begin{prop}
	\label{prop:mod.aboveschrod} 
	Assume $r,r', s, s' \in \mathbb{R}$, $r>r'>-\frac{1}{2}$ and $s>s'$. Assume that $G\in\Psip{0}{0}$ is elliptic at $\SR_{\pm}$. Then there exists $Q\in\Psip{0}{0}$ elliptic at $\SR_{\pm}$ such that, if $u \in H_\mathrm{par}^{M,N}$, $Gu\in H_{\SD}^{s',r';k}$ and $GP u\in H_{\modu}^{s-1,r+1;k}$, then $Qu\in H_{\SD}^{s,r;k}$ and there exists $C>0$ such that
	\begin{equation}\label{eq:mod.radial.schrod.above}
		\|Qu\|_{H_{\SD}^{s,r;k}}\leq C(\|GP u\|_{H_{\SD}^{s-1,r+1;k}}+\|Gu\|_{H_{\SD}^{s',r';k}}+\|u\|_{H_{\mathrm{par}}^{M,N}}).
	\end{equation}
\end{prop}
\begin{proof}
	The proof of these estimates is almost identical to the proof of \cite[Proposition~5.9]{TDSL} although there is only a single module involved.
	
	We take a product of generators $A_\alpha=\prod_{j=0}^N \mathbf{A}_j^{\alpha_j}$ that lies in the reduced power  $\SD^{(k)}$ and $A$ as in \cite[Proposition~5.9]{TDSL} and we consider the commutator
	$$i[A_\alpha^*AA_\alpha,P].$$
	
	Using \eqref{eq:commgen}, \eqref{eq:module properties} and Lemma  \ref{lem:P.crit} we can commute $P$ past the generators in the expression $i[A_\alpha^*AA_\alpha,P]$ to obtain the same commutator identity as in \cite[(5.31)]{TDSL} with the terms $E_\alpha,E_\alpha'\in\SM^{(k')}$ for $k'< j$ and $A_\beta\in \SM^{(k)}$.
	We can then follow the positive commutator argument where the matrix of operators in \cite[(5.32)]{TDSL} is replaced with 
	$$S_k=\{\alpha\in \NN^{N+1}:|\alpha|\leq k\textrm{ and } A^\alpha\in \Psip{k}{k}\}$$
	to account for the different norm $\|\cdot\|_{H_{\SD}}$ \eqref{eq:mod.reg.sob.norm} in this paper.
	This yields the a priori estimate \eqref{eq:radial.schrod.below} as in \cite[(5.50)]{TDSL}.
	
	To complete the proof from the a priori estimate we proceed by induction in $k$ as in \cite{TDSL}, with a slight modification. It is no longer the case that $H_{\modu}^{s+1,r+1;k}\subset H_{\modu}^{s,r;k+1}$, however it is true that $H_{\modu_{c}}^{s+2,r+2;k}\subset H_{\modu}^{s,r;k+2}$,and so the induction is best done with step size $2$.
	
	 For $k=0$, the result is just the standard radial set estimate in the parabolic calculus \cite[Proposition~5.3]{TDSL}.
	
	Assuming now $k\geq 2$ and the result has been proved for module regularity of order $k-2$, we take $u$ satisfying the conditions of Proposition \ref{prop:belowschrod} and take $$u'=S_\eta u $$ where 
	$$S_\eta=q_L\left(\left(\frac{\rhobase}{\rhobase+1}\frac{\rhofib}{\rhofib+1}\right)^2\right)\in \Psip{-2}{-2}.$$
	The inductive hypothesis implies that $Qu\in H^{s,r;k-2}$ and so upon redefining $Q$ by shrinking its microsupport, we get that $Qu'\in H^{s,r;k}$ and the a priori estimate \eqref{eq:mod.radial.schrod.below} is valid for $u'$. We conclude that \eqref{eq:mod.radial.schrod.below} is valid for $u$ as well by taking $\eta\to 0$ and using a weak limit argument as in \cite{TDSL}. This completes the proof of Proposition \ref{prop:belowschrod} for even $k$. For odd $k$ the same induction argument works provided that we can establish the result for $k=1$. To do this, we consider the submodule $\tilde{\SD}$ of $\modu$ generated by $\tilde{\mathbf{G}}\cap\Psip{1}{1}$. This submodule satisfies  $\tilde{\SD}^{(k)}=\tilde{\SD}^{k}$ for all $k$ and  satisfies \eqref{eq:commgen}, \eqref{eq:module properties} and Lemma \ref{lem:P.crit} as well. 
	Hence \cite[Proposition~5.9]{TDSL} applies to the module regularity spaces $H_{\tilde{\SD}}^{s,r;k}$ directly, and the observation that $H_{\tilde{\SD}}^{s,r;1}=H_{\modu}^{s,r;1}$ completes the proof. 
	
	The proof of Proposition \ref{prop:aboveschrod} is obtained by making similar modifications to \cite[Proposition~5.10]{TDSL}.
\end{proof}

We can compile these various
estimates into a global estimate in module regularity spaces by using
an appropriate microlocal partition of unity.
\begin{proposition}\label{prop:big mod reg invert}
  Let $\mathrm{r}_+ \in C^{\infty}(\SymbSpa)$ be
  monotone nonincreasing along the
  Hamilton flow and constant in neighborhoods of $\mathcal{R}_\pm$
  with
  \begin{equation}
\mp ( \mathrm{r}_+ + 1/2) > 0 \mbox{ near } \mathcal{R}_{\pm} . \label{eq:rplus}
  \end{equation}
Let $s \in \mathbb{R}$ and $k \in  \mathbb{N}_0$.  Then there is $C >
0$ such that for
  all $u \in \mathcal{S}'$ with $Pu \in H^{s-1, \mathrm{r}_+;
    k}_{\modu}$ and $u \in H^{s,\mathrm{r}_+; k}_{\modu}$, that is, $u \in \mathcal{X}^{s, \mathrm{r}_+;k}_{\modu}$, 
  \begin{equation}
    \label{eq:mod reg global}
    		\|u\|_{H_{\SD}^{s,\mathrm{r}_+;k}}\leq C(\|P
                u\|_{H_{\SD}^{s-1,\mathrm{r}_+ + 1;k}}+\|u\|_{H_{\mathrm{par}}^{M,N}}).
  \end{equation}

Moreover, the mapping
\begin{equation}
  \label{eq:invertible mapping}
  P \colon \mathcal{X}^{s, \mathrm{r}_+;k}_{\modu} \lra H^{s-1,
    \mathrm{r}_+ + 1; k}_{\modu}
\end{equation}
is invertible.

The same holds for $\mathrm{r}_- \in  C^{\infty}(
  \SymbSpa)$ monotone nondecreasing along the
  Hamilton flow and constant in neighborhoods of $\mathcal{R}_\pm$
  with $\pm (\mathrm{r}_- + 1/2) > 0$ near $\mathcal{R}_\pm$.
\end{proposition}
\begin{proof}
  This follows via the now standard argument of writing
  $u = Q_- u + Q_1 u + Q_+ u$ with
  $Q_\pm, Q_1 \in \Psi^{0,0}_{\mathrm{par}}$ and $Q_\pm$
  microsupported sufficiently close to $\mathcal{R}_\pm$ that the
  radial points estimates apply, so that one obtains directly from the
  previous propositions initially an
  estimate 
  $$
    		\|u\|_{H_{\SD}^{s,\mathrm{r}_+;k}}\leq C(\|P u\|_{H_{\SD}^{s-1,\mathrm{r}_+;k}}
		+  \|G u\|_{H_{\SD}^{s',-1/2 + \epsilon';k}}+ \|u\|_{H_{\mathrm{par}}^{M,N}}).
                $$
for some $G \in \Psi^{0,0}_{\mathrm{par}}$ supported near
$\mathcal{R}_-$ (arising from the $Gu$ term in \eqref{eq:mod.radial.schrod.above}) , $s' < s$,   and $-1/2 < -1/2 + \epsilon'  < -1/2 + \epsilon =
\mathrm{r}_+$ near $\mathcal{R}_-$.  Then, since $s'$ is between $M$ and $s$, and $-1/2 + \epsilon'$ is between $N$ and $-1/2 + \epsilon$, we can bound the $Gu$ norm by 
$$
\delta \|u\|_{H_{\SD}^{s,\mathrm{r}_+;k}} + \frac{C}{\delta} \|u\|_{H_{\mathrm{par}}^{M,N}}
$$
for small $\delta$ and suitable $M,N$, and thereby absorb the first term on the LHS, arriving at the Fredholm estimate \eqref{eq:mod reg global}. 
\end{proof}

\begin{remark} It is important to realize that to obtain the estimate \eqref{eq:mod reg global}, one must assume \emph{a priori} that $u \in H^{s,\mathrm{r}_+; k}_{\modu}$ --- that is, it does not follow from the assumption that $Pu \in H^{s-1, \mathrm{r}_+;
    k}_{\modu}$. This is in contrast to the situation for elliptic estimates, where, for example, in \eqref{eq:mod ellip}, the fact that $Q'Pu
        \in H_{\modu}^{s-2,\sw;k}$ implies that $Qu \in H_{\modu}^{s,\sw;k}$ without any such \emph{a priori} assumption. 
\end{remark}
  
%It is now straightforward to prove Theorem \ref{thm:main module propagator}.
%\begin{proof}[Proof of Theorem \ref{thm:main module propagator}]
%  Let $\mathrm{r}_+$ from Proposition \ref{prop:big mod reg invert}
%  satisfy $\mathrm{r}_+ = -1/2 + r_0$.  By the inclusion $H_{\modu}^{s
%    + 1, 1/2 +  r_0} \subset H_{\modu}^{s
%    + 1, \mathrm{r}_+ + 1}$, the invertibility in \eqref{eq:invertible
%    mapping} gives \eqref{eq:isomorphism} directly.
%\end{proof}

\bibliographystyle{plain}
\bibliography{nls}
\end{document}